\documentclass[reqno,a4paper,11pt]{amsart}
\usepackage[all,poly]{xy}
\usepackage{amsfonts}
\usepackage[mathcal]{eucal}
\usepackage{amssymb}
\usepackage{amsmath}
\usepackage{mathrsfs}
\usepackage{amsthm}
\usepackage{color}
\usepackage[pagebackref,colorlinks]{hyperref}
\usepackage{enumerate}

\theoremstyle{plain}
\newtheorem{lemma}{Lemma}[section] 
\newtheorem{theorem}[lemma]{Theorem}
\newtheorem{corollary}[lemma]{Corollary}
\newtheorem{proposition}[lemma]{Proposition}
\newtheorem{conjecture}[lemma]{Conjecture}

\theoremstyle{definition}
\newtheorem{remark}[lemma]{Remark}
\newtheorem{example}[lemma]{Example}
\newtheorem{definition}[lemma]{Definition}

\newcommand{\M}{\operatorname{\mathbb M}}
\newcommand{\gr}{\operatorname{gr}}
\newcommand\Gr[1][]{{\operatorname{{Gr}^{#1}-}}}
\newcommand{\Pgr}[1][]{{\operatorname{{Pgr}^{#1}-}}}
\newcommand{\ol}{\overline}

\newcommand{\Z}{\mathbb Z}
\newcommand{\C}{\mathbb C}

\newcommand{\Hom}{\operatorname{Hom}}

\newcommand\sima{\sim^{\hskip-.22cm a}}
\newcommand\sims{\sim^{\hskip-.22cm *}}

\newcommand{\so}{\mathbf{s}}
\newcommand{\ra}{\mathbf{r}}

\title[K-theory classification of graded ultramatricial algebras with involution]{K-theory Classification of Graded Ultramatricial Algebras with Involution}

\author{Roozbeh Hazrat}
\address{Centre for Research in Mathematics,
Western Sydney University,
Australia}\email{r.hazrat@westernsydney.edu.au}

\author{Lia Va\v s}
\address{Department of Mathematics, Physics and Statistics, University of the Sciences, Philadelphia, PA 19104, USA}
\email{l.vas@usciences.edu}


\thanks{We would like to thank the referee for a careful reading, insightful comments and useful suggestions. The first author would like to acknowledge Australian Research Council grant DP160101481. A part of this work was done at the University of Bielefeld, where the first author was a Humboldt Fellow.}

\begin{document}
 
\begin{abstract} We consider a generalization $K_0^{\gr}(R)$ of the standard Grothendieck group $K_0(R)$ of a graded ring $R$ with involution. If $\Gamma$ is an abelian group, we show that $K_0^{\gr}$ completely classifies graded ultramatricial $*$-algebras over a $\Gamma$-graded $*$-field $A$ such that 
(1) each nontrivial graded component of $A$ has a unitary element in which case we say that $A$ has enough unitaries, and (2) the zero-component $A_0$ is 2-proper (for any $a,b\in A_0,$ $aa^*+bb^*=0$ implies $a=b=0$) and $*$-pythagorean (for any $a,b\in A_0,$ $aa^*+bb^*=cc^*$ for some $c\in A_0$). If the involutive structure is not considered, our result implies that $K_0^{\gr}$ completely classifies graded ultramatricial algebras over any graded field $A.$ If the grading is trivial and the involutive structure is not considered, we obtain some well-known results as corollaries. 

If $R$ and $S$ are graded matricial $*$-algebras over a $\Gamma$-graded $*$-field $A$ with enough unitaries and $f: K_0^{\gr}(R)\to K_0^{\gr}(S)$ is a contractive $\Z[\Gamma]$-module homomorphism, we present a specific formula for a
graded $*$-homomorphism $\phi: R\to S$ with $K_0^{\gr}(\phi) = f.$ If the grading is trivial and the involutive structure is not considered, our constructive proof implies the known results with existential proofs. If $A_0$ is 2-proper and $*$-pythagorean, we also show that two graded $*$-homomorphisms $\phi,\psi:R\to S,$ are such that $K_0^{\gr}(\phi) = K_0^{\gr}(\psi)$ if and only if there is a unitary element $u$ of degree zero in $S$ such that $\phi(r) = u\psi(r)u^*$ for any $r\in R.$  

As an application of our results, we show that the graded version of the Isomorphism Conjecture holds for a class of Leavitt path algebras: 
if $E$ and $F$ are countable, row-finite, no-exit graphs in which every infinite path ends in a sink or a cycle and $K$ is a 2-proper and $*$-pythagorean field, then the Leavitt path algebras $L_K(E)$ and $L_K(F)$ are isomorphic as graded rings if any only if they are isomorphic as graded $*$-algebras. We also present examples which illustrate that $K_0^{\gr}$ produces a finer invariant than $K_0.$
\end{abstract}

\maketitle
 
\section{Introduction}

Let $A$ be a field and let $\mathcal C$ be a class of unital $A$-algebras. One says that the $K_0$-group is a complete invariant for algebras in $\mathcal C$ and that $K_0$ completely classifies algebras in $\mathcal C$ if any algebras $R$ and $S$ from $\mathcal C$ are isomorphic as algebras if and only if there is a group isomorphism $K_0(R) \cong K_0(S)$ which respects the natural pre-order structure of the $K_0$-groups and their order-units. If $\mathcal C$ is a class of $A$-algebras which are not necessarily unital, one can define the $K_0$-groups using unitization of the algebras. The role of an order-unit is taken by the generating interval. A map which respects the pre-order structure and the generating interval is referred to as {\em contractive}. In this case, $K_0$ {\em completely classifies algebras} in $\mathcal C$ if any algebras $R$ and $S$ from $\mathcal C$ are isomorphic as algebras if and only if there is a contractive group isomorphism $K_0(R) \cong K_0(S).$ If so, then any algebras $R$ and $S$ from $\mathcal C$ are isomorphic as algebras if and only if $R$ and $S$ are isomorphic as rings. 

Assume that the field $A$ and the algebras in the class $\mathcal C$ are equipped with an involution so that the algebras from $\mathcal C$ are $*$-algebras over $A.$ In this case, the $*$-homomorphisms take over the role of homomorphisms and there is a natural action of $\Z_2:=\Z/2\Z$ on the $K_0$-groups given by $[P]\mapsto [\Hom_R(P, R)]$ for any $R\in \mathcal C.$ We say that $K_0$ completely classifies the $*$-algebras in $\mathcal C$ if any $R, S\in \mathcal C$ are $*$-isomorphic as $*$-algebras if and only if there is a contractive $\Z[\Z_2]$-module isomorphism $K_0(R) \cong K_0(S).$ If so, any $*$-algebras $R$ and $S$ from $\mathcal C$ are isomorphic as $*$-algebras if and only if $R$ and $S$ are isomorphic as $*$-rings. In particular, if the action of $\Z_2$ on $K_0$-groups is trivial and $K_0$ completely classifies $*$-algebras in $\mathcal C,$ then any $*$-algebras from $\mathcal C$ are isomorphic as rings precisely when they are isomorphic as any of the following: $*$-rings, algebras, and $*$-algebras.   

In the seminal paper \cite{Elliott}, George Elliott introduced the $K_0$-group precisely as an invariant which completely classifies approximately finite-dimensional $C^*$-algebras (AF-algebras) as $*$-algebras. Elliott's work ignited an interest in classifying a wider class of $C^*$-algebras using related $K$-theory invariants and this agenda became known as the Elliott's Classification Program. The $K_0$-group has been generalized for non-unital algebras in \cite{Goodearl_Handelman} where it was shown that $K_0$ completely classifies ultramatricial algebras (i.e., countable direct limits of matricial algebras) over any field (\cite[Theorem 12.5]{Goodearl_Handelman}). 
 
Recall that an involution $*$ on $A$ is said to be \emph{$n$-proper} if, for all $x_1, \dots, x_n \in A$, $\sum_{i=1}^n x_ix_i^* = 0$ implies $x_i=0$ for each $i=1,\ldots,n$ and that it is {\em positive definite} if it is $n$-proper for every $n.$ By \cite[Theorem 3.4]{Baranov}, $K_0$ is a complete invariant for ultramatricial $*$-algebras over an algebraically closed field with the identity involution which is positive definite. In \cite{Ara_matrix_rings}, Ara considered another relevant property of the involution: a $*$-ring $A$ is said to be {\em $*$-pythagorean} if for every $x,y$ in $A$ there is $z$ such that $xx^*+yy^*=zz^*.$ It is direct to check that a $*$-pythagorean ring is 2-proper if and only if it is positive definite. By \cite[Proposition 3.3]{Ara_matrix_rings}, if $A$ is a 2-proper, $*$-pythagorean field, then $K_0$ completely classifies the ultramatricial $*$-algebras over $A.$ 

In \cite{Roozbeh_Annalen}, the $K_0$-group is generalized for {\em graded} algebras. It has been shown that the graded $K_0$-group preserves information about algebras which is otherwise lost in the non-graded case (see \cite[Introduction]{Roozbeh_Annalen}) and the examples in \cite{Roozbeh_Annalen} illustrate that the new invariant $K^{\gr}_0$ is better behaved than the usual, non-graded $K_0$-group. In \cite{Roozbeh_Annalen}, the involutive structure of the algebras was not taken into account. 

There seem to be two trends: to consider the involutive structure of algebras (e.g. \cite{Ara_matrix_rings} and \cite{Baranov}) and to consider the graded structure of algebras (e.g. \cite{Roozbeh_Annalen}). However, no attempt has been made to combine these two structures. Our work fills this gap by studying graded rings with involution such that the graded and the involutive structures agree. Both the graded and the involutive structures are relevant in many examples including ultramatricial algebras, group rings, corner skew Laurent polynomial rings and Leavitt path algebras. 

If $A$ is an involutive ring graded by an abelian group $\Gamma,$ the involutive and the graded structures agree if $A_\gamma^*=A_{-\gamma},$ for any $\gamma\in \Gamma.$ In this case, we say that $A$ is a graded $*$-ring. The graded Grothendieck group $K^{\gr}_0(A)$ of a graded $*$-ring $A$ has both a natural $\Z_2$-action as well as a natural $\Gamma$-action and we study these actions in section \ref{section_graded_star}.  In particular, $\Z_2$ acts trivially on the $K_0^{\gr}$-groups of graded ultramatricial $*$-algebras over a graded $*$-field.

We say that $K_0^{\gr}$ completely classifies $*$-algebras in a class $\mathcal C$ of graded $*$-algebras if any graded $*$-algebras $R$ and $S$ from $\mathcal C$ are isomorphic as $*$-algebras if and only if there is a contractive $\Z[\Gamma]$-$\Z[\Z_2]$-bimodule isomorphism $K^{\gr}_0(R) \cong K^{\gr}_0(S).$ 

The main results of sections \ref{section_fullness}, \ref{section_faithfulness}, and \ref{section_ultramatricial} are contained in Theorems  \ref{fullness},  \ref{faithfulness} and \ref{ultramatricial} which we summarize as follows. 
\begin{theorem} \label{main}
Let $\Gamma$ be an abelian group and let $A$ be a $\Gamma$-graded $*$-field. 

If $A$ has enough unitaries (every nontrivial graded component of A has a unitary element), $R$ and $S$ are graded matricial $*$-algebras over $A$ and $f:K^{\gr}_0(R)\rightarrow K^{\gr}_0(S)$ is a contractive $\Z[\Gamma]$-module homomorphism, then there is a graded $A$-algebra $*$-homomorphism $\phi:R\rightarrow S$ such that $K^{\gr}_0(\phi)=f$. Furthermore, if $f([R])=[S]$, then $\phi$ is a unital homomorphism. 

If the zero-component $A_0$ of $A$ is 2-proper and $*$-pythagorean and $R$ and $S$ are graded matricial $*$-algebras over $A,$ then  
every two graded $A$-algebra $*$-homomorphisms $\phi, \psi: R\to S$ are such that $K^{\gr}_0(\phi)=K^{\gr}_0(\psi)$ if and only if 
there exist a unitary element of degree zero $u\in S$ such that $\phi(r)=u\psi(r)u^*,$ for all $r\in R.$  

If $A$ has enough unitaries and $A_0$ is 2-proper and $*$-pythagorean, then $K_0^{\gr}$ completely classifies the class of graded ultramatricial $*$-algebras over $A.$ In particular, if $R$ and $S$ are graded ultramatricial $*$-algebras and $f: K^{\gr}_0(R)\to K^{\gr}_0(S)$ is a contractive  $\Z[\Gamma]$-module isomorphism, then there is a graded $A$-algebra $*$-isomorphism $\phi: R\to S$ such that $K^{\gr}_0(\phi)=f.$
\end{theorem}

As a corollary, if $A$ is a graded $*$-field which satisfies the assumptions of the last part of Theorem \ref{main}, then any graded ultramatricial $*$-algebras over $A$ are isomorphic as graded rings precisely when they are isomorphic as any of the following: graded $*$-rings, graded algebras, or graded $*$-algebras.

Theorem \ref{main} may seem to be a specialization of the well-known classification results on ultramatricial algebras over a field because 
it may seem that we are adding structure to the field. However, our result is actually a {\em generalization}: if the grade group is trivial, we obtain \cite[Proposition 3.3]{Ara_matrix_rings}. If we ignore the involutive structure, we obtain \cite[Theorem 2.13]{Roozbeh_Annalen} and \cite[Theorem 5.2.4]{Roozbeh_graded_ring_notes}. When the grade group is trivial and the involutive structure disregarded, we obtain \cite[Theorem 15.26]{Goodearl_book}. For the field of complex numbers and the complex conjugate involution with a trivial grade group, we obtain the classification theorem for AF $C^*$-algebras, \cite[Theorem 4.3]{Elliott} as a corollary of Theorem \ref{main}.

Our proof of the first claim of Theorem \ref{main} is constructive. We emphasize that, besides being more general, our proof is constructive while the proofs of the  non-graded, non-involutive version from \cite{Goodearl_book} and the non-involutive version from \cite{Roozbeh_Annalen} (and \cite{Roozbeh_graded_ring_notes}) are existential.

In section \ref{section_LPAs}, we apply our results to a class of Leavitt path algebras. Leavitt path algebras are the algebraic counterpart of graph $C^*$-algebras. The increasing popularity of graph $C^*$-algebras in the last two decades is likely due to the relative simplicity of their construction (especially compared to some other classes of $C^*$-algebras) and by the wide diversity of $C^*$-algebras which can be realized as graph $C^*$-algebras. Leavitt path algebras are even simpler objects than graph $C^*$-algebras and the class of algebras which can be represented as Leavitt path algebras also contains very diverse examples.  

Constructed from an underlying field and a directed graph $E$ to which so-called ghost edges have been added to the ``real'' ones, Leavitt path algebras were introduced by two groups of authors in \cite{Abrams_Aranda_Pino_1} and in \cite{Ara_Moreno_Pardo}. While both groups emphasized the close relation of a Leavitt path algebra $L_\C(E)$ over the complex numbers $\C$ and its $C^*$-algebra completion, the graph $C^*$-algebra $C^*(E)$, the second group of authors was particularly interested in the $K_0$-groups of algebras $L_\C(E)$ and $C^*(E)$ and showed that these two algebras have isomorphic $K_0$-groups when $E$ is a row-finite graph (\cite[Theorem 7.1]{Ara_Moreno_Pardo}). 
 
Any Leavitt path algebra is naturally equipped with an involution. Considering Leavitt path algebras over $\C$ equipped with the complex conjugate involution, the authors of \cite{Abrams_Tomforde} conjecture that any such algebras are isomorphic as algebras precisely when they are isomorphic as $*$-algebras, which is known as the {\em Isomorphism Conjecture} (IC), and that any such algebras are isomorphic as rings precisely when they are isomorphic as $*$-algebras, which is known as the {\em Strong Isomorphism Conjecture} (SIC). Results \cite[Propositions 7.4 and 8.5]{Abrams_Tomforde} show that (SIC) holds for Leavitt path algebras of countable acyclic graphs as well as row-finite cofinal graphs with at least one cycle and such that every cycle has an exit. 

When the underlying field is assumed to be any involutive field instead of $\C$, (IC) becomes the {\em Generalized Isomorphism Conjecture} (GIC) and (SIC) becomes the {\em Generalized Strong Isomorphism Conjecture} (GSIC). By \cite[Theorem 6.3]{Gonzalo_paper}, (GSIC) holds for the Leavitt path algebras of finite graphs in which no cycle has an exit (no-exit graphs). 

These isomorphism conjectures are closely related to the consideration of the $K_0$-group. It is known that the $K_0$-group alone is not sufficient for classification of Leavitt path algebras (see \cite{Ruiz_Tomforde}, examples of \cite{Roozbeh_Annalen} or our Examples \ref{infinite_graphs} and \ref{line_and_clock}). While \cite{Gabe_et_al} and \cite{Ruiz_Tomforde} present the classification of classes of Leavitt path algebras using higher $K$-groups, the first author of this paper took a different route to overcome the insufficiency of the $K_0$-group alone and attempted a classification of Leavitt path algebras by taking their grading and the {\em graded} $K_0$-group into account. 

A Leavitt path algebra is naturally equipped with a $\Z$-grading originating from the lengths of paths in the graph. Using this grading, one can replace the $K_0$-group with its graded version and attempt the classification of Leavitt path algebras by this new invariant. As demonstrated in \cite{Roozbeh_Israeli} and \cite{Roozbeh_Annalen}, the graded $K_0$-group carries more information about the Leavitt path algebra than the non-graded group and \cite[Theorem 4.8 and Corollary 4.9]{Roozbeh_Annalen} show that $K_0^{\gr}$ completely classifies the class of Leavitt path algebras of a class of finite graphs which are referred to as the polycephaly graphs. 

However, to relate the graded $K_0^{\gr}$-group with the isomorphism conjectures, one has to take into account the involution too.
As the results of  \cite{Gonzalo_Ranga_Lia} illustrate, the involutive structure of the underlying field $K$ directly impacts the involutive structure of the algebra $L_K(E).$  Consequently, it is important to consider {\em both the graded and the involutive structure} as we do in our approach. In particular, using Theorem \ref{main}, we show that the graded version of (GSIC) holds for a class of Leavitt path algebras. We summarize the main results of section \ref{section_LPAs}, Theorem \ref{classification} and Corollary \ref{graded_GSIC}, as follows. 

\begin{theorem}\label{main_LPA}
Let $E$ and $F$ be countable, row-finite, no-exit graphs in which each infinite path ends in a sink or a cycle and let $K$ be a 2-proper, $*$-pythagorean field.  

The Leavitt path algebras $L_K(E)$ and $L_K(F)$ are graded $*$-isomorphic if and only if there is a contractive $\Z[x,x^{-1}]$-module isomorphism $K^{\gr}_0(L_K(E))\cong K^{\gr}_0(L_K(F)).$ Consequently, $L_K(E)$ and $L_K(F)$ are isomorphic as graded $*$-algebras if and only if they are isomorphic as graded rings. 
\end{theorem}

By extending the class of Leavitt path algebras for which $K_0^{\gr}$ provides a complete classification, this result widens the class of algebras for which the Classification Conjectures from \cite{Roozbeh_Annalen} hold. With the involutive structure of Leavitt path algebras taken into account, \cite[Conjecture 1]{Roozbeh_Annalen} can be reformulated as follows.  

\begin{conjecture} 
If $K$ is a 2-proper, $*$-pythagorean field, then $K^{\gr}_0$ completely classifies Leavitt path algebras over $K$ as graded $*$-algebras.  
\end{conjecture}
Theorem \ref{main_LPA} implies that this conjecture holds for Leavitt path algebras of countable, row-finite, no-exit graphs in which each infinite path ends in a sink or a cycle.

\section{Graded rings with involution}\label{section_graded_star}

We start this section by recalling some concepts of the involutive and the graded ring theory. Then, we introduce the graded $*$-rings and study the action of $\Z_2$ on their graded Grothendieck groups. 

In sections \ref{section_graded_star}--\ref{section_faithfulness}, all the rings are assumed to be unital while in sections \ref{section_ultramatricial} and \ref{section_LPAs} we emphasize when that may not be the case.  

\subsection{Rings with involution}\label{subsection_non-graded_involutive}

Let $A$ be a ring with an involution denoted by  $a\mapsto  a^*$, i.e., ${}^*:A \rightarrow A$ is 
an anti-automorphism of order two. In this case the ring $A$ is also called a $*$-ring. The involution $^*$ is a ring isomorphism between $A$ and the opposite ring of $A.$ Thus, every right $A$-module $M$ also has a left $A$-module structure given by 
\[am:=m a^*.\]
As a consequence, the categories of the left and right $A$-modules are equivalent. 

In particular, if $A$ is an involutive ring and $M$ a right $A$-module, the left $A$-module $\Hom_A(M,A)$ is also a right $A$-module by $(fa)(x)=a^* f(x)$ for any $f\in \Hom_A(M,A),$ $a\in A$ and $x\in M.$ This action induces the action $[P]\mapsto [\Hom_A(P, A)]$ of the group $\Z_2$ on the monoid of the isomorphism classes of finitely generated projective modules and, ultimately, the $\Z_2$-action on $K_0(A).$

The matrix ring  $\M_n(A)$ also becomes an involutive ring with $(a_{ij})^*=(a_{ji}^*)$ and we refer to this involution as the $*$-transpose. If $p$ is an idempotent matrix in $\M_n(A),$ then the modules $p^*A^n$ and $\Hom_A(pA^n, A)$ are isomorphic as right $A$-modules. Thus, the $\Z_2$-action on $K_0(A)$ represented via the conjugate classes of idempotent matrices is given by $[p]\mapsto [p^*].$ The natural isomorphism of $K_0(A)$ and $K_0(\M_n(A))$ is a $\Z_2$-module isomorphism. 

If $A$ is a $*$-field, any idempotent of $\M_n(A)$ is conjugate with the diagonal matrix with $0$ or $1$ on each diagonal entry and the action of $\Z_2$ on $K_0(A)$ is trivial. The involution of $A\times A$ where $A$ is a field, given by $(a,b)^*=(b,a),$ induces a nontrivial action on $K_0(A\times A)\cong \Z\times \Z$ given by $(m,n)\mapsto (n,m)$ for $m, n\in \Z.$ 

If a $*$-ring $A$ is also a $K$-algebra for some commutative, involutive ring $K,$ then $A$ is a {\em
$*$-algebra} if $(kx)^*=k^*x^*$ for $k\in K$ and $x\in A.$

\subsection{Graded rings} 

If $\Gamma$ is an abelian group, a ring $A$ is called a \emph{$\Gamma$-graded ring} if $ A=\bigoplus_{ \gamma \in \Gamma} A_{\gamma}$ such that each $A_{\gamma}$ is
an additive subgroup of $A$ and $A_{\gamma}  A_{\delta} \subseteq
A_{\gamma + \delta}$ for all $\gamma, \delta \in \Gamma$. The group $A_\gamma$ is called the $\gamma$-\emph{component} of $A.$
When it is clear from the context that a ring $A$ is graded by a group $\Gamma,$ we simply say that $A$ is a \emph{graded ring}. 

For example, if $A$ is any ring, the group ring $A[\Gamma]$ is $\Gamma$-graded 
with the $\gamma$-component consisting of the elements of the form $a\gamma$ for $a\in A.$

The subset $\bigcup_{\gamma \in \Gamma} A_{\gamma}$ of a graded ring $A$ is called the set of \emph{homogeneous elements} of $A.$ The nonzero elements of $A_\gamma$ are called \emph{homogeneous of degree $\gamma$} and we write $\deg(a) = \gamma$ for $a \in A_{\gamma}\backslash \{0\}.$ The set $\Gamma_A=\{ \gamma \in \Gamma \mid A_\gamma \not = 0 \}$ is called the \emph{support}  of $A$. We say that a $\Gamma$-graded ring $A$ is \emph{trivially graded} if the support of $A$ is the trivial group $\{0\}$, i.e., $A_0=A$ and $A_\gamma=0$ for $\gamma \in \Gamma \backslash \{0\}$. Note that every ring can be trivially graded by any abelian group. 

If $A$ is a $\Gamma$-graded ring, we define the {\em inversely graded ring} $A^{(-1)}$ to be the $\Gamma$-graded ring with the $\gamma$-component 
\[A^{(-1)}_\gamma = A_{-\gamma}.\]
We note that the inversely graded ring is the $(-1)$-th Veronese subring of $A$ as defined in \cite[Example 1.1.19]{Roozbeh_graded_ring_notes} if $\Gamma$ is the group of integers. 

A graded ring $A$ is a \emph{graded division ring} if every
nonzero homogeneous element has a multiplicative inverse. If a graded division ring $A$ is also commutative then $A$ is a {\em graded field}. 
Note that if $A$ is a $\Gamma$-graded division ring, then $\Gamma_A$ is a subgroup of $\Gamma.$ 

A ring homomorphism $f$ of $\Gamma$-graded rings $A$ and $B$ is a $\Gamma$-graded ring homomorphism, or simply a {\em graded homomorphism}, if $f(A_\gamma)\subseteq B_\gamma$
for all $\gamma\in \Gamma$. If such map $f$ is also an isomorphism, we say that it is a {\em graded isomorphism} and write $A \cong_{\gr} B.$ 

A \emph{graded right $A$-module} is a right $A$-module $M$ with a direct sum decomposition $M =\bigoplus_{\gamma\in\Gamma} M_\gamma$ where $M_\gamma$
is an additive subgroup of $M$ such that $M_\gamma A_\delta \subseteq M_{\gamma+\delta}$ for all $\gamma,\delta\in \Gamma$. In this case, for $\delta\in\Gamma,$ the
$\delta$-\emph{shifted} graded right $A$-module $M(\delta)$ is defined by  
\[M(\delta) =\bigoplus_{\gamma \in \Gamma} M(\delta)_\gamma, \text{ where }M(\delta)_\gamma = M_{\delta+\gamma}.\] 

A right $A$-module homomorphism $f$ of graded right $A$-modules $M$ and $N$ is a {\em graded homomorphism} if $f(M_\gamma)\subseteq N_\gamma$
for any $\gamma\in \Gamma$. If such map $f$ is also an isomorphism, we say that it is a {\em graded isomorphism} and write $M \cong_{\gr} N.$ 
A graded ring $A$ is a \emph{graded algebra} over a graded field $K$, if $A$ is an algebra over $K$ and the unit map $K\to A$ is a graded homomorphism.

A right $A$-module homomorphism $f$ of graded right $A$-modules $M$ and $N$ is a {\em graded homomorphism of degree $\delta$} if $f(M_\gamma)\subseteq N_{\gamma+\delta}$ for all $\gamma\in \Gamma$. 
When $M$ is finitely generated, every element of the abelian group $\Hom_A(M, N)$ can be represented as a sum of the graded homomorphisms of degree $\delta$ for $\delta\in \Gamma.$ Thus, $\Hom_A(M, N)$ is $\Gamma$-graded and we write $\Hom_A(M, N)_\delta$ for the $\delta$-component of this group (see \cite[Theorem 1.2.6]{Roozbeh_graded_ring_notes}). If $\Gr A $ denotes the category of graded right $A$-modules and graded homomorphisms, then $\Hom_{\Gr A}(M,N)$ corresponds exactly to the zero-component $\Hom_A(M,N)_0.$ Moreover, $\Hom_A(M, N)_\gamma =\Hom_{\Gr A}(M(-\gamma), N)$ as sets (see \cite[(1.15)]{Roozbeh_graded_ring_notes}).

Graded left modules, graded bimodules, and graded homomorphisms of left modules and bimodules are defined analogously. 
If $M$ is a finitely generated  graded right $A$-module, then the abelian group $\Hom_A(M,A)$ has a structure of graded left $A$-module. Thus, the {\em  dual} $M^*=\Hom_A(M,A)$ of $M$
is a graded left $A$-module with 
\begin{equation}\label{dual_shift}
M^*(\gamma)=\Hom_A(M, A)(\gamma)=\Hom_A(M(-\gamma), A)=M(-\gamma)^*
\end{equation}
by \cite[Theorem 1.2.6 and (1.19)]{Roozbeh_graded_ring_notes}. 

\subsection{Graded rings with involution} 

\begin{definition}
Let $A$ be a $\Gamma$-graded ring with involution $^*$  and let $A_\gamma ^*$ denote the set $\{ a^* \,|\, a\in A_\gamma\}.$ We say that $A$ is a \emph{graded $*$-ring} if 
$$A_\gamma ^*\subseteq A_{-\gamma}$$
for every $\gamma\in \Gamma.$
\end{definition}
Note that if $A_\gamma ^*\subseteq A_{-\gamma}$ for every $\gamma\in \Gamma$ then  
$$A_\gamma ^*= A_{-\gamma}$$ for every  $\gamma\in \Gamma$. Indeed, if $a\in A_{-\gamma}$ then $a^*\in A_{-\gamma}^*\subseteq A_{\gamma}$ and so $a=(a^*)^*\in A_{\gamma}^*.$

A graded ring homomorphism $f$ of graded rings $A$ and $B$ is a {\em graded $*$-homomorphism} if $f(x^*)=f(x)^*$ for every $x\in A.$  

For example, if $A$ is a $*$-ring, the group ring $A[\Gamma]$ is a $\Gamma$-graded $*$-ring with the grading given by $A[\Gamma]_\gamma=\{a\gamma\, |\,a\in A\}$ and the involution given by $(a\gamma)^*=a^*(-\gamma)$ for $a\in A$ and $\gamma\in\Gamma.$ Leavitt path algebras (see section \ref{subsection_LPAs}) are also graded $*$-rings. Example \ref{example_graded} exhibits more graded $*$-rings and some graded $*$-homomorphisms.

If $M$ is a graded right module over a $\Gamma$-graded $*$-ring $A$, the {\em inversely graded module} $M^{(-1)}$ is the $\Gamma$-graded left $A$-module with $M^{(-1)}_\gamma=M_{-\gamma}$ and the left $A$-action given by $am:=ma^*$ for $a\in A$ and $m\in M.$   
If $N$ is a graded left $A$-module, we define the inversely graded module $N^{(-1)}$ analogously. If $M$ is a graded right module and $\gamma, \delta\in \Gamma,$ then $\left(M^{(-1)}(\gamma)\right)_\delta=\left(M^{(-1)}\right)_{\delta+\gamma}=M_{-\delta-\gamma}=M(-\gamma)_{-\delta}=\left(M(-\gamma)^{(-1)}\right)_\delta.$ Thus   
\begin{equation}\label{inverse_shift}
M^{(-1)}(\gamma)=M(-\gamma)^{(-1)} 
\end{equation}
for any $\gamma\in \Gamma.$

\begin{lemma}\label{hom_lemma}
Let $M$ be a finitely generated graded right $A$-module. Then the graded right $A$-modules below are graded isomorphic.  
\[
\Hom_A(M, A)^{(-1)} \cong_{\gr} \Hom_A(M^{(-1)}, A)
\]
\end{lemma}
\begin{proof}
Consider the map $\Hom_A(M, A)^{(-1)}\to \Hom_A(M^{(-1)}, A)$ given by  $f\mapsto f^*$ where $f^*(m)=(f(m))^*.$ It is direct to check that this is a well-defined map and that it is  
a right $A$-module isomorphism. It is a graded map too since for $f\in \Hom_A(M, A)^{(-1)}_\gamma$, we have $f\in \Hom_A(M, A)_{-\gamma}$ for any $\gamma\in \Gamma$. This means that $f(M_\delta)\subseteq A_{\delta-\gamma}$ for any $\delta\in \Gamma.$ Thus, 
$$f^*(M^{(-1)}_{\delta})=f^*(M_{-\delta})=f(M_{-\delta})^*\subseteq (A_{-\delta-\gamma})^*\subseteq A_{\delta+\gamma}.$$
This shows that $f^*$  is in $\Hom_A(M^{(-1)}, A)_\gamma.$ 
\end{proof}

By this lemma, we have that 
\begin{equation}\label{inverse_dual}
\left(M^*\right)^{(-1)}\cong_{\gr}(M^{(-1)})^* 
\end{equation}
for any finitely generated right $A$-module $M$. Moreover, the action $\underline{\hskip.3cm}^{*(-1)}$ commutes with the shifts by (\ref{inverse_shift}) and (\ref{dual_shift}) so that 
\begin{equation}\label{action_and_shifts}
M^{*(-1)}(\gamma)=M(\gamma)^{*(-1)}  
\end{equation}
for any finitely generated right $A$-module $M$ and any $\gamma\in \Gamma$.

\subsection{Graded matrix *-rings and graded matricial *-algebras}\label{matrices}
For a $\Gamma$-graded ring $A$ and $(\gamma_1,\dots,\gamma_n)$ in $\Gamma^n$, let $\M_n(A)(\gamma_1,\dots,\gamma_n)$ denote the $\Gamma$-graded ring $\M_n(A)$ with the $\delta$-component consisting of the matrices $(a_{ij}),$ $i,j=1,\ldots, n,$ such that $a_{ij}\in A_{\delta+\gamma_j-\gamma_i}$ (more details in \cite[Section 1.3]{Roozbeh_graded_ring_notes}). We denote this fact by writing 
\begin{equation}\label{grading_on_matrices}
\M_n(A)(\gamma_1,\dots,\gamma_n)_{\delta} =
\begin{pmatrix}
A_{ \delta+\gamma_1 - \gamma_1} & A_{\delta+\gamma_2  - \gamma_1} & \cdots &
A_{\delta +\gamma_n - \gamma_1} \\
A_{\delta + \gamma_1 - \gamma_2} & A_{\delta + \gamma_2 - \gamma_2} & \cdots &
A_{\delta+\gamma_n  - \gamma_2} \\
\vdots  & \vdots  & \ddots & \vdots  \\
A_{\delta + \gamma_1 - \gamma_n} & A_{ \delta + \gamma_2 - \gamma_n} & \cdots &
A_{\delta + \gamma_n - \gamma_n}
\end{pmatrix}.
\end{equation}

If $A$ is a graded $*$-ring, then the \emph{$*$-transpose} $(a_{ij})^*=(a_{ji}^*)$, for $(a_{ij}) \in \M_n(A)(\gamma_1,\dots,\gamma_n)$, makes $\M_n(A)(\gamma_1,\dots,\gamma_n)$ into a graded $*$-ring.

The following proposition is the involutive version of \cite[Theorem 1.3.3]{Roozbeh_graded_ring_notes} and we shall extensively use it.  We also prove its generalization to matrices of possibly infinite size in Proposition \ref{permutation_general}. 

\begin{proposition}\label{permutation_of_components}
Let $A$ be a $\Gamma$-graded $*$-ring and let $\gamma_i \in \Gamma$, $i=1, \ldots,n$. 
\begin{enumerate}[\upshape(1)]
\item If $\delta$ is in $\Gamma$, and $\pi$ is a permutation  of $\{1, 2, \ldots, n \},$ then  the matrix rings
\begin{center}
$\M_n (A)(\gamma_1, \ldots, \gamma_n)\;\;$ and $\;\;\M_n (A)(\gamma_{\pi(1)}+\delta, \ldots, \gamma_{\pi(n)}+\delta)$
\end{center}
are graded $*$-isomorphic.

\item   If $\delta_1,\dots, \delta_n \in \Gamma$ are such that there is an element  $a_i\in A_{\delta_i}$ with the property that $a_ia^*_i=a^*_ia_i=1$ for $i=1,\ldots, n,$ then the matrix rings
\begin{center}
$\M_n (A)(\gamma_1, \ldots, \gamma_n)\;\;$ and $\;\;\M_n (A)(\gamma_1+\delta_1, \ldots, \gamma_n+\delta_n)$
\end{center}
are graded $*$-isomorphic.
\end{enumerate}
\end{proposition}
\begin{proof}
To prove the first part, we note that adding $\delta$ to each of the shifts does not change the matrix ring by formula (\ref{grading_on_matrices}). Consider the $n\times n$ permutation matrix $P_{\pi}$, with 1 at the $(i,\pi(i))$-th place for $i=1, \ldots, n$ and zeros elsewhere and the map
\[\phi:\M_n (A)(\gamma_1, \ldots, \gamma_n) \rightarrow  \M_n (A)(\gamma_{\pi(1)}, \ldots, \gamma_{\pi(n)})\;\;\;\mbox{ given by }\;\;\;
\phi: M \mapsto P_{\pi} M P_{\pi}^{-1}.\]
Since $P_{\pi}^{-1}=P_{\pi}^*$, $\phi$ is a $*$-isomorphism of rings. 

Note that $P_{\pi} e_{ij} P_{\pi}^{-1}=e_{\pi^{-1}(i)\pi^{-1}(j)}$ for any standard matrix unit 
$e_{ij},$  $i,j=1, \ldots, n,$ of $\M_n(A)(\gamma_1, \ldots, \gamma_n).$ In addition, we have that 
$\deg(e_{\pi^{-1}(i),\pi^{-1}(j)})=\delta_{\pi(\pi^{-1}(i))}-\delta_{\pi(\pi^{-1}(j))}=\delta_{i}-\delta_j=\deg(e_{ij}).$ This shows that $\phi$ is also a graded isomorphism.  

To prove the second part, let $a_i \in A_{\delta_i}$ such that $a_ia_i^*=a_i^*a_i=1$ and consider the diagonal matrix $P$ with $a_1,\dots,a_n$ on the diagonal. One can check that the map $M\mapsto P^{-1}MP$ gives a $*$-graded isomorphism. 
\end{proof}

\begin{remark}\label{noninvolutive_permutation} As we mentioned, Proposition \ref{permutation_of_components} is the involutive version of \cite[Theorem 1.3.3]{Roozbeh_graded_ring_notes}. The first parts of both Proposition \ref{permutation_of_components} and \cite[Theorem 1.3.3]{Roozbeh_graded_ring_notes} have analogous formulations and proofs. 

The assumptions of the second part of \cite[Theorem 1.3.3]{Roozbeh_graded_ring_notes} are adjusted to fit the involutive structure. In particular, the second part of \cite[Theorem 1.3.3]{Roozbeh_graded_ring_notes} states the following: if $\delta_1,\dots, \delta_n \in \Gamma$ are such that there is an invertible element  $a_i\in A_{\delta_i}$ for $i=1,\ldots, n,$ then the matrix rings
$\M_n (A)(\gamma_1, \ldots, \gamma_n)$ and $\M_n (A)(\gamma_1+\delta_1, \ldots, \gamma_n+\delta_n)$
are graded isomorphic. 
\end{remark}

\begin{definition}
Let  $A$ be a $\Gamma$-graded $*$-field. A \emph{graded matricial $*$-algebra over $A$}  is a graded $A$-algebra of the form 
\[\M_{n_1}(A)(\overline \gamma_1) \oplus \dots \oplus   \M_{n_k}(A)(\overline \gamma_k),\;\;\mbox{ for }\;\;\overline \gamma_i\in  \Gamma^{n_i},\; i=1, \ldots, k,\] where the involution is the $*$-transpose in each coordinate. 
\end{definition}

If $\Gamma$ is the trivial group, and the involutive structure is not considered, this definition reduces to the definition of a matricial algebra as in \cite[Section 15]{Goodearl_book}. Note that if $R$ is a graded matricial $*$-algebra over a graded field $A$, then $R_0$ is a matricial $A_0$-algebra with the involution being the coordinate-wise $*$-transpose.

\subsection{Graded Morita theory}
 
If $A$ is a $\Gamma$-graded ring, a graded free right $A$-module is defined as a graded right module which is a free right $A$-module with a
homogeneous basis (see \cite[Section 1.2.4]{Roozbeh_graded_ring_notes}). Graded free left modules are defined analogously.  
If $A$ is a $\Gamma$-graded ring, then  
\begin{equation}\label{n-th_power}
A^n(\gamma_1, \ldots,\gamma_n):=A(\gamma_1)\oplus \dots \oplus A(\gamma_n)
\end{equation} 
for $\gamma_1, \ldots,\gamma_n\in\Gamma$ is a graded free right $A$-module. Conversely, any finitely generated graded free right $A$-module is of this form. If we denote $\overline \gamma=(\gamma_1, \dots, \gamma_n)$ then 
$A^n(\overline \gamma)$ is a graded right $\M_n(A)(\overline \gamma)$-module and $A^n(-\overline \gamma)$ is a graded left $\M_n(A)(\overline \gamma)$-module.

Using $\Gr A$ to denote the category of graded right $A$-modules again, we recall the following result from \cite{Roozbeh_graded_ring_notes}. 

\begin{proposition}\cite[Proposition 2.1.1]{Roozbeh_graded_ring_notes}\label{equivalences_phi_and_psi}
Let $A$ be a $\Gamma$-graded ring and let $\ol \gamma = (\gamma_1 , \ldots ,
\gamma_n)\in \Gamma^n$. Then the functors
\begin{center}
\begin{tabular}{lllll}
$\phi : \Gr A$ & $\rightarrow \Gr \M_n(A)(\ol \gamma),\;\;$ & given by & 
$M \mapsto M \otimes_A A^n(\ol \gamma)$ & and \\
$\psi : \Gr \M_n(A)(\ol \gamma)$ & $\rightarrow \Gr A,\;\;$ & given by &
$N \mapsto  N\otimes_{\M_n(A)(\ol \gamma)} A^n(-\ol \gamma)$ &
\end{tabular}
\end{center}
are inverse equivalences of categories and they commute with the shifts. 
\end{proposition} 

In particular, from the proof of \cite[Proposition 2.1.1]{Roozbeh_graded_ring_notes}, it follows that the following graded $A$-bimodules are graded isomorphic for any $\ol \gamma\in \Gamma^n.$
\begin{equation}\label{bimodules_and_tensor}
A\cong_{\gr} A^n(\ol \gamma)\otimes_{\M_n(A)(\ol \gamma)} A^n(-\ol \gamma)
\end{equation}

A graded projective right $A$-module is a projective object in the abelian
category $\Gr A$. Equivalently (see \cite[Proposition 1.2.15]{Roozbeh_graded_ring_notes}), it is a graded module which is graded isomorphic to a direct summand of a graded free right $A$-module. If 
$e_{ij}$, $i,j=1, \ldots, n$ denote the standard matrix units in $\M_n(A)(\overline \gamma)$, then $e_{ii}\M_n(A)(\overline \gamma)$ is a finitely generated graded projective right $\M_n(A)(\overline \gamma)$-module. It is direct to check that there is a graded $\M_n(A)(\overline \gamma)$-module isomorphism 
\begin{equation} \label{matrix_unit_generated}
e_{ii}\M_n(A)(\overline \gamma)\cong_{\gr} A(\gamma_1-\gamma_i)\oplus A(\gamma_2-\gamma_i)\oplus \dots \oplus A(\gamma_n-\gamma_i), 
\end{equation}
given by $e_{ii}(x_{jl}) \mapsto  (x_{i1},x_{i2},\dots,x_{in}).$
 
We use the following fact in the proof of the next lemma. If $M$ is a graded right and $N$ a graded left $A$-module, then
\begin{equation}\label{tensor_product_shift} 
(M\otimes_A N)(\gamma)= M(\gamma)\otimes_A N= M\otimes_A N(\gamma)
\end{equation}
for every $\gamma\in \Gamma.$ This follows from \cite[Section 1.2.6]{Roozbeh_graded_ring_notes}

\begin{lemma}\label{matrix_units_lemma}
If $A$ is a $\Gamma$-graded ring, $\overline\gamma\in \Gamma^n,$ $e_{ii}$ is the standard matrix unit in $\M_n(A)(\overline \gamma)$ for some $i=1,\ldots, n,$ and $\psi$ is the equivalence from Proposition~\ref{equivalences_phi_and_psi}, then  $$\psi\left(e_{ii}\M_n(A)(\overline \gamma)\right)= A(-\gamma_i)$$
and 
\[e_{ii}\M_n(A)(\overline \gamma)\cong_{\gr}e_{jj}\M_n(A)(\overline \gamma)\;\;\mbox{ if and only if }\;\;\gamma_i-\gamma_j\in \Gamma_A\]
for any $i,j=1,\ldots, n.$
\end{lemma}
\begin{proof} For any $i=1,\ldots, n$ and a matrix unit $e_{ii},$ we have the following by the definition of $\psi,$ (\ref{matrix_unit_generated}), (\ref{n-th_power}), (\ref{tensor_product_shift}), and (\ref{bimodules_and_tensor}) respectively. 
$$\begin{array}{ll}
\psi\left(e_{ii}\M_n(A)(\overline \gamma)\right) & = e_{ii}\M_n(A)(\overline \gamma) \otimes_{\M_n(A)(\ol \gamma)} A^n(-\ol \gamma)   \\ 
& \cong_{\gr}\left( A(\gamma_1-\gamma_i)\oplus A(\gamma_2-\gamma_i)\oplus \dots \oplus A(\gamma_n-\gamma_i) \right)   \otimes_{\M_n(A)(\ol \gamma)} A^n(-\ol \gamma)  \\
& \cong_{\gr} A^n(\ol \gamma)(-\gamma_i)\otimes_{\M_n(A)(\ol \gamma)} A^n(-\ol \gamma)\\
& \cong_{\gr} \big( A^n(\ol \gamma)\otimes_{\M_n(A)(\ol \gamma)} A^n(-\ol \gamma)\big)(-\gamma_i)\\
& \cong_{\gr} A(-\gamma_i).  
\end{array}
$$ 
The second part of the lemma follows from the first and \cite[Proposition 1.3.17]{Roozbeh_graded_ring_notes} by which  
\begin{center}
$A(-\gamma_i)\cong_{\gr}A(-\gamma_{j})$ if and only if $\gamma_i-\gamma_j\in\Gamma_A$
\end{center}
for any $i,j=1,\ldots, n.$
\end{proof}

\subsection{Graded K-theory}

For a $\Gamma$-graded ring $A$, let $\Pgr[\Gamma] A$  denote the category of  finitely generated
graded  projective right $A$-modules. This is an exact category with the usual notion of (split) short exact
sequences. Thus, one can apply Quillen's construction (\cite{Quillen}) to obtain $K$-groups
$K_i(\Pgr[\Gamma] A),$ for $i\geq 0$, which we denote by $K_i^{\gr}(A)$. 
Note that for $\gamma \in \Gamma$,  the  $\gamma$-shift functor $\mathcal T_\gamma:\Gr A\rightarrow \Gr A$, $M \mapsto M(\gamma)$  is an isomorphism with the property
$\mathcal T_\gamma \mathcal T_\delta=\mathcal T_{\gamma + \delta}$, $\gamma,\delta\in \Gamma$.
Furthermore, $\mathcal T_\gamma$ restricts to $\Pgr A$. Thus the group $\Gamma$ acts on the category
$\Pgr[\Gamma] A$ by $(\gamma, P) \mapsto P(\gamma)$. By functoriality of the $K$-groups this equips $K_i^{\gr}(A)$
with the structure of a $\Z[\Gamma]$-module. 

Now, consider $A$ to be a graded $*$-ring and $P$ a finitely generated graded projective right $A$-module.   
Since $P^{**} \cong_{\gr} P$ as graded right $A$-modules, $(P^*)^{(-1)}\cong_{\gr}(P^{(-1)})^*$ and $P^{*(-1)}(\gamma)= P(\gamma)^{*(-1)}$ by (\ref{inverse_dual}) and (\ref{action_and_shifts}), 
the group $\Z_2$ acts on $\Pgr[\Gamma] A$, by $P \mapsto {P^*}^{(-1)}$ and this action commutes with the $\Gamma$-action. 
This makes $K^{\gr}_i$ a $\Z[\Gamma]$-$\Z[\Z_2]$-bimodule. In this paper, we exclusively work with the graded Grothendieck group $K^{\gr}_0$. We review an explicit construction of this group in the following section.

\subsection{The graded Grothendieck group of a graded *-ring.}

Let $A$ be a graded $*$-ring and let $\mathcal V^{\gr}(A)$ denote the monoid of graded isomorphism classes $[P]$ of finitely generated graded projective $A$-modules $P$ with the direct sum $[P]+[Q]=[P\oplus Q]$ as the addition operation. The $\Gamma$-$\Z_2$-action on the elements of the monoid $\mathcal V^{\gr}(A)$ from the previous section agrees with the addition. 

The \emph{graded Grothendieck group}  $K_0^{\gr}(A)$ is defined as the group completion of  the monoid $\mathcal V^{\gr}(A)$ which naturally inherits 
the $\Z[\Gamma]$-$\Z[\Z_2]$-bimodule structure from $\mathcal V^{\gr}(A)$. In the case when the group $\Gamma$ is trivial, $K^{\gr}_0(A)$ becomes the usual $K_0$-group equipped with the $\Z_2$-action we considered in section \ref{subsection_non-graded_involutive}. 

The action of $\Z_2$ is trivial on the elements of $\mathcal V^{\gr}(A)$ of the form $[P]$ where $P$ is a finitely generated graded {\em free} module as the following lemma shows.

\begin{lemma}\label{fin_gen_free}
Let $A$ be a $\Gamma$-graded $*$-ring and let $P$ be a finitely generated graded {\em free} right $A$-module. Then 
\[P^{*(-1)}\cong_{\gr} P.\]
\end{lemma}
\begin{proof}
Note that $A\cong_{\gr} (A^{(-1)})^*$ by the map $a\mapsto L_{a^*}$ where $L_{a^*}$ is the left multiplication by $a^*.$ This right $A$-module isomorphism is also a graded homomorphism since if $a\in A_\gamma$ and $b\in A_{\delta}$ then $L_{a^*}(b)=a^*b\in A_{\delta-\gamma}$ and so $L_{a^*}\in \Hom_A(A, A)_{-\gamma}=\Hom_A(A, A)^{(-1)}_{\gamma}=(A^*)^{(-1)}_\gamma.$ 

If $P$ is a finitely generated graded free right $A$-module, so that $P\cong_{\gr} A(\gamma_1)\oplus \dots \oplus A(\gamma_n)$, for some positive integer $n$ and $\gamma_i \in \Gamma,$ $i=1,\ldots, n$, then  
$$
\begin{array}{llllll}
P^{*(-1)} & \cong_{\gr} & \Big(A(\gamma_1)^* \Big.& \oplus\;\; \dots\;\; \oplus &\Big. A(\gamma_n)^*\Big)^{(-1)}&\\
& = & \Big(A^*(-\gamma_1)\Big. &\oplus\;\; \dots\;\;  \oplus & \Big.A^*(-\gamma_n)\Big)^{(-1)}& \mbox{ by (\ref{dual_shift})}\\
&\cong_{\gr}& A^*(-\gamma_1)^{(-1)} &\oplus\;\; \dots\;\;  \oplus & A^*(-\gamma_n)^{(-1)}&\\
&=& A^{*(-1)}(\gamma_1) & \oplus\;\; \dots\;\;  \oplus & A^{*(-1)}(\gamma_n)& \mbox{ by  (\ref{inverse_shift})}\\
& \cong_{\gr}& A(\gamma_1) & \oplus\;\; \dots\;\;  \oplus & A(\gamma_n)& \mbox{ by  the previous paragraph}\\
&\cong_{\gr}& P.&& \\
\end{array}$$   
\end{proof}

Lemma \ref{fin_gen_free} is used in the following result which demonstrates that the $\Z_2$-action on $K^{\gr}_0(A)$ is trivial when $A$ is a graded division $*$-ring. We also describe the graded Grothendieck group in this case. We use $\mathbb N$ to denote the set of nonnegative integers.    

\begin{proposition}\label{graded_division_ring}
Let $A$ be a graded division $*$-ring.
\begin{enumerate}[\upshape(1)]
\item The action of $\Z_2$ on $K_0^{\gr}(A)$ is trivial.

\item The action of $\Z_2$ on $K^{\gr}_0(\M_n(A)(\gamma_1,\dots,\gamma_n))$ is trivial for any positive integer $n$ and $\gamma_i \in \Gamma,$ for $i=1,\ldots, n.$ 

\item \cite[Proposition 3.7.1]{Roozbeh_graded_ring_notes} There is a monoid isomorphism of $\mathcal V^{\gr}(A)$ and $\mathbb N[\Gamma /\Gamma_A]$ which induces a canonical $\Z[\Gamma]$-module isomorphism 
\begin{center}
$K_0^{\gr}(A) \rightarrow \Z[\Gamma /\Gamma_A]\;\;$ given by 
$\;\;[A^n(\gamma_1,\dots,\gamma_n)] \mapsto \sum_{i=1}^n \gamma_i +\Gamma_A.$ 
\end{center}
\end{enumerate}
\end{proposition}
\begin{proof}
The first part follows directly from Lemma \ref{fin_gen_free} since any graded module over a graded division ring is graded free.

To prove the second part, let $\overline{\gamma}=(\gamma_1,\dots,\gamma_n)\in \Gamma^n$ and let $P$ be a finitely generated graded projective right $\M_n(A)(\overline{\gamma})$-module. If $\psi$ is the equivalence of categories from Proposition \ref{equivalences_phi_and_psi}, then $\psi(P)^{*(-1)}$ and  $\psi(P)$ are graded isomorphic by the first part. Thus, to show that $P$ and $P^{*(-1)}$ are graded isomorphic, it is sufficient to show that $\psi(P)^{*(-1)}$ and $\psi(P^{*(-1)})$ are graded isomorphic. Let $F=A(-\gamma_1)\oplus \dots \oplus A(-\gamma_n)$ so that $\psi(P)=P \otimes_{\M_n(A)(\ol \gamma)}F$ and $\psi(P^{*(-1)})=P^{*(-1)} \otimes_{\M_n(A)(\ol \gamma)} F.$ Since $F^{*(-1)}$ is graded isomorphic to $F$ by part (1), we have that  
\begin{multline*}
\psi(P)^{*(-1)} = (P \otimes_{\M_n(A)(\ol \gamma)} F )^{*(-1)}\cong_{\gr} (F^* \otimes_{\M_n(A)(\ol \gamma)} P^*)^{(-1)}\cong_{\gr}\\
P^{*(-1)} \otimes_{\M_n(A)(\ol \gamma)} F^{*(-1)}\cong_{\gr}P^{*(-1)} \otimes_{\M_n(A)(\ol \gamma)} F=\psi(P^{*(-1)}).
\end{multline*}

The third part of the lemma holds since the map 
\[\mathcal V^{\gr}(A)\to \mathbb N[\Gamma /\Gamma_A]\mbox{ given by }[A(\gamma_1)^{r_1}\oplus\ldots\oplus A(\gamma_n)^{r_n}]\mapsto r_1(\gamma_1+\Gamma_A)+ \ldots +r_n(\gamma_n+\Gamma_A)\]
is a well-defined monoid isomorphism. The details can be found in \cite[Proposition 3.7.1]{Roozbeh_graded_ring_notes}. 
\end{proof}

\subsection{Pre-order on the graded Grothendieck group}
The group $K^{\gr}_0(A)$ of a $\Gamma$-graded ring $A$ is a pre-ordered $\Z[\Gamma]$-module with the monoid $\mathcal V^{\gr}(A)$ as the positive cone of the pre-order, i.e. $[P]\leq [Q]$ if and only if $[Q]-[P]\in \mathcal V^{\gr}(A).$ An element $u$ of $K^{\gr}_0(A)$ is called an \emph{order-unit}  if for any $x\in K^{\gr}_0(A)$, there is a positive integer $n$ and $\gamma_1,\dots,\gamma_n \in \Gamma,$ such that 
$x\leq \sum_{i=1}^n \gamma_i u .$
Thus, the element $[A]\in K^{\gr}_0(A)$ is an order-unit. If $\Gamma$ is the trivial group, the definition of an order-unit boils down to the one in \cite[page 203]{Goodearl_book}. 

If $A$ and $B$ are graded rings, a $\Z[\Gamma]$-module homomorphism $f: K_0^{\gr}(A)\to K_0^{\gr}(B)$ is {\em contractive} if $f$ is order-preserving (i.e. $x\geq 0$ implies $f(x)\geq 0$) and $0\leq x\leq [A]$ implies that $0\leq f(x)\leq [B].$ The map $f$ is {\em unit-preserving} if  $f([A])=[B].$
It is direct to check that $f$ is contractive if and only if $f$ is order-preserving and $f([A])\leq [B].$ We use the term contractive following the terminology from the $C^*$-algebra theory.

It is direct to check that any graded ring homomorphism induces a contractive map on the $K^{\gr}_0$-groups. Moreover, if $A$ and $B$ are $\Gamma$-graded $*$-rings and $\phi: A\to B$ is a graded $*$-homomorphism, then $K^{\gr}_0(\phi)$ is a contractive $\Z[\Gamma]$-$\Z[\Z_2]$-bimodule homomorphism. 

\subsection{The graded Grothendieck group via the idempotent matrices.}
The graded Grothendieck group of a $\Gamma$-graded ring $A$ can be represented by the equivalence classes of idempotent matrices, just as in the non-graded case. We refer the reader to \cite[Section 3.2]{Roozbeh_graded_ring_notes} for details of the fact that every finitely generated graded projective right $A$-module $P$ is graded isomorphic to the module $pA^n(-\overline\gamma)$ for some idempotent $p\in\M_n(A)(\overline\gamma)_0$ and some $\overline\gamma\in \Gamma^n.$ Two such finitely generated graded projective modules $P\cong_{\gr}pA^n(-\overline\gamma)$ and $Q\cong_{\gr}qA^m(-\overline\delta),$ where $q$ is an idempotent in $\M_m(A)(\overline\delta)_0$ and $\overline\delta\in\Gamma^m,$ are graded isomorphic if and only if the idempotents $p$ and $q$ are equivalent in the following sense: there are  $x\in \Hom_{\Gr A}(A^n(-\overline\gamma), A^m(-\overline\delta))$ and $y\in\Hom_{\Gr A}(A^m(-\overline\delta), A^n(-\overline\gamma))$ such that $xy = p$ and $yx = q$ 
where the idempotents $p$ and $q$ are considered as maps in $\Hom_{\Gr A}(A^n(-\overline\gamma), A^n(-\overline\gamma))$ and $\Hom_{\Gr A}(A^m(-\overline\delta), A^m(-\overline\delta))$ respectively. The equivalence of homogeneous matrix idempotents can also be defined using the mixed shifts on matrices as in \cite[Section 1.3.4]{Roozbeh_graded_ring_notes}. In particular, $\Hom_{\Gr A}(A^n(-\overline\gamma), A^m(-\overline\delta))=\Hom_A(A^n(-\overline\gamma), A^m(-\overline\delta))_0=\M_{m\times n}(A)[-\ol\delta][-\ol\gamma]$ and $\Hom_{\Gr A}(A^n(-\overline\gamma), A^n(-\overline\gamma))=\M_{n\times n}(A)[-\ol\gamma][-\ol\gamma]=\M_n(A)(\ol\gamma)_0$.
 
Thus, $K_0^{\gr}(A)$ can be defined as the Grothendieck group of the monoid of the equivalence classes $[p]$ of homogeneous idempotents $p$ with the addition given by 
\[[p]+[q]=\left[\left(
\begin{array}{cc}
p & 0 \\0 & q 
\end{array}\right)\right].
\]
The two representations of $K_0^{\gr}(A)$ are isomorphic by $[p]\mapsto [pA^n(-\ol\gamma)]$ for $p\in\M_n(A)(\ol\gamma)_0$. 
The group $\Gamma$ acts on the equivalence classes of idempotent matrices by $\delta[p]=[\delta p]$ where $\delta p$ is represented by the same matrix as $p\in \M_n(A)(\ol\gamma)_0$ except that it is considered as an element of $\M_n(A)(\ol\gamma-\delta)_0$ for $\ol\gamma-\delta=(\gamma_1-\delta, \ldots, \gamma_n-\delta).$ The matrix rings $\M_n(A)(\ol\gamma)_0$ and $\M_n(A)(\ol\gamma-\delta)_0$ are the same by definition (see formula (\ref{grading_on_matrices})). 

Consider the standard matrix units $e_{ii}\in\M_n(A)(\ol\gamma)$ for $i=1,\ldots, n$ now. Lemma \ref{matrix_units_lemma} and the definition of the action of $\Gamma$ on $K_0^{\gr}(\M_n(A)(\ol\gamma))$ imply that 
\begin{equation}\label{matrix_units_identity}
[e_{ii}]=(\gamma_j-\gamma_i)[e_{jj}] 
\end{equation}
for any  $i,j=1,\ldots, n.$

If the graded ring $A$ is also a graded $*$-ring, we show that the $\Z_2$-action on $K_0^{\gr}(A)$ can be represented using the $*$-transpose. We need a preliminary lemma and some notation before proving this fact in Proposition \ref{K0_via_idempotents}. 

If $A$ is a $\Gamma$-graded ring and $\ol\gamma\in \Gamma^n$ we let 
$e_i, i=1,\ldots, n$ denote the standard basis of $A^n(-\ol\gamma)$ represented as the column vectors
$\left(\begin{array}{c}
1\\ 0 \\ \vdots\\ 0   
 \end{array}\right),$
 $\left(\begin{array}{c}
0\\ 1 \\ \vdots\\ 0   
 \end{array}\right), \ldots$
 $\left(\begin{array}{c}
0\\ 0 \\ \vdots\\ 1   
 \end{array}\right).$
Note that $\deg(e_i)=\gamma_i.$ Also, let $e^i$ denote the standard basis of $A^n(\ol\gamma)$ represented as the row vectors
$(1\; 0\; \ldots\; 0), (0\; 1\; \ldots\; 0),\ldots, (0\; 0\; \ldots\; 1).$ Note that $\deg(e^i)=-\gamma_i.$

\begin{lemma}\label{dual_of_idempotent_generated_module}
Let $A$ be a $\Gamma$-graded ring, let $\ol\gamma\in \Gamma^n,$ and let $p\in \M_n(A)(\ol \gamma)_0$ be an idempotent. Then 
$$ \Hom_A(pA^n(-\ol\gamma), A)\cong_{\gr} A^n(\ol\gamma)p$$
and the maps 
$$
\begin{array}{ll}
\phi: \Hom_A(pA^n(-\ol\gamma), A)\to A^n(\ol\gamma)p\hskip1cm & f \mapsto \sum_{i=1}^n f(pe_i)e^ip\mbox{ and}\\
\psi: A^n(\ol\gamma)p\to \Hom_A(pA^n(-\ol\gamma), A)\hskip1cm & xp=\sum_{i=1}^n x_ie^ip\mapsto \left(pe_i\mapsto x_i\right),\;\;i=1,\ldots,n
\end{array}
$$
are mutually inverse graded homomorphisms of graded left $A$-modules. Analogously, the graded right modules $\Hom_A(A^n(\ol\gamma)p, A)$ and $pA^n(-\ol\gamma)$ are graded isomorphic.  
\end{lemma}

\begin{proof}
It is direct to check that $\phi$ and $\psi$ are mutually inverse homomorphisms of left $A$-modules. We show that $\phi$ and $\psi$ are graded homomorphisms. Since $p\in \Hom_A(A^n(-\ol\gamma), A^n(-\ol\gamma))_0,$ we have that $\deg(pe_i)=\gamma_i$ for all $i=1,\ldots,n.$ Similarly, $\deg(e^ip)=-\gamma_i$ for $i=1,\ldots,n.$
If $f\in \Hom_A(pA^n(-\ol\gamma), A)_\delta,$ for some $\delta\in\Gamma,$ then $\deg(f(pe_i))=\gamma_i+\delta$ for all $i=1,\ldots,n.$ Thus, $\deg( f(pe_i)e^ip)=\gamma_i+\delta-\gamma_i=\delta$ for any $i=1,\ldots,n$ and hence $\deg(\psi(f))=\delta.$ Consequently, $\phi$ is a graded homomorphism. 

To show that $\psi$ is a graded homomorphism, let $x=\sum_{i=1}^n x_ie^i\in A^n(\ol\gamma)$ be such that $\deg(xp)=\delta$ so that $\deg(x_i)=\delta+\gamma_i.$ We show that $\psi(xp)$ is a graded homomorphism of degree $\delta.$ Indeed, if $py=\sum_{i=1}^n pe_iy_i$ is of degree $\alpha$ for some $\alpha\in \Gamma,$ then  $\deg(y_i)=\alpha-\gamma_i.$ Thus,  
$\deg(x_iy_i)=\gamma_i+\delta+\alpha-\gamma_i=\delta+\alpha$ for any $i=1,\ldots,n,$ so that $\psi(xp)(py)=\sum_{i=1}^n x_iy_i$ has degree $\delta+\alpha.$ Thus, $\psi$ is a graded homomorphism.
\end{proof}

\begin{proposition}\label{K0_via_idempotents}
If $A$ is a $\Gamma$-graded $*$-ring, $P$ a finitely generated graded projective right $A$-module and $p\in M_n(A)(\overline{\gamma})_0$ an idempotent such that $pA^n(-\overline{\gamma})\cong_{\gr}P,$ then the modules $p^*A^n(-\overline{\gamma})$ and $(pA^n(-\overline{\gamma}))^{*(-1)}$ are isomorphic as graded right $A$-modules. 
Thus,  the $\Z_2$-action 
\begin{center}
$[P]=[pA^n(-\overline{\gamma})]\mapsto [P^{*(-1)}]=[p^*A^n(-\overline{\gamma})]$ corresponds to $[p]\mapsto [p^*].$
\end{center}
\end{proposition}
\begin{proof}
Note that the module $pA^n(-\overline{\gamma})^{(-1)}$ is graded isomorphic to $A^n(\ol\gamma)p^*$ by the map $px\mapsto x^*p^*.$ Thus we have that \[(pA^n(-\overline{\gamma}))^{*(-1)}=\Hom_A(pA^n(-\overline{\gamma}), A)^{(-1)}\cong_{\gr}\Hom_A(pA^n(-\overline{\gamma})^{(-1)}, A)\cong_{\gr} \Hom_A(A^n(\overline{\gamma})p^*, A)\]
by Lemma \ref{hom_lemma}. The module $\Hom_A(A^n(\overline{\gamma})p^*, A)$ is graded isomorphic to  $p^*A^n(-\overline{\gamma})$ by Lemma \ref{dual_of_idempotent_generated_module}.
\end{proof}

By this proposition, if an idempotent matrix $p\in\M_n(A)(\ol\gamma)_0$ is a {\em projection} (i.e. a self-adjoint idempotent), then the action of $\Z_2$ on $[pA^n(-\gamma)]$ is trivial.

\section{Fullness}\label{section_fullness}

Let $\mathcal C$ denote the category whose objects are graded matricial algebras over a $\Gamma$-graded $*$-field $A$ and whose morphisms are graded $*$-algebra homomorphisms. Let $\mathcal P$ denote the category whose objects are of the form $(G, u)$ where $G$ is a pre-ordered $\Z[\Gamma]$-module and $u$ is an order-unit in $G,$ and whose morphisms are contractive $\Z[\Gamma]$-module homomorphisms (not necessarily unit-preserving). Then $K_0^{\gr}: \mathcal C \to \mathcal P$ defines a functor. The goal of this section is to show that this functor is full under the assumption that each nontrivial graded component of $A$ has a unitary element in which case we say that $A$ has enough unitaries. This assumption is satisfied in most of the relevant cases as well as when $A$ is trivially graded. 

Thus, starting from a contractive homomorphism $f: K^{\gr}_0(R, [R])\rightarrow K^{\gr}_0(S, [S])$, of graded matricial $*$-algebras $R$ and $S$ over a graded $*$-field $A$ with enough unitaries, we present a specific formula for a graded $*$-homomorphism $\phi: R\rightarrow S$ such that $K_0^{\gr}(\phi)=f.$ We also show that if $f$ is unit-preserving, then $\phi$ is unital. This result, contained in Theorem \ref{fullness} is a graded, involutive generalization of an analogous non-graded and non-involutive result from \cite{Goodearl_book} and a non-involutive result from \cite{Roozbeh_Annalen}. Besides being more general, we emphasize that our proof of Theorem \ref{fullness} is {\em constructive} while the proofs of the  non-graded, non-involutive version from \cite{Goodearl_book} and the non-involutive version from \cite{Roozbeh_Annalen} (and \cite{Roozbeh_graded_ring_notes}) are existential. 

After introducing some notation, we briefly summarize the idea of the proof in the non-graded case. Our reasons for doing so are the following. The proof of Theorem \ref{fullness} involves formulas with cumbersome subscripts and superscripts. So, first we consider the non-graded case to gradually introduce the notation and formulas. This approach also highlights the specifics of the graded case. In addition, since the proofs of the analogous, non-involutive statements are not constructive, we include our constructive proof in the non-graded case too.  

Although rings in this section are assumed to be unital, the homomorphisms between them are not necessarily unit-preserving.

\subsection{Notation} We introduce some notation we extensively use in this section. 

If $A$ is any ring, we let $1_n$ denote the identity matrix in the matrix ring $\M_n(A)$  and $0_n$ the zero matrix in $\M_n(A).$ 
If $x\in M_n(A)$ and $y\in \M_m(A)$ we define their direct sum by
\[x\oplus y =
\left( 
\begin{array}{cc}
x & 0_{n\times m} \\
0_{m\times n} & y\\
\end{array}
\right)\in \M_{n+m}(A).\]

If $A$ is a $*$-ring and $k$ a positive integer, we define $*$-algebra homomorphisms $\phi_k^n:\M_n(A)\to \M_{nk}(A)$ for any positive integer $n$
as follows. For $x\in \M_n(A),$ we let $\phi^n_k (x)=x\otimes 1_{k}
\in \M_{kn}(A)$ where 
\[
(x_{ij})\otimes 1_{k}:=
\left( 
\begin{array}{cccc}
x_{11}1_k& x_{12}1_k& \ldots &x_{1n}1_k\\
x_{21}1_k& x_{22}1_k& \ldots & x_{2n}1_k\\
\vdots & \vdots &\ddots &\vdots\\
x_{n1}1_k& x_{n2}1_k& \ldots & x_{nn}1_k\\
\end{array}
\right)\;\;\mbox{ for }\;\;
x_{ij}1_k=
\left( 
\begin{array}{cccc}
x_{ij}& 0& \ldots & 0\\
0& x_{ij}& \ldots & 0\\
\vdots & \vdots &\ddots &\vdots\\
0& 0& \ldots & x_{ij}\\
\end{array}
\right)
\]
for all $i,j=1,\ldots, n.$ 
It is direct to check that $\phi^n_k$ is a unital $*$-algebra homomorphism. Note that $\phi^n_1$ is the identity map for any $n$.

If $R$ is a direct sum of rings $R_i, i\in I,$ we let $\pi_i$ denote the projection of $R$ onto $R_i$ and $\iota_i$ denote the inclusion of $R_i$ into $R.$
Thus, every element of $R$ can be represented as $a=\sum_{i\in I} \iota_i\pi_i(a).$

\subsection{The dimension formulas in the non-graded case}\label{subsection_non-graded} 

Consider matricial $*$-algebras 
\begin{center}
$R=\bigoplus_{i=1}^n \M_{p(i)}(A)\;\;$ and $\;\;S=\bigoplus_{j=1}^m \M_{q(j)}(A)$
\end{center} over a (non-graded) $*$-field $A.$ Let $e_{kl}^i$, $i=1, \ldots, n$,  $k,l=1,\ldots, p(i),$ be the elements of $R$ such that $\pi_i(e^i_{kl})$ are the standard matrix units of 
$\M_{p(i)}(A)$ for every $i=1,\ldots, n,$ and let $f_{kl}^j$,  $j=1, \ldots, m$,  $k,l=1,\ldots, q(j),$ be the elements of $S$ such that 
$\pi_j(f^j_{kl})$ are the standard matrix units of $\M_{q(j)}(A)$ for every $j=1,\ldots, m.$
 
Let $1_R$ and $1_S$ be the identity maps on $R$ and $S$ respectively. 
If $f: K_0(R)\rightarrow K_0(S)$ is a contractive homomorphism, let $$f([e_{11}^i])=\sum_{j=1}^m a_{ji}[f_{11}^j]$$ so that $(a_{ji}):\Z^n\to \Z^m$ is an $m\times n$ matrix which corresponds to the map $f$ under the isomorphisms $ K_0(R)\cong \Z^n$ and $K_0(S)\cong \Z^m.$ The fact that $f$ is order-preserving implies that $a_{ji}$ are {\em nonnegative} integers. The fact that $f([1_R])\leq [1_S]$ implies the  {\em dimension formulas}.
\footnote{In the case when $A$ is the field of complex numbers with the complex-conjugate involution, the dimension formulas hold for maps already on the algebra level not just on the $K_0$-group level (see e.g. \cite[page 75]{Davidson}).} 
\begin{equation}\label{non-graded_dimension}
\sum_{i=1}^n a_{ji}\,p(i)\leq q(j)
\end{equation}
for all $j=1,\ldots, m.$ Moreover, it can be shown that $\sum_{i=1}^n a_{ji}\,p(i)= q(j)$ if and only if $f([1_R])=[1_S],$ i.e. if $f$ is unit-preserving.

We define a $*$-map $\phi: R\to S$ corresponding to $f: K_0(R)\to K_0(S)$ as follows. Let $N_j$ denote the sum $\sum_{i=1}^n a_{ji}p(i)$ if $a_{ji}>0$ for at least one $i=1,\ldots, n$ and let $N_j=0$ otherwise. Then let 
\[
\phi(x)= \sum_{j=1}^m\iota_j\left(\left(\bigoplus_{i=1, a_{ji}>0}^n \pi_i(x)\otimes 1_{a_{ji}}\right)\oplus 0_{q(j)-N_j}\right)
\]
for any $x\in R.$ If  $q(j)-N_j=0,$ we consider $0_{q(j)-N_j}=\varnothing$ as a matrix of size $0\times 0.$ If $N_j=0,$ we consider the term  $\bigoplus_{i=1, a_{ji}>0}^n \pi_i(x)\otimes 1_{a_{ji}}$ to be a $0\times 0$ matrix. 

Note that $\pi_i(x)\otimes 1_{a_{ji}}$ is a $p(i)a_{ji}\times p(i)a_{ji}$ matrix for every $i$ and $j$ with $a_{ji}>0.$ The sum $$\bigoplus_{i=1, a_{ji}>0}^n \pi_i(x)\otimes  1_{a_{ji}}$$
is an $N_j\times N_j$ matrix and the sum $$\left(\bigoplus_{i=1, a_{ji}>0}^n \pi_i(x)\otimes  1_{a_{ji}}\right)\oplus 0_{q(j)-N_j}$$ is a  $q(j)\times q(j)$ matrix for any $j.$ The term $$\iota_j\left(\left(\bigoplus_{i=1, a_{ji}>0}^n \pi_i(x)\otimes  1_{a_{ji}}\right)\oplus 0_{q(j)-N_j}\right)$$ is an element of $S$
and so $\phi(x)$ is in $S$ for any $x\in R.$ The map  $\phi$ is a $*$-algebra homomorphism since the maps $\phi^{p(i)}_{a_{ji}}, \iota_j,$ and $\pi_j$ are  $*$-algebra homomorphisms. 

Consider the image $\phi(e^i_{11})$ of $e_{11}^i\in R$ in two cases: if $a_{ji}=0$ for all $j=1,\ldots, m$ and if $a_{ji}>0$ for some $j=1,\ldots, m.$
In the first case, $\phi(e_{11}^i)=0\in S$ by the definition of $\phi$ and $\sum_{j=1}^m a_{ij}[f_{11}^j]=0\in K_0(S)$ so that $[\phi(e_{11}^i)]=\sum_{j=1}^m a_{ij}[f_{11}^j].$

In the second case, assume that $a_{ji}>0$ for some $j=1,\ldots, m.$ Then $\pi_i(e^i_{11})\otimes  1_{a_{ji}} \in \M_{a_{ij}p(i)}(A)$
is a diagonal matrix with the first $a_{ij}$ entries on the diagonal being 1 and the rest being 0. 
Since $\pi_{i'}(e^i_{11})=0_{p(i')}$ for $i'\neq i$ we have that $E_{11}^{ij}:=\left(\bigoplus_{i'=1, a_{ji'}>0}^n \pi_{i'}(e^i_{11})\otimes  1_{a_{ji'}}\right)\oplus 0_{q(j)-N_j},$
is a diagonal $q(j)\times q(j)$ matrix with exactly $a_{ji}$ nonzero entries of value 1 on the diagonal. Thus, 
\[[\iota_j(E_{11}^{ij})]=  a_{ij}[f_{11}^j]\] in $K_0(S).$ 

If $a_{j'i}=0$ for some $j'=1,\ldots, m,$ we let $E_{11}^{ij'}$ be the $0\times 0$ matrix so that we have   
$\phi(e_{11}^i)= \sum_{j=1}^m\iota_j(E_{11}^{ij})$ by the definition of $\phi$ and the assumption that $a_{ji}> 0$ for at least one $j.$ Thus we have that 
\[K_0(\phi)([e^i_{11}])=[\phi(e_{11}^i)]= [\sum_{j=1}^m\iota_j(E^{ij}_{11})]=\sum_{j=1, a_{ji}>0}^m [\iota_j(E^{ij}_{11})]=
\sum_{j=1, a_{ji}>0}^m a_{ji}[f^j_{11}]=
\sum_{j=1}^m a_{ji}[f^j_{11}]=f([e_{11}^i]).\]
Note that we are using the assumption that  $a_{ji}> 0$ for at least one $j$ in the equality $\sum_{j=1, a_{ji}>0}^m a_{ji}[f^j_{11}]=
\sum_{j=1}^m a_{ji}[f^j_{11}].$ The equality $K_0(\phi)([e^i_{11}])=f([e_{11}^i])$ for every $i=1,\ldots, n$ implies that $K_0(\phi)=f.$ 

One can also show that $\phi$ maps $1_R$ onto $1_S$ if $f$ is unit-preserving (which happens exactly when $N_j=q(j)$ for every $j=1,\ldots,m$) but we leave the proof of this statement for the graded case. 

We illustrate this construction with an example.

\begin{example} If $R=\M_2(A)\oplus A$ and $S=\M_5(A)\oplus \M_4(A),$  any possible contractive map on $K_0$-groups is a $2\times 2$ matrix 
$\left(\begin{array}{cc}
a_{11} & a_{12}\\
a_{21} & a_{22}\end{array}\right)$ with nonnegative entries satisfying the dimension formulas  
\begin{center}
$2a_{11}+1a_{12}\leq 5\;\;$ and $\;\;2a_{21}+1a_{22}\leq 4.$
\end{center}
Let us consider one of them, for example, 
$f=\left(\begin{array}{cc}
2 & 1\\
0 & 3\end{array}\right).$ This map will induce a non-unital map $\phi$ since the inequality in the second dimension formula is strict. The induced map $\phi$ is given by 
\[(\left(\begin{array}{cc}
a & b\\
c & d\end{array}\right), e)\mapsto (
\left(\begin{array}{cccc|c|}
a & 0 & b & 0 & 0\\
0 & a & 0 & b & 0 \\ 
c & 0 & d & 0 & 0\\
0 & c & 0 & d & 0\\ \hline 
0 & 0 & 0 & 0 & e \\ \hline
\end{array}\right),
\left(\begin{array}{|ccc|c} \hline
e & 0 & 0 & 0\\
0 & e & 0 & 0 \\ 
0 & 0 & e & 0\\\hline
0 & 0 & 0 & 0\\ 
\end{array}\right)).\]\label{example_nongraded}
\end{example}

\subsection{The dimension formulas in the graded case}\label{subsection_graded_fullness} 
Let $A$ be a $\Gamma$-graded $*$-field now. Consider 
a graded matrix $*$-ring $\M_n(A)(\gamma_1,\dots,\gamma_n)$, where $\gamma_i \in \Gamma$, for $i=1,\ldots, n$. 
Since $A$ is a graded field, $\Gamma_A$ is a subgroup of $\Gamma$. The shifts $\gamma_1, \ldots, \gamma_n$ can be partitioned such that the elements of the same partition part belong to the same coset of the quotient group $\Gamma/\Gamma_A$ and the elements from different partition parts belong to different cosets. Let $\gamma_{l1},\dots,\gamma_{lr_l}$ be the $l$-th partition part for $l=1, \ldots, k$ where $k$ is the number of parts and $r_l$ is the number of elements in the $l$-th partition part. Thus, we partition $(\gamma_1,\dots,\gamma_n)$ as 
\begin{equation}\label{rearanging}
(\gamma_{11},\dots,\gamma_{1r_1}, \gamma_{21},\dots,\gamma_{2r_2}, \;\;\ldots,\;\; \gamma_{k1},\dots,\gamma_{kr_k })
\end{equation}
Note that $\sum_{l=1}^k r_l=n.$ Moreover, we can assume that $\gamma_1=0$ and, thus, to have that $\gamma_{11}=0$ by the first part of Proposition \ref{permutation_of_components}.

By Lemma \ref{matrix_units_lemma} and formula (\ref{matrix_units_identity}), we have that 
$[e_{ii}]=[e_{jj}]$ if and only if $\gamma_i$ and $\gamma_j$ are in the same $\Gamma/\Gamma_A$ coset. Otherwise we have that $[e_{ii}]=(\gamma_j-\gamma_i)[e_{jj}]\neq[e_{jj}].$
Under the assumption that $\gamma_1=0,$ we have that 
\begin{equation}\label{identity_equation}
[1_n]=\sum_{i=1}^n[e_{ii}]= \sum_{i=1}^n -\gamma_i[e_{11}]=\sum_{l=1}^k\sum_{l'=1}^{r_l} -\gamma_{ll'}[e_{11}]=
\sum_{l=1}^k -r_l \gamma_{l1}[e_{11}].
\end{equation} 

Consider graded matricial $*$-algebras $R$ and $S$ over $A$ now. Let 
\[R=\bigoplus_{i=1}^n\M_{p(i)}(A)(\gamma^i_1,\dots,\gamma^i_{p(i)})\;\;\;\mbox{ and }\;\;\;S=\bigoplus_{j=1}^m\M_{q(j)}(A)(\delta^j_1,\dots,\delta^j_{q(j)}).\]
Analogously to formula (\ref{rearanging}) above, we can partition the shifts $(\delta^j_1,\dots,\delta^j_{q(j)})$ in the $j$-th component of $S$ and represent them 
as follows.  
\[(\delta^j_{11},\dots,\delta^j_{1s^j_1}, \delta^j_{21},\dots,\delta^j_{2s^j_2}, \;\;\ldots, \;\; \delta^j_{l^j1},\dots,\delta^j_{l^j s^j_{l^j}})\]
Note that for every $j=1,\ldots, m$ we have that 
$\sum_{j'=1}^{l^j} s^j_{j'}=q(j).$ Also, using the first part of Proposition \ref{permutation_of_components}, we can assume that $\gamma^i_1=0=\delta^j_1$ for every $i=1,\ldots,n$ and $j=1,\ldots,m.$

If we let $e_{kl}^i\in R$, $i=1, \ldots, n$,  $k,l=1,\ldots, p(i),$ be such that $\pi_i(e_{kl}^i)$ are the standard matrix units in $\pi_i(R),$ and $f_{kl}^j\in S$,  $j=1, \ldots, m$,  $k,l=1,\ldots, q(i),$ such that $\pi_j(f^j_{kl})$ are the standard matrix units in $\pi_j(S),$ then we have that $\{[e_{11}^i]\;|\; i=1,\ldots, n\}$ generates $K^{\gr}_0(R)$ and  $\{[f_{11}^j]\;|\; j=1,\ldots, m\}$ generates $K^{\gr}_0(S)$ by Proposition \ref{equivalences_phi_and_psi} and Lemma \ref{matrix_units_lemma}. 

Let $f: K^{\gr}_0(R)\rightarrow K^{\gr}_0(S)$ be a contractive $\Z[\Gamma]$-module homomorphism. Our ultimate goal is to find a graded $*$-algebra homomorphism $\phi$ such that $K_0^{\gr}(\phi)=f.$
Let 
\[f([e_{11}^i])=\sum_{j=1}^m\sum_{t=1}^{k_{ji}} a_{jit}\alpha_{jit}[f_{11}^j]\]
for some $a_{jit}\in \Z$ and $\alpha_{jit}\in \Gamma.$ Since $f$ is order-preserving, we have that $\sum_{j=1}^m\sum_{t=1}^{k_{ji}} a_{jit}\alpha_{jit}[f_{11}^j]\geq 0.$   
Under the isomorphisms $ K^{\gr}_0(R)\cong \Z[\Gamma/\Gamma_A]^n$ and $K^{\gr}_0(S)\cong \Z[\Gamma/\Gamma_A]^m$ from part (3) of Proposition \ref{graded_division_ring}, 
this last relation corresponds to $\sum_{j=1}^m\sum_{t=1}^{k_{ji}} a_{jit}(\alpha_{jit}+\Gamma_A)\geq 0.$ By  part (3) of Proposition \ref{graded_division_ring} again, this implies that $a_{jit}\geq 0$ for all $j=1,\ldots, m,$ $i=1,\ldots, n$ and $t=1, \ldots, k_{ji}.$ 

If we let $\ol a_{ji}:=\sum_{t=1}^{k_{ji}} a_{jit}\alpha_{jit},$ then the $m\times n$ matrix $(\ol a_{ji}):\Z[\Gamma/\Gamma_A]^n\to \Z[\Gamma/\Gamma_A]^m$ corresponds to the map $f$ under the the isomorphisms $ K^{\gr}_0(R)\cong \Z[\Gamma/\Gamma_A]^n$ and $K^{\gr}_0(S)\cong \Z[\Gamma/\Gamma_A]^m.$ If $f$ is the zero map, we can trivially take $\phi$ to be the zero map as well, so let us assume that the map $f$ is nonzero. In this case, at least one $\ol a_{ji}$ is nonzero
so we can write $\ol a_{ji}=\sum_{t=1}^{k_{ji}} a_{jit}\alpha_{jit}$ with  {\em positive} $a_{jit}$ for every $t=1,\ldots, k_{ji}$ or $\ol a_{ji}=0.$ If $\ol a_{ji=0},$ we define $k_{ji}$ to be 0.

With this convention, we have that 
\[f([e_{11}^i])= \sum_{j=1, \ol a_{ji}\neq 0}^m \ol a_{ji}[f^j_{11}]= \sum_{j=1, k_{ji}>0}^m \sum_{t=1}^{k_{ji}} a_{jit}\alpha_{jit}[f_{11}^j]\]
if $k_{ji}>0$ for some $j=1,\ldots, m$ and $f([e_{11}^i])=0$ otherwise. We also have that 
\[\sum_{i=1}^n f([e_{11}^i])=\sum_{i=1}^n\sum_{j=1, k_{ji}>0}^m \sum_{t=1}^{k_{ji}} a_{jit}\alpha_{jit}[f_{11}^j]= \sum_{j=1}^m \sum_{i=1, k_{ji}>0}^n\sum_{t=1}^{k_{ji}} a_{jit}\alpha_{jit}[f_{11}^j]\] where the first equality holds by the assumption that $k_{ji}$ is positive for at least one value of $i$ and $j$. 

Using (\ref{identity_equation}) and the assumption that $f$ is contractive, we have that 
\begin{multline*}
f([1_R])=f\left([\sum_{i=1}^n\sum_{k=1}^{p(i)} e^i_{kk}]\right) = f\left([\sum_{i=1}^n\sum_{k=1}^{p(i)} -\gamma_{k}^ie^i_{11}]\right) = \\
\sum_{i=1}^n\sum_{k=1}^{p(i)} -\gamma_{k}^if\left([e^i_{11}]\right)=
\sum_{j=1}^m \sum_{i=1, k_{ji}>0}^n\sum_{t=1}^{k_{ji}}\sum_{k=1}^{p(i)} -\gamma_{k}^i a_{jit}\alpha_{jit}[f_{11}^j]=\\
\sum_{j=1}^m\sum_{i=1, k_{ji}>0}^n \sum_{t=1}^{k_{ji}} \sum_{k=1}^{p(i)}  a_{jit}(-\gamma_{k}^i+\alpha_{jit})[f_{11}^j]\leq  
\sum_{j=1}^m\sum_{j'=1}^{l^j} -s^j_{j'}\delta_{j'1}^j [f^j_{11}]= [1_S].
\end{multline*}
Since we assumed that $\delta^j_1=0$ for all $j=1,\ldots, m,$ the class $[f^j_{11}]$ corresponds to the coset $\Gamma_A$ under the graded $*$-isomorphism $S\cong_{\gr} \Z[\Gamma/\Gamma_A]^m.$ Thus, the last relation above is equivalent to
\[\sum_{j=1}^m \sum_{i=1, k_{ji}>0}^n\sum_{t=1}^{k_{ji}}\sum_{k=1}^{p(i)} a_{jit}(-\gamma_{k}^i+ \alpha_{jit}+\Gamma_A)  \leq  
\sum_{j=1}^m\sum_{j'=1}^{l^j} s^j_{j'} (-\delta_{j'1}^j+\Gamma_A)
\]
and implies that
\begin{equation}\label{predimension} 
\sum_{i=1, k_{ji}>0}^n\;  \sum_{t=1}^{k_{ji}}\sum_{k=1}^{p(i)} a_{jit}(-\gamma_{k}^i+ \alpha_{jit}+\Gamma_A)\leq \sum_{j'=1}^{l^j} s^j_{j'} (-\delta_{j'1}^j+\Gamma_A)
\end{equation}
for all $j=1,\ldots, m$ for which $k_{ji}>0$ for at least one $i=1,\ldots, n.$ We refer to this family of formulas as the {\em pre-dimension formulas}. 

The pre-dimension formulas imply two sets of relations. First, 
note that any coset on the left-hand side with multiplicity $a_{jit}$ should appear on the right-hand side as well at least $a_{jit}$ times, i.e. for any $i=1,\ldots,n,$ with $k_{ji}>0$ and any $t=1, \ldots, k_{ji}$, and $k=1, \ldots, p(i),$ there is $j'=1, \ldots, l^j$ such that 
\[
a_{jit}(-\gamma_{k}^i+ \alpha_{jit}+\Gamma_A)=a_{jit}(-\delta_{j'1}^j+\Gamma_A)
\]
Since $a_{jit}>0,$ we have that 
\begin{equation}\label{coset_equation}
-\gamma_{k}^i+ \alpha_{jit}+\Gamma_A=-\delta_{j'1}^j+\Gamma_A. 
\end{equation} 
in this case. We refer to the formulas above as the {\em coset equations}. 

For any $j=1,\ldots, m,$ let $S_j$ be the set 
\[S_j=\{(i,t,s,k)\,|\, i=1, \ldots, n,\mbox{ with }k_{ji}>0,\; t=1,\ldots, k_{ji}, s=1, \ldots, a_{jit}, k=1,\ldots, p(i)\}.\]
Note that this set is empty exactly when $k_{ji}=0$ for all $i=1,\ldots, n$ which is equivalent to $\ol a_{ji}=0$ for all $i=1,\ldots, n.$ 
If $N_j$ is the cardinality of $S_j,$ then 
\[N_{j}=\sum_{i=1, k_{ji}>0}^n\;\sum_{t=1}^{k_{ji}}a_{jit}\,p(i).\] Thus $N_j=0$ if and only if $k_{ji}=0$ for all $i=1,\ldots, n.$  

For any $j=1,\ldots, m$ and any $j'=1,\ldots, l^j,$ we define
\[S_{jj'}:=\{(i, t, s, k)\in S_j\; |\;-\gamma_{k}^i+ \alpha_{jit}+\Gamma_A=-\delta_{j'1}^j+\Gamma_A\}\]
and let $N_{jj'}$ denote the cardinality of $S_{jj'}.$ 

If $S_j$ is nonempty, there is a coset equation holding for any $(i,t,s,k)\in S_j$ and so $S_j=\bigcup_{j'=1}^{l^j}S_{jj'}.$ If $S_j$ is empty, then all the sets $S_{jj'}, j'=1,\ldots, l^j$ are empty too and we can still write $S_j=\bigcup_{j'=1}^{l^j}S_{jj'}.$ 

In addition, since $\delta_{j'1}^j+\Gamma_A\neq \delta_{{j''}1}^j+\Gamma_A$ for every $j'\neq j''$ when $j, j''=1,\ldots, l^j,$ we have that
$S_{jj'}$ and $S_{jj''}$ are disjoint if $j'\neq j''.$ As a result, we have that $N_j=\sum_{j'=1}^{l^j}N_{jj'}.$ 

If $S_{jj'}$ is nonempty, each pre-dimension formula (\ref{predimension}) implies that the number of $(i,t,s,k)\in S_{jj'}$ cannot be larger than $s^j_{j'}.$ Thus, we have that
\begin{equation}\label{dimension_formulas}
N_{jj'}\leq  s^j_{j'}
\end{equation} 
for every $j=1,\ldots,m,$ and every $j'=1,\ldots, l^j.$ We refer to this set of formulas as the {\em dimension formulas}. 

Adding the $(j, j')$-th dimension formulas for all $j'=1,\ldots,l^j,$ we obtain that 
\[
\sum_{j'=1}^{l^j} N_{jj'}=N_{j}= \sum_{i=1, k_{ji}>0}^n\sum_{t=1}^{k_{ji}}a_{jit}\,p(i) \leq  \sum_{j'=1}^{l^j} s^j_{j'}= q(j)  
\]
for any $j=1,\ldots,m.$   

The dimension formulas enable us to define an injective map $\Pi^j:S_j\to \{1, \ldots, q(j)\}$ with the property that 
\begin{equation}\label{map_Pi}
-\gamma_{k}^i+ \alpha_{jit}+\Gamma_A=-\delta^j_{\Pi^j(i,t,s,k)}+\Gamma_A
\end{equation}
for all $(i,t,s,k)\in S_j.$ Thus, the image of $\Pi^j$ has cardinality $N_j.$ The case when $N_j=q(j)$ corresponds exactly to the case when $\Pi^j$ is a bijection. 

\begin{remark}
If $\Gamma$ is the trivial group, we claim that both the pre-dimension formulas (\ref{predimension}) and the dimension formulas (\ref{dimension_formulas}) reduce to the non-graded dimension formulas (\ref{non-graded_dimension}). Indeed, if $\Gamma$ is trivial, all the shifts are zero and $l^j=1,$ $s^j_1=q(j),$ $N_j=N_{j1},$ and
$k_{ji}$ is either 0 or 1. The case $k_{ji}=1$ corresponds exactly to $a_{ji}>0.$ Thus, 
$S_j=S_{j1}=\{(i, 1, s, k)\; |\; i\in \{1,\ldots,n \}, a_{ji}\neq 0,\; k\in \{1, \ldots, p(i)\}, s=1,\ldots, a_{ji}\}$ is nonempty exactly when $a_{ji}>0$ for some $i=1,\ldots, n.$ In this case  the pre-dimension formulas (\ref{predimension}) become
\[\sum_{i=1, a_{ji}\neq 0}^n\;\sum_{k=1}^{p(i)} a_{ji}=\sum_{i=1}^n\sum_{k=1}^{p(i)} a_{ji} = \sum_{i=1}^n a_{ji}\,p(i)\leq \sum_{j'=1}^{1} s^j_{j'}=s^j_{1} =q(j)\]
for every $j=1,\ldots, m,$ and the dimension formulas (\ref{dimension_formulas})
become 
\[N_{j1}=N_j=\sum_{i=1}^n\sum_{k=1}^{p(i)} a_{ji}= \sum_{i=1}^n a_{ji}\,p(i) \leq  s^j_{1}=q(j).\]
\end{remark}

\begin{remark}
We also note that the dimension formulas and the coset equations imply the pre-dimension formulas. We shall not use this fact in Theorem \ref{fullness}, but we still illustrate why this fact holds. So, let us assume that the dimension formulas and the coset equations hold. For every $j=1,\ldots, m,$  $j'=1,\ldots, l^j,$ we have that
\[\sum_{(i,t,s,k)\in S_{jj'}}-\gamma_{k}^i+ \alpha_{jit}+\Gamma_A= N_{jj'}(-\delta_{j'1}^j+\Gamma_A)\leq s^j_{j'}(-\delta_{j'1}^j+\Gamma_A).\] 
Adding these inequalities for every $j'=1,\ldots, l^j$ produces 
\[\sum_{j'=1}^{l^j}\sum_{(i,t,s,k)\in S_{jj'}}-\gamma_{k}^i+ \alpha_{jit}+\Gamma_A \leq \sum_{j'=1}^{l^j} s^j_{j'}(-\delta_{j'1}^j+\Gamma_A).\]
If $N_j>0,$ which is exactly the case under which we have the $j$-th pre-dimension equation, we have that 
\[\sum_{j'=1}^{l^j}\sum_{(i,t,s,k)\in S_{jj'}}-\gamma_{k}^i+ \alpha_{jit}+\Gamma_A=\sum_{i=1, k_{ji}>0}^n\;\sum_{t=1}^{k_{ji}}\sum_{k=1}^{p(i)} a_{jit}(-\gamma_{k}^i+ \alpha_{jit}+\Gamma_A)\] and so the pre-dimension formulas hold. 
 
This shows that we have the following equivalence.

\medskip
\begin{center}
\begin{tabular}{ccc}
\begin{tabular}{|c|}\hline $\;\;$
pre-dimension formulas $\;\;$ \\
\hline
\end{tabular} & $\Longleftrightarrow$ & 
\begin{tabular}{|c|}\hline
$\;\;$ coset equations $\;$ + $\;$ dimension formulas $\;\;$\\
\hline
\end{tabular}
\end{tabular}
\end{center}
\medskip

\noindent This observation also implies that the equality holds in all the pre-dimension formulas (\ref{predimension}) if and only if the equality holds in all dimension formulas (\ref{dimension_formulas}). 
\end{remark}

We summarize our findings in the following proposition. 

\begin{proposition}
Let  $f: K^{\gr}_0(R)\rightarrow K^{\gr}_0(S)$ be a $\Z[\Gamma]$-module homomorphism for graded matricial $*$-algebras $R$ and $S$ over a graded $*$-field $A.$ 
\begin{enumerate}
\item[{\em (1)}] If $f$ is order-preserving, then $a_{jit}$ are nonnegative integers for all $j=1,\ldots, m,$ $i=1,\ldots, n,$ and $t=1,\ldots, k_{ji}.$

\item[{\em (2)}] If $f$ is contractive, then the formulas (\ref{predimension}) (equivalently (\ref{coset_equation}) and  (\ref{dimension_formulas})) hold. 
\end{enumerate}
In addition, if $f$ is contractive and unit-preserving, then the equality holds in formulas (\ref{predimension}) (equivalently (\ref{dimension_formulas})).
\end{proposition}
\begin{remark}
It can be shown that the converse of the statements (1) and (2) also holds. We do not prove the converses since we will not need them in the proof of the main result. The converse of the last sentence also clearly holds.  
\end{remark}

We adapt the definition of the maps $\phi^n_{k}$ from section \ref{subsection_non-graded} to the graded setting now. Note that for any positive integers $k$ and $n$ and $\ol\gamma\in \Gamma^n$ the definition of the $*$-algebra homomorphisms $\phi_k^n$ from section \ref{subsection_non-graded} becomes
\begin{equation}\label{action_graded_k}
\phi_{k}^{n, \ol \gamma}:\M_n(A)(\gamma_1,\ldots,\gamma_n)\to \M_{nk}(A)(\gamma_1,\dots,\gamma_1, \gamma_2,\dots,\gamma_2, \;\;\ldots,\;\; \gamma_n,\dots,\gamma_n)\;\mbox{ with }\;\phi_{k}^{n, \ol \gamma}:x\mapsto x\otimes 1_{k}
\end{equation}
where the matrix $x\otimes 1_{k}$ is defined in the same way as in the non-graded case in section \ref{subsection_non-graded}. 
Just as in the non-graded case, the map $\phi^{n, \ol \gamma}_k$ is a unital $*$-algebra homomorphism. Moreover, $\phi^{n, \ol \gamma}_k$ is a {\em graded} homomorphism. Indeed, if $x$ is a matrix of degree $\delta,$ then the elements of the entire $(i, j)$-th block of $x\otimes 1_{k}$ are in the component $A_{\delta+\gamma_j-\gamma_i}$ for any $i, j=1,\ldots, n.$ Thus, the matrix $x\otimes 1_{k}$ has degree $\delta.$   

Let $\delta\in\Gamma.$  We define a graded $*$-algebra homomorphisms $\phi^{n, \ol\gamma}_{\delta}$ as follows.
\begin{equation}\label{action_graded_delta}
\phi_{\delta}^{n, \ol \gamma}:\M_n(A)(\gamma_1,\ldots,\gamma_n)\to \M_n(A)(\gamma_1-\delta,\ldots,\gamma_n-\delta)\;\mbox{ is given by }\;
\phi_{\delta}^{n, \ol\gamma}:\; (x_{ij})\mapsto (x_{ij}).
\end{equation}
If $x_{ij}\in A_{\alpha+\gamma_j-\gamma_i},$ then it is considered as an element of $A_{\alpha+\gamma_j-\delta-(\gamma_i-\delta)}$ in the image of $\phi^{n,\ol\gamma}_\delta.$ Thus, this map is clearly a graded $*$-algebra homomorphism. 

For positive integers $k$ and $n,$ $\ol\gamma\in \Gamma^n$ $\delta\in \Gamma,$ and $x\in \M_n(A)(\ol\gamma),$ we denote the image of $x$ under the composition $\phi^{n, \ol\gamma}_k\phi^{n,\ol\gamma}_\delta$ as follows. 
\begin{equation}\label{action_graded}
\phi^{n, \ol\gamma}_k\phi^{n,\ol\gamma}_\delta(x)=x\otimes 1_{k\delta}. 
\end{equation}

The finishing touch we need for the proof of the main theorem of this section is an assumption on the grading of the $*$-field $A$. Namely, our proof requires us to be able to use the second part of Proposition \ref{permutation_of_components} for {\em any} element of $\Gamma_A.$ In other words, we need the assumption that any nonzero component $A_\gamma$ contains a unitary element $a$ i.e. an element $a$ such that $aa^*=a^*a=1.$ This motivates the following definition.
\begin{definition}
A  $\Gamma$-graded $*$-ring $A$ has {\em enough unitaries} if any nonzero component $A_\gamma$ contains a unitary element $a$ i.e. an element $a$ such that $aa^*=a^*a=1.$
\end{definition}

Note that any trivially $\Gamma$-graded $*$-field has enough unitaries since $1\in A_0$ is a unitary. In particular, any $*$-field graded by the trivial group has enough unitaries. Also if $A$ is a $*$-field, the $*$-field $R=A[x^{n}, x^{-n}],$ $\Z$-graded by  $R_{m}=Ax^{m}$ if $m\in n\Z$ and $R_m=0$ otherwise, has enough unitaries. Indeed, since  $\Gamma_R=n\Z,$ we have that $x^{kn}$ is a unitary element in $R_{nk}$ for any $k\in \Z.$ More generally, the $\Gamma$-graded group ring $A[\Gamma]$ with $A[\Gamma]_\gamma=\{k\gamma\, |\, k\in A\}$ for any $\gamma\in \Gamma$ equipped with the standard involution ($(k\gamma)^*=k^*\gamma^{-1}$) has enough unitaries since $\gamma\in\Gamma$ is a unitary element in $A[\Gamma]_\gamma.$

\begin{theorem}\label{fullness}
Let $A$ be a $\Gamma$-graded $*$-field with enough unitaries and let $R$ and $S$ be graded matricial $*$-algebras over $A$. 
If $f:K^{\gr}_0(R)\rightarrow K^{\gr}_0(S)$ is a contractive $\Z[\Gamma]$-module homomorphism, then there is a graded $A$-algebra $*$-homomorphism $\phi:R\rightarrow S$ such that $K^{\gr}_0(\phi)=f$. Furthermore, if $f$ is unit-preserving, then $\phi$ is unital. 
\end{theorem}
\begin{proof}
Let us keep our existing notation for $R$ and $S$ and the assumptions that $\gamma_1^i=\delta^j_1=0$ for all $i=1,\ldots, n$ and $j=1,\ldots, m.$ 
Since we can take $\phi$ to be the zero map if $f$ is the zero map, let us assume that $f$ is nonzero and represent it using positive integers $a_{jit}$ and  $\alpha_{jit}\in\Gamma$ as in our previous discussion. Thus, the pre-dimension and dimension formulas (\ref{predimension}) and (\ref{dimension_formulas}) as well as the coset equations (\ref{coset_equation}) hold. We also keep the definitions of the sets $S_{jj'}$ and $S_j,$ their cardinalities $N_{jj'}$ and $N_j$ as well as the map $\Pi^j: S_{j}\to \{1, \ldots, q(j)\}$ for all $j=1,\ldots, m$ 
with property (\ref{map_Pi}). For any $(i,t,s,k)\in S_j,$ each equation of the form  
$-\gamma_{k}^i+ \alpha_{jit}+\Gamma_A=-\delta^j_{\Pi^j(i, t, s, k)}+\Gamma_A$ implies that 
$\delta^j_{\Pi^j(i,t,s,k)}-\gamma_{k}^i+\alpha_{jit}\in \Gamma_A.$ Since $A$ has enough unitaries, 
there is $\varepsilon^j_{\Pi^j(i, t, s, k)}\in\Gamma_A$ such that there is a unitary element in $A_{\varepsilon^j_{\Pi^j(i, t, s, k)}}$ and such that 
\[
-\gamma_{k}^i+ \alpha_{jit}=-\delta^j_{\Pi^j(i, t, s, k)}+\varepsilon^j_{\Pi^j(i, t, s, k)}.
\]

For every $j=1,\ldots, m$ with $N_j>0,$ let $\M_{N_j}(A)(\overline{(\gamma^i_k-\alpha_{jit})^{a_{jit}}})$ shorten the notation for the graded matrix algebra of
$N_{j}\times N_{j}$ matrices with shifts 
\begin{equation}\label{shifts}
\begin{array}{l}
(\gamma^1_1-\alpha_{j11})^{a_{j11}}, \ldots, (\gamma^1_{p(1)}-\alpha_{j11})^{a_{j11}},\;\;\; \ldots,\;\;\;
(\gamma^1_1-\alpha_{j1k_{1j}})^{a_{j1k_{1j}}}, \ldots, (\gamma^1_{p(1)}-\alpha_{j1k_{1j}})^{a_{j1k_{1j}}},\\ 
\ldots\\
(\gamma^n_1-\alpha_{jn1})^{a_{jn1}}, \ldots, (\gamma^n_{p(n)}-\alpha_{jn1})^{a_{jn1}},\;\;\; \ldots,\;\;\;
(\gamma^n_1-\alpha_{jn1})^{a_{jnk_{nj}}}, \ldots, (\gamma^n_{p(n)}-\alpha_{jnk_{nj}})^{a_{jnk_{nj}}}\\
\end{array}
\end{equation}
where $(\gamma^i_k-\alpha_{jit})^{a_{jit}}$ represent the term $\gamma^i_k-\alpha_{jit}$ listed $a_{jit}$ times for every $i=1,\ldots, n,$ $t=1, \ldots, k_{ji},$ and $k=1, \ldots, p(i).$  
Analogously to this convention, we let $\M_{N_j}(A)(\overline{\delta^j_{\Pi^j(i, t, s, k)}-\varepsilon^j_{\Pi^j(i, t, s, k)}})$ and $\M_{N_j}(A)(\overline{\delta^j_{\Pi^j(i, t, s, k)}})$ shorten the notation for the graded matrix algebra with shifts listed in the order which corresponds the order of shifts in (\ref{shifts}) for any $j=1, \ldots, m.$
    
Since the cardinality of the image of the map $\Pi^j$ is $N_j,$ let  $\sigma^j$ be a bijection of $\{1, \ldots, N_j\}$ and the image of $\Pi^j$ so that for every $l\in \{1, \ldots, N_j\}$ we have that $\sigma^j(l)=\Pi^j(i, t, s, k)$ for a unique element $(i, t, s, k)$ of $S_j,$ and, as a result,
\[-\gamma_{k}^i+ \alpha_{jit}=-\delta^j_{\Pi^j(i, t, s, k)}+\varepsilon^j_{\Pi^j(i, t, s, k)}=-\delta^j_{\sigma^j(l)}+\varepsilon^j_{\sigma^j(l)}.\]
In the case when $N_j$ is strictly less than $q(j),$ the complement of the image of $\Pi^j$ has $q(j)-N_j$ elements. Thus, using any bijective mapping of $\{N_j+1, \ldots, q(j)\}$ onto the complement of the image of $\Pi^j,$ we can extend $\sigma^j$ to a permutation of the set $\{1, \ldots, q(j)\}$ which we continue to call $\sigma^j.$ Let $\rho^j$ denote the inverse of $\sigma^j.$

Let $i_j(\M_{N_j}(A)(\delta^j_{\sigma^j(1)}, \ldots, \delta^j_{\sigma^j(N_j)}))$ denote the graded isomorphic copy of $\M_{N_j}(A)(\delta^j_{\sigma^j(1)}, \ldots, \delta^j_{\sigma^j(N_j)})$ in $\M_{q(j)}(A)(\delta^j_{\sigma^j(1)}, \ldots, \delta^j_{\sigma^j(q(j))}).$ 
By the definition of $\sigma^j$ and $\rho^j$ and the first part of Proposition \ref{permutation_of_components}, if $P_{\rho^j}$ denotes the permutation matrix which corresponds to the permutation $\rho^j,$ the conjugation by $P_{\rho^j}$ defines a graded $*$-isomorphism
$\M_{q(j)}(A)(\delta^j_{\sigma^j(1)}, \ldots, \delta^j_{\sigma^j(q(j))})\to \M_{q(j)}(A)(\delta^j_1, \ldots, \delta^j_{q(j)})$
such that its restriction on  $i_j(\M_{N_j}(A)(\delta^j_{\sigma^j(1)}, \ldots, \delta^j_{\sigma^j(N_j)}))$ is a graded $*$-isomorphism 
\[i_j(\M_{N_j}(A)(\delta^j_{\sigma^j(1)}, \ldots, \delta^j_{\sigma^j(N_j)}))\to \M_{N_j}(A)(\delta^j_1, \ldots, \delta^j_{N_j})\oplus 0_{q(j)-N_j}.\]
In the case when $N_j=q(j),$ we consider $0_{q(j)-N_j}$ to be a matrix $\varnothing$ of size $0\times 0.$

By the second part of  Proposition \ref{permutation_of_components}, there is a diagonal $N_j\times N_j$ matrix $D_j$ such that conjugating with $D_j$ produces a graded  
$*$-isomorphism 
$$\M_{N_j}(A)(\delta^j_{\sigma^j(1)}-\varepsilon^j_{\sigma^j(1)}, \ldots, \delta^j_{\sigma^j(N_j)}-\varepsilon^j_{\sigma^j(N_j)})\cong_{\gr} \M_{N_j}(A)(\delta^j_{\sigma^j(1)}, \ldots, \delta^j_{\sigma^j(N_j)}).$$ Thus, we have the following graded $*$-isomorphism. 
\begin{multline*}
\Phi^j: \M_{N_j}(A)(\overline{(\gamma^i_k-\alpha_{jit})^{a_{jit}}})= \M_{N_j}(A)(\delta^j_{\sigma^j(1)}-\varepsilon^j_{\sigma^j(1)}, \ldots, \delta^j_{\sigma^j(N_j)}-\varepsilon^j_{\sigma^j(N_j)})\cong_{\gr}\\ \M_{N_j}(A)(\delta^j_{\sigma^j(1)}, \ldots, \delta^j_{\sigma^j(N_j)})
\cong_{\gr} i_j(\M_{N_j}(A)(\delta^j_{\sigma^j(1)}, \ldots, \delta^j_{\sigma^j(N_j)}))
\cong_{\gr}\\
\M_{N_j}(A)(\delta^j_1, \ldots, \delta^j_{N_j})\oplus 0_{q(j)-N_j}\subseteq  \M_{q(j)}(A)(\delta^j_1, \ldots, \delta^j_{q(j)})
\end{multline*}
for any $j=1,\ldots,m.$  

If $N_j=0$ for some $j=1,\ldots, m,$ we consider $\M_{N_j}(A)(\overline{(\gamma^i_k-\alpha_{jit})^{a_{jit}}})$ to be the $0\times 0$ matrix and define $\Phi^j(\varnothing)$ to be $0\in \M_{q(j)}(A)(\delta^j_1,\ldots, \delta^j_{q(j)}).$

We are ready to define the graded $*$-algebra map $\phi: R\to S$ which corresponds to the map $f$ now. 
We define $\phi$ using the definitions of the maps in (\ref{action_graded_k}), (\ref{action_graded_delta}), and (\ref{action_graded}) and the map $\Phi^j$ as follows: 
\[
\phi(x) =  \sum_{j=1}^m\iota_j\Phi^j\left(\bigoplus_{i=1, k_{ji}>0}^{n}\;\bigoplus_{t=1}^{k_{ji}} \pi_i(x)\otimes 1_{a_{jit}\alpha_{jit}}\right)
\]
for any $x\in R.$ Note that the condition $k_{ji}>0$ ensures that the term  $\pi_i(x)\otimes 1_{a_{jit}\alpha_{jit}}$ is well-defined for any $t=1, \ldots, k_{ji}.$

First, we demonstrate that $\phi(x)$ is indeed an element of $S$ for any $x\in R.$ For $x\in R,$ we have that $\pi_i(x)\in \M_{p(i)}(A)(\gamma^i_1, \ldots, \gamma^i_{p(i)}),$ and so 
\[\pi_i(x)\otimes 1_{a_{jit}\alpha_{jit}}\in \M_{a_{jit}p(i)}(A)((\gamma^i_1-\alpha_{jit})^{a_{jit}}, \ldots, (\gamma^i_{p(i)}-\alpha_{jit})^{a_{jit}})\]
for any $t=1,\ldots, k_{ji}$ if $k_{ji}>0.$ Since $N_{j}=\sum_{i=1, k_{ji}>0}^n\;\sum_{t=1}^{k_{ji}}a_{jit}\,p(i)$ we have that 
\[\bigoplus_{i=1, k_{ji}>0}^n\;\bigoplus_{t=1}^{k_{ji}} \pi_i(x)\otimes 1_{a_{jit}\alpha_{jit}}\in \M_{N_j}(\overline{(\gamma^{i}_k-\alpha_{jit})^{a_{jit}}})\] 
if $N_j>0.$ Thus, we have that    
\[\Phi^j\left(\bigoplus_{i=1, k_{ji}>0}^{n}\;\bigoplus_{t=1}^{k_{ji}}  \pi_i(x)\otimes 1_{a_{jit}\alpha_{jit}}\right)\in \M_{N_j}(A)(\delta^j_1, \ldots, \delta^j_{N_j})\oplus 0_{q(j)-N_j}\subseteq  \M_{q(j)}(A)(\delta^j_1, \ldots, \delta^j_{q(j)}).\]
The above formula holds in the case when $N_j=0$ by the definition of $\Phi^j$ as well. So, it holds for any $j=1, \ldots, m$ regardless of the value of $N_j.$ This demonstrates that 
\[\phi(x)\in \sum_{j=1}^m\iota_j\left(\M_{q(j)}(A)(\delta^j_{1}, \ldots, \delta^j_{q(j)})\right) = \bigoplus_{j=1}^m \M_{q(j)}(A)(\delta^j_{1}, \ldots, \delta^j_{q(j)})=S\]
and so the  map $\phi$ is well-defined. The map $\phi$ is also a graded $*$-homomorphism since it is defined by the graded $*$-homomorphisms $\iota_j,$ 
$\Phi^j,$ $\phi_{a_{jit}}$, $\phi_{\alpha_{jit}},$ and $\pi_i.$ 

We show that $K^{\gr}_0(\phi)=f$ now. Recall that 
\[f([e_{11}^i])= \sum_{j=1, k_{ji}>0}^m \sum_{t=1}^{k_{ji}} a_{jit}\alpha_{jit}[f_{11}^j]\]
if $k_{ji}>0$ for some $j=1,\ldots, m$ and $f([e_{11}^i])=0$ otherwise. So, we need to show that $[\phi(e^i_{11})]=\sum_{j=1, k_{ji}>0}^m \sum_{t=1}^{k_{ji}} a_{jit}\alpha_{jit}[f_{11}^j]$ if $k_{ji}>0$ for some $j=1,\ldots, m$ and $[\phi(e^i_{11})]=0$ otherwise.

If $k_{ji}=0$ for all $j=1,\ldots, m,$ then we have that $\phi(e_{11}^i)=0\in S$  by the definition of $\phi$ and so $[\phi(e^i_{11})]=0\in K^{\gr}_0(S).$ 

If $k_{ji}>0$ for some $j=1,\ldots, m,$ let us consider one such $j.$ Since $k_{ji}>0,$ we have that $N_j>0$ also. In this case the matrix $\pi_{i}(e_{11}^{i}))\otimes 1_{a_{jit}\alpha_{jit}}$ is a diagonal matrix with the first $a_{jit}$ entries on the diagonal equal to 1 and the rest equal to 0. Note also that the matrix $\pi_{i'}(e_{11}^i)$ is a zero matrix for $i'\neq i.$ Let us consider the following matrix:   
$$E_{11}^{ij}:=\bigoplus_{i'=1, k_{ji'}>0}^{n}\;\bigoplus_{t=1}^{k_{ji'}}\pi_{i'}(e^i_{11})\otimes 1_{a_{ji't}\alpha_{ji't}}.$$ 
The matrix $E_{11}^{ij}$
is an $N_{j}\times N_{j}$ diagonal matrix with exactly $\sum_{t=1}^{k_{ji}}a_{jit}$ diagonal entries equal to 1 and the rest of the diagonal entries equal to 0.
The nontrivial entries of $E_{11}^{ij}$ correspond to values $(i, t, s, 1)\in S_j.$ If $\Pi^j$ maps one such $(i, t, s, 1)$ onto $\sigma^j(l)$ for some $l\in\{1, \ldots, q(j)\},$ then $E^{ij}_{11}$ has 1 on the $l$-th diagonal entry. Note that $l=\rho^j\Pi^j(i, t, s, 1)$ in this case. 

Thus, if we let $g_{kl}^j$ denote the standard matrix units of $\M_{q(j)}(\delta^j_1,\ldots, \delta^j_{q(j)}),$ so that $\pi_j(f^j_{kl})=g_{kl}^j$ and $\iota_j(g_{kl}^j)=f_{kl}^j$ for $k,l=1,\ldots, q(j),$ then $g^j_{\rho(k)\rho(l)}$ are the standard matrix units of $\M_{q(j)}(\delta^j_{\sigma^j(1)}, \ldots, \delta^j_{\sigma^j(q(j))})$ by Proposition \ref{permutation_of_components}. 

With this notation and the convention that $(g^j_{kk})^a$ represents $g^j_{kk}+g^j_{(k+1)(k+1)}+\ldots+g^j_{(k+a-1)(k+a-1)}$
for any $k=1,\ldots, q(j)$ and any positive integer $a,$
we have the following:
\[E_{11}^{ij}=\sum_{t=1}^{k_{ji}}\sum_{s=1}^{a_{jit}}g^j_{\rho^j\Pi^j(i, t, s, 1)\rho^j\Pi^j(i, t, s, 1)}=
\sum_{t=1}^{k_{ji}}\left(g^j_{\rho^j\Pi^j(i, t, 1, 1)\rho^j\Pi^j(i, t, 1, 1)}\right)^{a_{jit}}.\]

Thus, in $K_0^{\gr}(\M_{N_j}(A)(\overline{\delta^j_{\Pi^j(i, t, s, k)}-\varepsilon^j_{\Pi^j(i, t, s, k)}}))=K_0^{\gr}(\M_{N_j}(A)(\delta^j_{\sigma^j(1)}-\varepsilon^j_{\sigma^j(1)},\ldots,\delta^j_{\sigma^j(N_j)}-\varepsilon^j_{\sigma^j(N_j)})),$ we have that 
\[[E_{11}^{ij}]=\sum_{t=1}^{k_{ji}}[\left(g^j_{\rho^j\Pi^j(i, t, 1, 1)\rho^j\Pi^j(i, t, 1, 1)}\right)^{a_{jit}}]=
\sum_{t=1}^{k_{ji}}a_{jit}\left[ g^j_{\rho^j\Pi^j(i, t, 1, 1)\rho^j\Pi^j(i, t, 1, 1)}\right]=
\sum_{t=1}^{k_{ji}}a_{jit}\left[ g^j_{l_tl_t}\right]\] 
if $l_t\in\{1, \ldots, N_j\}$ is such that  $\sigma^j(l_t)=\Pi^j(i, t, 1, 1)$ for $t=1, \ldots, k_{ji}.$ By formula  (\ref{matrix_units_identity}), this implies  
that \[[E_{11}^{ij}]=\sum_{t=1}^{k_{ji}}a_{jit}\left[ g^j_{l_tl_t}\right]=
\sum_{t=1}^{k_{ji}}a_{jit}(\delta^j_1-\varepsilon^j_{1}-\delta^j_{\sigma^j(l_t)}+\varepsilon^j_{\sigma^j(l_t)})\left[g^j_{\rho(1)\rho(1)}\right]
\]
where $\varepsilon^j_1$ is defined as previously if $\rho(1)\in\{1,\ldots, N_j\}$ and it is $0$ otherwise. Thus, \[
[\Phi^j(E_{11}^{ij})]=
\sum_{t=1}^{k_{ji}}a_{jit} (\delta^j_1-\varepsilon^j_{1}-\delta^j_{\sigma^j(l_t)}+\varepsilon^j_{\sigma^j(l_t)}) \left[ g^j_{11}\right]\]
by the definition of the map $\Phi^j$ and the fact that $\Phi^j$ induces a $\Z[\Gamma]$-isomorphism on the $K_0^{\gr}$-level. 

Note that $ (\delta^j_1-\varepsilon^j_{1}-\delta^j_{\sigma^j(l_t)}+\varepsilon^j_{\sigma^j(l_t)}) \left[ g^j_{11}\right]=(-\delta^j_{\sigma^j(l_t)}+\varepsilon^j_{\sigma^j(l_t)})\left[ g^j_{11}\right]$ since $\delta^j_1$ and $\varepsilon^j_1$ are in $\Gamma_A.$ 
By the definition of the map $\sigma^j,$ we obtain $-\delta^j_{\sigma^j(l_t)}+\varepsilon^j_{\sigma^j(l_t)}=-\gamma_1^i+\alpha_{jit}.$
Thus,  
\[[\Phi^j(E_{11}^{ij})]=
\sum_{t=1}^{k_{ji}} a_{jit}(-\delta^j_{\sigma^j(l_t)}+\varepsilon^j_{\sigma^j(l_t)})[g_{11}^j]=\sum_{t=1}^{k_{ji}} a_{jit}(-\gamma^i_1+\alpha_{jit})[g_{11}^j]\]
and, since $\gamma_1^i=0,$
\[[\Phi^j(E_{11}^{ij})]=\sum_{t=1}^{k_{ji}} a_{jit}(-\gamma^i_1+\alpha_{jit})[g_{11}^j]=\sum_{t=1}^{k_{ji}} a_{jit}\alpha_{jit}[g_{11}^j] \]
which implies \[[\iota_j(\Phi^j(E_{11}^{ij}))]=
\sum_{t=1}^{k_{ji}} a_{jit}\alpha_{jit}[f_{11}^j].\]
As $\phi(e^i_{11})=\sum_{j=1}^m \iota_j(\Phi^j(E_{11}^{ij}))=\sum_{j=1, k_{ji}>0}^m \iota_j(\Phi^j(E_{11}^{ij})),$ we have that  
\[K_0(\phi)([e^{i}_{11}])=[\phi(e_{11}^i)]=\sum_{j=1, k_{ji}>0}^m
[\iota_j(\Phi^j(E_{11}^{ij}))]=
\sum_{j=1, k_{ji}>0}^m\sum_{t=1}^{k_{ji}} a_{jit}\alpha_{jit}[f_{11}^j]=f([e_{11}^i])\]
which finishes the proof in the case when  $k_{ji}>0$ for some $j=1,\ldots, m.$ Thus, we have that $K_0(\phi)([e^{i}_{11}])=f([e_{11}^i])$ for any $i=1,\ldots,n$ in either of the two cases. So $K_0^{\gr}(\phi)=f.$

Lastly, let us assume that $f$ is unit-preserving. In this case both the pre-dimension and the dimension formulas have equalities, $N_j=q(j)>0$ and $\Pi^j$ is a bijection for every $j=1,\ldots, m.$ Thus, for every $j=1,\ldots,m$ there is $i=1,\ldots, n$ with $k_{ji}>0.$ Moreover, for every $i=1,\ldots, n$ there is $j=1,\ldots, m$ with $k_{ji}>0$ since $f([e_{11}^i])\neq 0$ for every $i=1,\ldots, n$ (otherwise $f$ would map $1_R=\sum_{i=1}^n\sum_{k=1}^{p(i)}-\gamma^i_k[e^i_{11}]$ to 0).  

This enables us to define $E^{ij}_{kk}$ for every $i=1,\ldots, n$ every $j=1,\ldots, m$ such that $k_{ji}>0,$ and every $k=1,\ldots, p(i),$ analogously to $E^{ij}_{11}$ as follows.
$$E_{kk}^{ij}:=\bigoplus_{i'=1, k_{ji'}>0}^{n}\;\bigoplus_{t=1}^{k_{ji'}}\pi_{i'}(e^i_{kk})\otimes 1_{a_{ji't}\alpha_{ji't}}$$ 
for $e^i_{kk}\in R.$  
We also have that 
\[E_{kk}^{ij}=\sum_{t=1}^{k_{ji}}\sum_{s=1}^{a_{jit}}g^j_{\rho^j\Pi^j(i, t, s, k)\rho^j\Pi^j(i, t, s, k)}\] which implies 
\[\sum_{k=1}^{p(i)}E_{kk}^{ij}=\sum_{k=1}^{p(i)}\sum_{t=1}^{k_{ji}}\sum_{s=1}^{a_{jit}}g^j_{\rho^j\Pi^j(i, t, s, k)\rho^j\Pi^j(i, t, s, k)}.\]
By the definition of $\phi$ and since $\Pi^j$ is a bijection for any $i$ and $j$ with $k_{ji}>0,$ we have that 
\begin{multline*}
\phi(1_R)= \sum_{i=1}^n\sum_{k=1}^{p(i)}\phi(e^i_{kk}) =\sum_{i=1}^n\sum_{k=1}^{p(i)}\sum_{j=1, k_{ji}>0}^m\;\iota_j(\Phi^j(E^{ij}_{kk}))=\\
\sum_{j=1}^m\sum_{i=1, k_{ji}>0}^n\;\sum_{k=1}^{p(i)}\iota_j(\Phi^j(E^{ij}_{kk}))=
\sum_{j=1}^m\iota_j(\Phi^j(\sum_{i=1, k_{ji}>0}^n\;\sum_{k=1}^{p(i)}E^{ij}_{kk}))=\\
\sum_{j=1}^m\iota_j(\Phi^j(\sum_{i=1, k_{ji}>0}^n\;\sum_{k=1}^{p(i)}\sum_{t=1}^{k_{ji}}\sum_{s=1}^{a_{jit}}g^j_{\rho^j\Pi^j(i, t, s, k)\rho^j\Pi^j(i, t, s, k)}))
=\sum_{j=1}^m \iota_j(1_{q(j)})=1_S
\end{multline*}
which finishes the proof. 
\end{proof}

Let us examine the construction from Theorem \ref{fullness} in an example. 

\begin{example} Let $\Gamma=\Z/3\Z=\langle x | x^3=1\rangle$ so that $\Z[\Gamma]$ can be represented as $\Z[x]/(x^3=1).$ We use the multiplicative notation for the operation in $\Gamma$ in this example in order to distinguish between 0 in $\Z$ and the identity in $\Gamma.$ 

Let $A$ be any $*$-field trivially graded by $\Gamma$ and  let
\[R=\M_2(A)(1,x)\oplus A(x)\;\;\mbox{  and }\;\;S=\M_5(A)(1, 1, x, x, x^2)\oplus \M_4(A)(1, 1, x^2, x^2).\] 
Thus we have that $n=m=2,$ $l^1=3,\; l^2=2,$ $s^1_1=2,\; s^1_2=2,\; s^1_3=1,$ $s^2_1=2,$ and $s^2_2=2.$ 

If $f$ is a contractive homomorphism $K^{\gr}_0(R)\to K_0^{\gr}(S),$ let $f([e^i_{11}])=\sum_{j=1}^2\sum_{t=1}^3 a_{jit} x^{t-1}[f^j_{11}]$
where the coefficients $a_{jit}$ are nonnegative (we shall not be requiring that they are positive just yet). Considering the $j$-th terms of the relation 
$f([1_R])\leq [1_S],$ we obtain the following relations: 
\[\mbox{ for }j=1, \;\;\;\; a_{111}+a_{112}x+a_{113}x^2+a_{111}x^2+a_{112}+a_{113}x+a_{121}x^2+a_{122}+a_{123}x\leq 2+2x^2+x, \] 
\[\mbox{ for }j=2,\;\;\;\; a_{211}+a_{212}x+a_{213}x^2+a_{211}x^2+a_{212}+a_{213}x+a_{221}x^2+a_{222}+a_{223}x\leq 2+2x. \;\;\;\;\;\;\;\;\]
These imply the following.
\[
\begin{array}{llll}
j=1  & j'=1 & (\delta_{11}^1=1) &  a_{111}+a_{112}+a_{122}\leq 2 \\
& j'=2 & (\delta^1_{21}=x) &  a_{111}+a_{113}+a_{121}\leq 2 \\
& j'=3 & (\delta^1_{31}=x^2) & a_{112}+a_{113}+a_{123}\leq 1 \\
j=2 & j'=1 & (\delta^2_{11}=1) & a_{211}+a_{212}+a_{222}\leq 2 \\
& j'=2 & (\delta^2_{21}=x^2) & a_{213}+a_{212}+a_{223}\leq 2 \\
&& (\mbox{terms with }x) & a_{211}+a_{213}+a_{221}\leq 0\\
\end{array}
\] 
From the last equation, it follows that $a_{211}=a_{213}=a_{221}=0.$ Substituting these values in the remaining equations for $j=2$ produces 
$a_{212}+a_{222}\leq 2$ and $a_{212}+a_{223}\leq 2.$ In addition, 
\begin{center}
\begin{tabular}{lll}
$j=1$ & $j'=1,\;\; a_{111}+a_{112}+a_{122}\leq 2\; \Rightarrow$ & one of $a_{111}, a_{112},$ or $a_{122}$ has to be zero. Say $a_{112}=0.$\\
& $j'=2,\;\; a_{111}+a_{113}+a_{121} \leq 2\; \Rightarrow$ & one of $a_{111}, a_{113}, $ or $a_{121}$ has to be zero. Say $a_{113}=0.$\\
& $j'=3,\;\; a_{112}+a_{113}+a_{123}\leq 1\; \Rightarrow$ & $a_{123}\leq 1$ since   $a_{112}=a_{113}=0.$
\end{tabular}
\end{center}
Let us consider one such possible map $f,$ for example with 
\[a_{111}=2\Rightarrow a_{122}=a_{121}=0,\mbox{ and  with }\;a_{123}=1,\; a_{212}=1,\;  a_{222}=1,\; a_{223}=1.\]
Thus, $f=\left(\begin{array}{cc}
2 & x^2\\
x & x+x^2
\end{array}
\right),$ $k_{11}=k_{12}=k_{21}=1,$ and $k_{22}=2.$ 

For $x=\left(\left(\begin{array}{cc}
a & b\\
c & d\end{array}\right), e\right)\in R$ we have that  

$\bigoplus_{i=1}^{2}\;\bigoplus_{t=1}^{k_{1i}} \pi_i(x)\otimes 1_{a_{1it}\alpha_{1it}}=
\left(\begin{array}{cc}
a & b\\
c & d\end{array}\right)\otimes 1_{2}\oplus
e \otimes 1_{x^2}=
\left(\begin{array}{cccc|c}
a & 0 & b & 0 & 0\\
0 & a & 0 & b & 0 \\ 
c & 0 & d & 0 & 0\\
0 & c & 0 & d & 0\\ \hline 
0 & 0 & 0 & 0 & e\end{array}\right)\in \M_5(A)(1, 1, x, x, x^2)
$ and

$\bigoplus_{i=1}^{2}\;\bigoplus_{t=1}^{k_{2i}} \pi_i(x)\otimes 1_{a_{2it}\alpha_{2it}}=
\left(\begin{array}{cc}
a & b\\
c & d\end{array}\right)\otimes 1_{x}
\oplus e\otimes 1_{x}
\oplus e\otimes 1_{x^2}=
\left(\begin{array}{cc|cc}
a & b & 0 & 0\\
c & d & 0 & 0 \\ \hline
0 & 0 & e & 0\\
0 & 0 & 0 & e\\ 
\end{array}\right)\in \M_4(A)(x^2, 1, 1, x^2).
$

With the permutation $\rho^1$ being the identity and the permutation $\rho^2$ being 
$\left(
\begin{array}{cccc}
1 & 2 & 3 & 4 \\
3 & 1 & 2 & 4
\end{array}
\right),$ the map $\phi$ is given by 
\[\left(\left(\begin{array}{cc}
a & b\\
c & d\end{array}\right), e\right)\mapsto \left(
\left(\begin{array}{ccccc}
a & 0 & b & 0 & 0\\
0 & a & 0 & b & 0 \\ 
c & 0 & d & 0 & 0\\
0 & c & 0 & d & 0\\  
0 & 0 & 0 & 0 & e\end{array}\right),
\left(\begin{array}{cccc}
d & 0 & c & 0\\
0 & e & 0 & 0 \\ 
b & 0 & a & 0\\
0 & 0 & 0 & e\\ 
\end{array}\right)\right)\in \M_5(A)(1, 1, x, x, x^2)\oplus\M_4(A)(1, 1, x^2, x^2).\]
\label{example_graded}
\end{example}

Since a field which is trivially graded by a group has enough unitaries, Theorem \ref{fullness} has the following direct corollary. 

\begin{corollary}\label{nongraded_fullness}
If $R$ and $S$ are matricial $*$-algebras over a $*$-field $A$ and $f:K_0(R)\rightarrow K_0(S)$ is
a contractive homomorphism of abelian groups, then
there is an $A$-algebra $*$-homomorphism  $\phi:R\rightarrow S$ such that $K^{\gr}_0(\phi)=f$. Furthermore, if $f$ is unit-preserving, then $\phi$ is unital. 
\end{corollary}

If we do not consider any involution on the field $A,$ we obtain the graded version of \cite[Lemma 15.23]{Goodearl_book} as a direct corollary of Theorem \ref{fullness} as we illustrate below. 

\begin{corollary}\label{noninvolutive_fullness}
Let $A$ be a $\Gamma$-graded field and let $R$ and $S$ be graded matricial algebras over $A$. If $f:K^{\gr}_0(R)\rightarrow K^{\gr}_0(S)$ is a contractive $\Z[\Gamma]$-module homomorphism, then there is a graded $A$-algebra homomorphism  $\phi:R\rightarrow S$ such that $K^{\gr}_0(\phi)=f$. Furthermore, if $f$ is unit-preserving, then $\phi$ is unital. 
\end{corollary}
\begin{proof}
Without the requirement that $A$ has enough unitaries, we can obtain the elements 
$\varepsilon^j_{\Pi^j(i, t, s, k)}\in\Gamma_A$ such that there is an invertible (but not necessarily unitary) element in $A_{\varepsilon^j_{\Pi^j(i, t, s, k)}}$ 
for every $j=1,\ldots,m,$ such that $k_{ji}>0$ for at least one $i=1,\ldots,n$ and $(i,t,s,k)\in S_j$ (see Remark \ref{noninvolutive_permutation}). With this modification, the map $\Pi^j$ is defined in the same way as in the proof of Theorem \ref{fullness}.
In this case, one can follow all the steps of the proof of Theorem \ref{fullness} and use the second part of 
\cite[Theorem 1.3.3]{Roozbeh_graded_ring_notes} instead of the second part of Proposition \ref{permutation_of_components} as Remark \ref{noninvolutive_permutation} illustrates. For a contractive $\Z[\Gamma]$-module homomorphism  $f:K^{\gr}_0(R)\rightarrow K^{\gr}_0(S)$, this process would produce a graded $A$-algebra homomorphism  $\phi:R\rightarrow S$ such that $K^{\gr}_0(\phi)=f$ and which is unital provided that $f$ is unit-preserving.
\end{proof}

\section{Faithfulness}\label{section_faithfulness}

In this section, we show that the functor  $K^{\gr}_0$ is faithful on a quotient of the category $\mathcal C$ (defined in the first paragraph of section \ref{section_fullness}) for graded $*$-fields $A$ in which the involution has certain favorable properties. More precisely, in Theorem \ref{faithfulness}, we show that if $A_0$ is a 2-proper and $*$-pythagorean graded $*$-field, $R$ and $S$ are graded matricial $*$-algebras over $A,$ and $\phi$ and $\psi$ are graded $*$-homomorphisms $R\to S,$ then $K^{\gr}_0(\phi)= K^{\gr}_0(\psi)$ if and only if there is a unitary element $u$ of degree zero in $S$ such that $\phi(r)=u\psi(r)u^*$ for any $r\in R.$ This result is a graded, involutive version of an analogous non-graded and non-involutive result from \cite{Goodearl_book} and a non-involutive result from \cite{Roozbeh_Annalen}. 

We begin with some preliminary results regarding algebraic equivalence of projections in $*$-rings. Just as in the previous section, all rings are assumed to be unital but homomorphisms between them are not necessarily unit-preserving. 

\subsection{Algebraic and *-equivalence in a *-ring.}
Recall that idempotents $p,q$ of any ring $A$ are said to be {\em algebraically equivalent}, written as $p\sima_A q$ if and only if there are $x,y\in A$ such that $xy=p$ and $yx=q.$ If it is clear from the context in which ring we are considering all the elements $p,q,x,$ and $y$ to be, we simply write $p\sima q.$ In this case, we say that $x$ and $y$ {\em realize} the equivalence $p\sima q.$ The idempotents $p$ and $q$ are {\em similar} if there is an invertible element $u\in A$ such that  $p=uqu^{-1}.$ 

If $x,y \in A$ realize the algebraic equivalence of idempotents $p$ and $q,$ replacing $x$ and $y$ with $pxq$ and $qyp$ respectively produces elements which also realize the same equivalence but also have the properties that $px=x=xq$ and $qy=y=yp.$  With these adjustments, if $x,y \in A$ realize the equivalence $p\sima q$, then the left multiplication $L_x$ is a right $A$-module isomorphism $qA\cong pA$ and $L_y$ is its inverse. Conversely, if $qA\cong pA$ as right $A$-modules for some idempotents $p,q\in A,$ $\phi$ denotes the isomorphism $qA\to pA$ and $\psi$ its inverse, then $\phi(q)=x$ and $\psi(p)=y$ realize the algebraic equivalence of the idempotents $p$ and $q.$ This shows that 
\begin{center}
$p\sima q$ if and only if $pA\cong qA$ as right $A$-modules
\end{center}
for any idempotents $p,q\in A.$ Analogously, $p\sima q$ if and only if $Ap\cong Aq$ as left $A$-modules for any idempotents $p,q\in A.$

In an involutive ring, the selfadjoint idempotents, i.e. {\em the projections}, take over the role of idempotents. In this case, the following relation of projections represents the involutive version of the algebraic equivalence. The projections  $p,q$ of a $*$-ring $A$ are {\em $*$-equivalent}, 
written as $p\sims q,$ if and only if there is $x\in A$ such that $xx^*=p$ and $x^*x=q.$ In this case, we say that $x$ realizes the equivalence $p\sims q.$ 
The projections $p$ and $q$ are {\em unitarily equivalent} if there is a unitary $u\in A$ (i.e. an invertible element $u$ with $u^{-1}=u^*$) such that $p=uqu^{-1}.$ 

Several other authors have been studying conditions enabling a transfer from algebraic to $*$-equivalence on matrix algebras over a $*$-ring. In order to have such a transfer, one may need to impose some requirements on the involution. First, we consider the requirement that the involution is positive definite. Recall that an involution $*$ on $A$ is said to be \emph{$n$-proper} (or $n$-positive definite as some call this condition) if, for all $x_1, \dots,
x_n \in A$, $\sum_{i=1}^n x_ix_i^* = 0$ implies $x_i=0$ for each $i=1,\ldots,
n$ and that it is {\em positive definite} if it is $n$-proper for every $n.$ A 1-proper involution is usually simply said to be proper. A $\ast$-ring with a positive definite ($n$-proper) involution is said to be a {\em positive definite ($n$-proper)} ring. It is straightforward to check that a $*$-ring $A$ is $n$-proper if and only if $\M_n(A)$ is proper (see, for example \cite[Lemma 2.1]{Gonzalo_Ranga_Lia}). This implies that a $*$-ring $A$ is positive definite if and only if $\M_n(A)$ is positive definite for any $n.$  

The following property of a $*$-ring is also relevant. A $*$-ring $A$ is {\em $*$-pythagorean} if condition (P) below holds.  
\begin{enumerate}
\item[(P)] For every $x$ and $y$ in $A,$ there is $z\in A$ such that $xx^* + yy^* = zz^*.$
\end{enumerate}
Our interest in this condition is based on the following fact: by \cite[Lemma 3.1]{Handelman} and \cite[Theorem 1.12]{Ara_matrix_rings}, the following conditions are equivalent for a 2-proper $*$-field $A.$
\begin{enumerate}
\item The field $A$ is $*$-pythagorean. 

\item Algebraically equivalent projections of $\M_2(A)$ are $*$-equivalent. 

\item Algebraically equivalent projections of $\M_n(A)$ are $*$-equivalent for any $n.$
\end{enumerate}
 
It is direct to show that a $*$-pythagorean ring is positive definite if and only if it is 2-proper. The properties of being 2-proper and $*$-pythagorean are independent as the field of rational numbers with the identity involution and the field of complex numbers with the identity involution illustrate.  

The fields of complex and real numbers with the identity involution are $*$-pythagorean since $\sqrt{x^2+y^2}$ can be taken for $z$ in condition (P). The form of the element $z$ in these examples indicates the connection between condition (P) and the existence of a positive square root of a positive element. We illustrate that this connection is more general than it may first seem. 

The concept of positivity in a $*$-ring can be defined as follows. An element of a $*$-ring
$A$ is \emph{positive} if it is a finite 
sum of elements of the form $xx^*$ for $x \in A.$ The notation $x\geq 0$ is usually
used to denote that $x$ is positive. Note that a positive element $x$ is selfadjoint and so $x^2=xx^*.$    

As it turns out, a weaker condition than the existence of a positive square root of a positive element is sufficient for a $*$-ring to be $*$-pythagorean provided that $A$ is 2-proper. Let $(x)''$ denote the {\it double commutant} of an element $x$ in a ring $A$, i.e. the set of all elements which commute with every element which commutes with $x$. A $*$-ring $A$ is said to satisfy the {\em  weak square root axiom,} or the (WSR) axiom for short, if the following condition holds. 
\begin{enumerate}
\item[(WSR)] For every   $x\in A$ there is $r\in (x^*x)''$ such that $x^*x=r^*r.$ 
\end{enumerate}
We also say that $A$ satisfies the {\em (WSR) axiom matricially}, if $\M_n(A)$ satisfies the (WSR) axiom for every positive integer $n.$

We note that every $C^*$-algebra satisfies the (WSR) axiom (see \cite[remark following Definition 10 in Section 13]{Berberian}). 

By \cite[Section 1, Exercise 8A]{Berberian} (and also \cite[Proposition 6.10]{Berberian_web}), a $*$-ring with (WSR) is such that algebraically equivalent projections are $*$-equivalent. Thus, if $A$ is a 2-proper $*$-field, the assumption that $\M_2(A)$ satisfies the (WSR) axiom is stronger than the assumption that $A$ is $*$-pythagorean. 
This explains our decision to formulate Theorems \ref{faithfulness} and \ref{ultramatricial} using the assumptions that the zero-component $A_0$ of a graded $*$-field $A$ is 2-proper and $*$-pythagorean instead of the assumption that $A_0$ satisfies the (WSR) axiom matricially.  

We summarize some properties of algebraic and $*$-equivalence which we shall need in the proof of Theorem \ref{faithfulness}. 
\begin{lemma}\label{nongraded_equivalence} 
Let $A$ be a ring and let $p$ and $q$ be idempotents in $A.$
\begin{enumerate}[\upshape(1)]
\item \cite[Proposition 5.5]{Berberian_web} If $x\in pAq$ and $x'\in qAp$ realize the equivalence $p\sima q$ and if $y\in (1-p)A(1-q)$ and $y'\in (1-q)A(1-p)$ realize the equivalence $1-p\sima 1-q,$ then $u=x+y$ is invertible with inverse $u^{-1}=x'+y'$ and $p=uqu^{-1}.$ 

\item \cite[Theorem 4.1]{Goodearl_book} If $A$ is unit-regular, then  $p\sima q$ if and only if   $1-p\sima 1-q.$

\item \cite[Proposition 6.5]{Berberian_web} If $A$ is a $*$-ring and $p$ and $q$ are projections such that $x\in pAq$ realizes the $*$-equivalence $p\sims q$ and $y\in (1-p)A(1-q)$ realizes the $*$-equivalence $1-p\sima 1-q,$ then $u=x+y$ is unitary and $p=uqu^{-1}.$  
 
\item \cite[Theorem 3]{Prijatelj_Vidav} If $A$ is a 2-proper, $*$-pythagorean, regular $*$-ring, then $\M_n(A)$ has the same properties for any $n.$

\item  \cite[Lemma 3.1]{Handelman} and \cite[Theorem 1.12]{Ara_matrix_rings} If $A$ is a 2-proper, commutative, regular $*$-ring, then $A$ is $*$-pythagorean if and only if algebraically equivalent projections of $\M_n(A)$ are $*$-equivalent for any $n.$ 
\end{enumerate}
\end{lemma}

The references for the proofs of the statements in this lemma are provided in the lemma itself. The parts (4) and (5) can be formulated with the assumption that the ring is $*$-regular instead of regular and then they follow from statements of \cite[Theorem 3]{Prijatelj_Vidav} and \cite[Theorem 1.12]{Ara_matrix_rings} directly. The statements (4) and (5) still hold with the seemingly weaker assumption since a $*$-ring is $*$-regular if and only if it is proper and regular (\cite[Exercise 6A, Section 3]{Berberian}).

We adapt the concepts and claims from this section to the graded setting now.  

\subsection{Algebraic and *-equivalence in a graded *-ring}
Let $A$ be a $\Gamma$-graded ring. Any homogeneous idempotent $p$ is necessarily in the ring  $A_0.$ Thus, there are three notions of algebraic equivalence for idempotents $p,q \in A_0$ corresponding to the following scenarios. 
\begin{enumerate}
\item There are $x,y\in A_0$ such that $xy=p$ and $yx=q.$ This corresponds to the non-graded algebraic equivalence $p\sima_{A_0} q.$
 
\item There are $x\in A_\gamma$ and $y\in A_{-\gamma}$ such that $xy=p$ and $yx=q.$ In this case we say that  $p$ and $q$ are {\em graded algebraically equivalent} and write  $p\sima_{\gr} q.$

\item There are $x,y \in A$ such that $xy=p$ and $yx=q.$ This corresponds to the non-graded algebraic equivalence $p\sima_{A} q.$
\end{enumerate}
In any of these instances, we say that $x$ and $y$ {\em realize} the corresponding equivalence. Note that all three equivalence relations are equal if the ring is trivially graded. 

Just as in the non-graded case, if $x$ and $y$ realize any of the three versions of the algebraic equivalence of $p$ and $q,$ the elements $x$ and $y$ can be replaced by $pxq$ and $qyp$ respectively producing elements which also realize the same equivalence but also have the properties that $px=x=xq$ and $qy=y=yp.$

With these adjustments, if $x,y \in A_0$ realize the equivalence $p\sima_{A_0} q$, then the left multiplication $L_x$ is a graded right module isomorphism $qA\cong_{\gr}pA$ and $L_y$ is its inverse. In this case,  $L_x$ and $L_y$ are in $\Hom_{\Gr A}(A, A).$ 

Similarly, if $x\in A_\gamma$ and $y \in A_{-\gamma}$ realize the equivalence $p\sima_{\gr} q$, then $L_x$ is an isomorphism $qA\cong_{\gr}pA(\gamma)$ of graded right modules and $L_y$ is its inverse and $L_x$ and $L_y$ are in $\Hom_A(A, A).$ 

The converse of this statement also holds: if $qA\cong_{\gr}pA(\gamma)$ for some $\gamma\in \Gamma,$ then there is $x\in A_\gamma$ and $y\in A_{-\gamma}$ such that $xy=p$ and $yx=q.$ Thus, we have that 
\begin{center}
$p\sima_{\gr} q$ if and only if  $qA\cong_{\gr}pA(\gamma)$ as graded right $A$-modules for some $\gamma\in \Gamma$
\end{center}
for all homogeneous idempotents $p,q\in A.$ In particular, 
\begin{center}
$p\sima_{A_0} q$ if and only if  $qA\cong_{\gr}pA$ as graded right $A$-modules. 
\end{center}
This equivalence illustrates the reason our results require us to consider the relation $\sima_{A_0}$ only.    

If homogeneous idempotents of a ring represent the same element in the graded Grothendieck group, they can be enlarged to algebraically equivalent matrices but they are not necessarily algebraically equivalent themselves. However, for homogeneous idempotents of a graded field, a stronger statement holds and we note it in the following lemma.   
\begin{lemma}\label{sima_in_graded_field}
Let $A$ be a $\Gamma$-graded field and $R$ a graded matricial $A$-algebra. If $p$ and $q$ are homogeneous idempotents in $R,$ then 
\begin{center}
$[p]=[q]$ in $K^{\gr}_0(R)$ if and only if $p\sima_{R_0} q$. 
\end{center}
\end{lemma}
\begin{proof}
If $R$ is a matrix algebra and $[p]=[q]$ in $K^{\gr}_0(R),$ then $[pR]$ and $[qR]$ are the same element of $K^{\gr}_0(R)$ using the representation of  $K^{\gr}_0(R)$ by the classes of finitely generated graded projective modules. By \cite[Proposition 5.1.2]{Roozbeh_graded_ring_notes} and part (3) of Proposition \ref{graded_division_ring} (or, equivalently, \cite[Proposition 3.7.1]{Roozbeh_graded_ring_notes}), $pR$ and $qR$ are both graded isomorphic to a module of the form $A^{k_1}(\gamma_i)\oplus \ldots\oplus A^{k_n}(\gamma_n)$  where $\sum_{i=1}^n k_i\gamma_i+\Gamma_A$ is a unique element  of $\Z[\Gamma/\Gamma_A]$ which corresponds to $[pR]=[qR]$ under the isomorphism of $K_0^{\gr}(R)$ and $\Z[\Gamma/\Gamma_A].$ Thus,  $pR\cong_{\gr} qR$ and so $p\sima_{R_0} q.$

This argument extends from a single graded matrix algebra to a direct sum of graded matrix algebras, i.e. to a graded matricial algebra. 
\end{proof} 

If $A$ is a graded $*$-ring, there are three notions of $*$-equivalence for projections $p,q \in A_0.$  
\begin{enumerate}
\item There is $x\in A_0$ such that $xx^*=p$ and $x^*x=q.$ This corresponds to the non-graded $*$-equivalence $p\sims_{A_0} q.$ 
 
\item There is $x\in A_\gamma$ such that $xx^*=p$ and $x^*x=q.$ In this case we say that  $p$ and $q$ are {\em graded $*$-equivalent} and write  $p\sims_{\gr} q.$

\item There is $x\in A$ such that $xx^*=p$ and $x^*x=q.$ This corresponds to the non-graded $*$-equivalence $p\sims_{A} q.$
\end{enumerate}
In any of these instances, we say that $x$ {\em realizes} the corresponding $*$-equivalence.  

The three notions of algebraic equivalence in a graded ring $A$ correspond to the three notions of similarity. Since we shall be working with $\sima_{A_0}$ exclusively, the only similarity relation relevant for us is the following: two idempotents $p, q\in A_0$ are similar if there is an invertible element $x\in A_0$ such that $q=xpx^{-1}.$ Analogously, the three notions of $*$-equivalence in a graded $*$-ring $A$ correspond to the three notions of unitary equivalence. For us, just the following is relevant: two projections $p,q\in A_0$ are unitarily equivalent if $p=uqu^{-1}$ for some unitary $u\in A_0.$

One could consider the graded version of the definition of a $*$-pythagorean ring by restricting the elements from the definition to homogeneous components. However, since we only work with relations $\sima_{A_0}$ and $\sims_{A_0}$ in a graded $*$-ring $A,$ it will only be relevant for us whether $A_0$ is $*$-pythagorean or not. Thus, the non-graded definition of a $*$-pythagorean ring is sufficient for our consideration. Moreover, all the properties of the involution we shall need in the proof of Theorem \ref{faithfulness} reduce to the properties of the ring $A_0.$  

Before proving the main result of this section, we note that the zero-component of a graded matricial ring over a graded field $A$ is a unit-regular ring.  

\begin{lemma}\label{unit-regular}
If $A$ is a $\Gamma$-graded field, $n$ a positive integer and $\ol\gamma\in \Gamma^n,$ then 
$\M_n(A)(\ol\gamma)_0$ is isomorphic to a matricial algebra over $A_0$ and, as such, it is a unit-regular ring. 
\end{lemma}
\begin{proof} The shifts $\gamma_1, \ldots, \gamma_n$ can be partitioned such that the elements of the same partition part belong to the same coset of the quotient group $\Gamma/\Gamma_A$ and the elements from different partition parts belong to different cosets just as in formula (\ref{rearanging}) in the first paragraph of section \ref{subsection_graded_fullness}. Let $\gamma_{l1},\dots,\gamma_{lr_l}$ be the $l$-th part for $l=1, \ldots, k$ where $k$ is the number of partition parts and $r_l$ is the number of elements in the $l$-th part. Using Proposition \ref{permutation_of_components}, we have that 
\[
\M_n(A)(\ol\gamma)\cong_{\gr}\M_n(A)(\gamma_{11},\dots,\gamma_{1r_1}, \gamma_{21},\dots,\gamma_{2r_2}, \;\;\ldots,\;\; \gamma_{k1},\dots,\gamma_{kr_k }).
\]
Considering the zero-components of both sides, we have that 
\[
\M_n(A)(\ol\gamma)_0\cong \M_{r_1}(A_0)\oplus\ldots\oplus \M_{r_k}(A_0).\]

Since $A_0$ is a field, a matricial algebra over $A_0$ is a unit-regular ring (\cite[Corollary 4.7]{Goodearl_book}). 
\end{proof}
Note that part (4) of Lemma \ref{nongraded_equivalence} implies that if $A$ is a $\Gamma$-graded $*$-field such that $A_0$ is 2-proper and $*$-pythagorean, then 
$\M_n(A)(\ol\gamma)_0$ is 2-proper, $*$-pythagorean, $*$-regular and unit-regular for any positive integer $n$ and $\ol\gamma\in \Gamma^n.$

We are ready to prove the main result of this section. Our proof adapts the ideas of the proof of \cite[Lemma 15.23]{Goodearl_book} to the graded and involutive setting. 

\begin{theorem}\label{faithfulness}
Let $A$ be a $\Gamma$-graded $*$-field such that $A_0$ is 2-proper and $*$-pythagorean. 
If $R$ and $S$ are graded matricial $*$-algebras and $\phi, \psi: R\to S$ are graded $A$-algebra $*$-homomorphisms (not necessarily unital), then the following are equivalent. 
\begin{enumerate}[\upshape(1)]
\item $K^{\gr}_0(\phi)=K^{\gr}_0(\psi)$

\item There exists an element  $x\in S_0$ such that $\phi(r)=x\psi(r)x^*,$ for all $r\in R,$ and such that $xx^*=\phi(1_R)$ and $x^*x=\psi(1_R).$ 

\item There exists a unitary element $u\in S_0$ such that $\phi(r)=u\psi(r)u^{-1},$ for all $r\in R.$  

\item There exists an invertible element $u\in S_0$ such that $\phi(r)=u\psi(r)u^{-1},$ for all $r\in R.$  
\end{enumerate}  
\end{theorem}

\begin{proof}
Note that the assumptions on the field $A$ imply that any homogeneous projections $p,q$ of a graded matricial algebra over $A$ which are algebraically equivalent as elements of the zero-component are $*$-equivalent as elements of the zero-component by part (5) of Lemma \ref{nongraded_equivalence}. 

$(1) \Rightarrow (2).$  Let $R=\bigoplus_{i=1}^n\M_{p(i)}(A)(\gamma^i_1,\dots,\gamma^i_{p(i)})$ and let
$e_{kl}^i$, $i=1, \ldots, n$,  $k,l=1,\ldots, p(i),$ be the elements of $R$ such that $\pi_i(e^i_{kl})$ are the standard matrix units of 
$\M_{p(i)}(A)(\gamma^i_1,\dots,\gamma^i_{p(i)})$ where $\pi_i$ is the projection on the $i$-th component for every $i=1,\ldots, n.$  
Let $g_{kl}^i=\phi(e^i_{kl})$ and 
$h_{kl}^i=\psi(e^i_{kl})$ for  $k,l=1,\ldots,p(i)$ and $i=1,\ldots, n$ so that the elements $g_{kk}^{(i)}$ and $h_{kk}^{(i)}$ are projections in $S_0$ for every  $k=1,\ldots,p(i)$ and $i=1,\ldots, n.$
The condition (1) implies that 
\[[g_{11}^i]=[\phi(e_{11}^i)]=K^{\gr}_0(\phi)([e_{11}^i])=K_0^{\gr}(\psi)([e_{11}^i])=[\psi(e_{11}^i)]=[h_{11}^i].\]

By Lemma \ref{sima_in_graded_field}, $g^i_{11}\sima_{S_0}h^i_{11}$ and, by the assumptions of the theorem,  $g^i_{11}\sims_{S_0}h^i_{11}$. Let 
$x_i\in S_0$ for $i=1,\ldots, n$ be elements realizing the $*$-equivalence $g^i_{11}\sims h^i_{11}$ and let us choose them from $g^i_{11}S_0h^i_{11}.$ Let us consider an element $x\in S_0$ defined as follows:  
$$x=\sum_{i=1}^n\sum_{k=1}^{p(i)} g^i_{k1}x_ih_{1k}^i.$$ Note that $x^*=\sum_{i=1}^n\sum_{k=1}^{p(i)} h^i_{k1}x_i^*g_{1k}^i.$
We show that $xx^*=\phi(1_R),$ $x^*x=\psi(1_R),$ and that
\[
xh_{kl}^ix^*=g_{kl}^i \text{   and   } x^*g_{kl}^ix=h_{kl}^i
\]
for all $k,l=1,\ldots, p(i)$ and all $i=1,\ldots, n.$

Since $h_{1k}^i h_{l1}^j=0$ when $i\neq j$ and $k\neq l$ we have that 
\begin{align*}
xx^* & = \sum_{i=1}^n\sum_{k=1}^{p(i)} g^i_{k1}x_ih_{1k}^i \sum_{j=1}^n\sum_{l=1}^{p(j)} h^j_{l1}x_l^*g_{1l}^j = \sum_{i=1}^n\sum_{k=1}^{p(i)} g^i_{k1}x_ih_{1k}^i h^i_{k1}x_i^*g_{1k}^i  = \sum_{i=1}^n\sum_{k=1}^{p(i)} g^i_{k1}x_ih_{11}^ix_i^*g_{1k}^i\\
&  =  \sum_{i=1}^n\sum_{k=1}^{p(i)} g^i_{k1}x_ix_i^*g_{1k}^i 
 =  \sum_{i=1}^n\sum_{k=1}^{p(i)} g^i_{k1}g_{11}^ig_{1k}^i =  \sum_{i=1}^n\sum_{k=1}^{p(i)} g^i_{kk} =  \phi(\sum_{i=1}^n\sum_{k=1}^{p(i)} e^i_{kk}) = \phi(1_R).
\end{align*}
Analogously one shows that $x^*x=\psi(1_R).$ 

Let $i=1,\ldots, n$ and $k,l=1,\ldots, p(i)$ now. Using that $h_{1k}^i h_{l1}^j=0$ when $i\neq j$ and $k\neq l,$ we have that 
$$xh_{kl}^ix^* = g^i_{k1}x_ih_{1k}^i h_{kl}^i h^i_{l1}x_i^*g_{1l}^i=g^i_{k1}x_ih_{11}x_i^*g_{1l}^i=g^i_{k1}x_ix_i^*g_{1l}^i=g^i_{k1}g_{11}^ig_{1l}^i=g^i_{kl}.$$ 
Similarly we show that $x^*g_{kl}^ix=h_{kl}^i$ 
for any $k,l=1,\ldots, p(i)$ and any $i=1,\ldots, n.$  This implies $\phi(r)=x\psi(r)x^*$ for any $r\in R$ since the elements $e_{kl}^i,$  $k,l=1,\ldots, p(i),$  $i=1,\ldots, n,$ generate $R.$

$(2) \Rightarrow (3)$. Condition (2) implies that $\phi(1_R)\sims_{S_0}\psi(1_R).$ Moreover, the element $x\in S_0$ from condition (2) can be replaced by 
$\phi(1_R)x\psi(1_R)\in \phi(1_R)S_0\psi(1_R).$ Since the ring $S_0$ is unit-regular by Lemma \ref{unit-regular}, we can use part (2) of Lemma \ref{nongraded_equivalence} to conclude that 
$1_S-\phi(1_R)\sima_{S_0} 1_S-\psi(1_R)$ and, using the assumptions of the theorem, that  
$1_S-\phi(1_R)\sims_{S_0} 1_S-\psi(1_R).$ Let 
$y\in (1_S-\phi(1_R))S_0(1_S-\psi(1_R))$ be the element realizing $1_S-\phi(1_R)\sims_{S_0} 1_S-\psi(1_R).$ Then $u=x+y$ is unitary and $\phi(1_R)=u\psi(1_R)u^{-1}$ 
by part (3) of Lemma \ref{nongraded_equivalence}. 

Since $y=y(1_S-\psi(1_R)),$ we have that $y\psi(1_R)=0$ and $\psi(1_R)y^*=0.$ Thus, for any $r\in R,$ we have that 
\begin{align*}
u\psi(r)u^{-1} & = (x+y)\psi(r)(x^*+y^*) = x\psi(r)x^* +x\psi(r)y^*+y\psi(r)x^*+y\psi(r)z^*\\
&= x\psi(r)x^* +x\psi(r1_R)y^*+y\psi(1_Rr)x^*+y\psi(1_Rr1_R)y^*\\
&= x\psi(r)x^* +x\psi(r)\psi(1_R)y^*+y\psi(1_R)\psi(r)x^*+y\psi(1_R)\psi(r)\psi(1_R)y^*\\
&= x\psi(r)x^* = \phi(r).
\end{align*}

$(3) \Rightarrow (4)$ is a tautology.

$(4) \Rightarrow (1)$. Let $f_{kl}^j$,  $j=1, \ldots, m$,  $k,l=1,\ldots, q(j),$ be the elements of $S=\bigoplus_{j=1}^m \M_{q(j)}(A)(\delta^j_1,\ldots, \delta^j_{q(j)})$ such that $\pi_j(f^j_{kl}),$ for $k,l=1,\ldots, q(j),$ are the standard matrix units of $\M_{q(j)}(A))(\delta^j_1,\dots,\delta^j_{q(j)})$ for every $j=1,\ldots, m.$  Let $\theta$ denote the inner automorphism $a\mapsto uau^{-1}$ of $S.$ Since $f_{11}^j\sima_{S_0}\theta(f_{11}^j)$ for any $j=1,\ldots, m$ and the elements $[f_{11}^j], j=1,\ldots, m$ generate $K_0^{\gr}(S)$, we have that $K_0^{\gr}(\theta)$ is the identity map. Thus, $K_0^{\gr}(\phi)=K_0^{\gr}(\theta\psi)=K_0^{\gr}(\theta)K_0^{\gr}(\psi)=K_0^{\gr}(\psi).$
\end{proof}

\begin{remark}\label{remark_on_noninvolutive_faithfulness}
The direction $(4) \Rightarrow (1)$ of Theorem \ref{faithfulness} also follows from the following, more general argument. If $\theta:A\rightarrow A$ is an inner-automorphism of a  $\Gamma$-graded ring $A$  defined by $\theta(a)=u a u^{-1}$, where $u$ is an invertible element of degree $\gamma,$ then there is a   
right graded $A$-module isomorphism $P(-\gamma) \rightarrow P \otimes_{\theta} A$ given by $p
\mapsto p \otimes u$ with the inverse $p\otimes a \mapsto p u^{-1}a$ for any finitely generated graded projective right module $P.$ This induces an isomorphism between the functors $\underline{\hskip.3cm}\otimes_{\theta} A :\Pgr(A) \rightarrow \Pgr(A)$ and the 
$\gamma$-shift functor $\mathcal T_\gamma:\Pgr(A) \rightarrow \Pgr(A)$. Since the isomorphic functors  induce the same map on Quillen's $K_i$-groups, $i\geq 0$,  (see \cite[p.19]{Quillen}), one has $K^{\gr}_i(\underline{\hskip.3cm}\otimes_{\theta} A)=K^{\gr}_i(\mathcal T_\gamma)$. Therefore, if $\gamma=0,$ then $K^{\gr}_i(\theta):K^{\gr}_i(A)\rightarrow K^{\gr}_i(A)$ is the identity map.

Also note that (1) and (4) are equivalent conditions for {\em any} graded field $A$ as \cite[Theorem 2.13]{Roozbeh_Annalen} shows.
\end{remark}

If the grade group is trivial and the involution is not considered, \cite[Lemma 15.23]{Goodearl_book} shows that the conditions (1) and (4) are equivalent. The following direct corollary of Theorem \ref{faithfulness} relates these conditions with (2) and (3) for involutive matricial algebras.  

\begin{corollary}
Let $A$ be a 2-proper and $*$-pythagorean $*$-field. 
If $R$ and $S$ are matricial $*$-algebras and if $\phi, \psi: R\to S$ are algebra $*$-homomorphisms, then the following are equivalent. 
\begin{enumerate}[\upshape(1)]
\item $K_0(\phi)=K_0(\psi)$

\item There exists an element  $x\in S$ such that $\phi(r)=x\psi(r)x^*,$ for all $r\in R,$ and such that $xx^*=\phi(1_R)$ and $x^*x=\psi(1_R).$ 

\item There exists a unitary element $u\in S$ such that $\phi(r)=u\psi(r)u^{-1},$ for all $r\in R.$  

\item There exists an invertible element $u\in S$ such that $\phi(r)=u\psi(r)u^{-1},$ for all $r\in R.$  
\end{enumerate}   
\end{corollary}

\section{Graded ultramatricial *-algebras}\label{section_ultramatricial}

Using the results of the previous two sections, in this section we prove that the functor $K_0^{\gr}$ completely classifies graded ultramatricial $*$-algebras over a graded $*$-field $A$ satisfying the assumptions of Theorems \ref{fullness} and \ref{faithfulness}. Our main result, Theorem \ref{ultramatricial},
generalizes \cite[Theorem 5.2.4]{Roozbeh_graded_ring_notes} as well as \cite[Theorem 15.26]{Goodearl_book}. Our proof uses Elliott-Bratteli intertwining method analogously to how it is used in the proof of Elliott's Theorem for AF $C^*$-algebras (see \cite[Theorem IV.4.3]{Davidson} for example). However, our proof crucially depends on the results of the earlier sections as well as on graded and involutive generalizations of the analogous $C^*$-algebra statements (for example Proposition \ref{matricial_to_ultramatricial}). Theorem \ref{ultramatricial} is also a graded generalization of \cite[Proposition 3.3]{Ara_matrix_rings} but the proof of \cite[Proposition 3.3]{Ara_matrix_rings} uses different arguments than those used in \cite[Theorem 15.26]{Goodearl_book} or those used by us. 

In this section, rings are not assumed to be unital so we start the section with a quick review of the graded Grothendieck group of a graded, possibly non-unital, ring.

\subsection{Graded Grothendieck group of a non-unital ring} 
If $A$ is a ring possibly without an identity element, the Grothendieck group $K_0(A)$ can be defined using the unitization of $A$ (see, for example,  \cite{Goodearl_Handelman}). 
If $A$ is a $\Gamma$-graded ring which is possibly non-unital, define the unitization $A^u$ of $A$ to be $A\oplus \Z$ 
with the addition given component-wise, the multiplication given by
\[(a,m)(b,n)=(ab+na+mb, mn),\]
for  $a,b \in A,$ $m,n\in \Z,$ and the $\Gamma$-grading given by 
\[A^u_0= A_0\times\Z, \hskip1cm A^u_\gamma= A_\gamma\times \{0\},\;\;\text{ for }\;\;\gamma\not =0. \]
Then, $A^u$ is a graded ring with the identity $(0,1)$ and $A$ is graded isomorphic to the graded two-sided ideal $A\times\{0\}$ of $A^u.$ The graded epimorphism $A^u\rightarrow A^u/A$ produces a natural homomorphism $K^{\gr}_0(A^u) \rightarrow K^{\gr}_0(A^u/A)$. The group $K^{\gr}_0(A)$ is defined as the kernel of this homomorphism:
\[K^{\gr}_0(A) :=\ker\big(K^{\gr}_0(A^u) \rightarrow K^{\gr}_0(A^u/A)\big).\]
The group $K^{\gr}_0(A)$ inherits the $\Gamma$-action and the pre-order structure from $K_0^{\gr}(A^u).$

If $A$ is an algebra over a field $K$, then $K$ could be used instead of $\Z$ in the definition of $A^u.$ Using the graded version of the statements from \cite[Section 12]{Goodearl_Handelman}, one can show that both constructions produce isomorphic $\Z[\Gamma]$-modules for $K_0^{\gr}(A)$. Moreover, 
if $R$ is any unital graded ring containing $A$ as a graded ideal, $R$ can be used instead of $A^u.$ More details can be found in \cite[Section 3.4]{Roozbeh_graded_ring_notes} and some examples in \cite[Section 3.7]{Roozbeh_graded_ring_notes}.

If $A$ is a graded $*$-ring which is possibly non-unital, then $A^u$ becomes a unital, graded $*$-ring with the involution given by 
$(a,n)^*=(a^*, n)$ for $a\in A$ and $n\in \Z.$ In this case, $A$ is graded $*$-isomorphic to $A\times\{0\}$ and the epimorphism $A^u\rightarrow A^u/A$ is also a $*$-map where $A^u/A$ is equipped with the identity involution. In this case, the group $K^{\gr}_0(A)$ inherits the action of $\Z_2$ from $K^{\gr}_0(A^u)$ too.  

If $A$ is a graded ring which is possibly non-unital, the set  
\[
\big \{x\in \ker\left( K^{\gr}_0(A^u)\to K^{\gr}_0(A^u/A)\right)\; | \;0\leq x\leq [A^u]\big\},
\] which we refer to as the {\em generating interval}, takes over the role of the order-unit $[A]$. 
If $A$ is unital, then the definition of its $K_0^{\gr}$-group coincides with the construction using the projective modules and the generating interval is naturally isomorphic to $\{x\in K^{\gr}_0(A)\, |\, 0\leq x\leq [A]\}$ (the non-graded version of this statement is shown in \cite[Proposition 12.1]{Goodearl_Handelman}). 
By following the terminology used for $C^*$-algebras (as well as for the unital case), an order and generating interval preserving map is refer to as contractive. In other words, if $A$ and $B$ are, possibly non-unital, graded rings, then $f: K^{\gr}_0(A)\to K^{\gr}_0(B)$ is {\em contractive} if $x\geq 0$ implies that $f(x)\geq 0$ and $0\leq x\leq [A^u]$ implies that $0\leq f(x)\leq [B^u].$ 

This construction is functorial in the sense that if $A$ and $B$ are, possibly non-unital, graded rings and $\phi:A\to B$ is a homomorphism of graded rings, then $\phi^u: A^u\to B^u,$ given by $(a,n)\mapsto (\phi(a), n),$ is a homomorphism of unital graded rings such that $K_0^{\gr}(\phi^u)$ is contractive. Consequently, $K_0^{\gr}(\phi)$ is a contractive $\Z[\Gamma]$-module homomorphism and we have the following commutative diagram: 
\[
\xymatrix{
K^{\gr}_0(A) \ar@{^{(}->}[r] \ar[d]^{K^{\gr}_0(\phi)} & K^{\gr}_0(A^u) \ar[r] \ar[d]^{K^{\gr}_0(\phi^u)} & K^{\gr}_0(A^u/A) \ar@{=}[d] \\ 
K^{\gr}_0(B) \ar@{^{(}->}[r] & K^{\gr}_0(B^u) \ar[r]  & K^{\gr}_0(B^u/B).
} 
\]
If the rings $A$ and $B$ are graded $*$-rings and $\phi$ is also a $*$-homomorphism, then $K_0^{\gr}(\phi^u)$ and, consequently, $K_0^{\gr}(\phi)$ are $\Z[\Z_2]$-module homomorphisms. If, additionally, $\phi$ is an isomorphism, then $\phi^u$ is a unital isomorphism and so both $K_0^{\gr}(\phi^u)$ and $K_0^{\gr}(\phi)$ are contractive $\Z[\Gamma]$-$\Z[\Z_2]$-bimodule isomorphisms. 
 
\begin{definition}
Let $A$ be a $\Gamma$-graded $*$-field. A ring $R$ is a \emph{graded ultramatricial $*$-algebra over $A$} if 
there is an $A$-algebra $*$-isomorphism of $R$ and a direct limit of a sequence of graded matricial $*$-algebras.
Equivalently, $R$ is a directed union of an ascending sequence of graded matricial $*$-subalgebras $R_1 \subseteq R_2 \subseteq \dots$ 
such that the inclusion $R_n\subseteq R_{n+1}$ is a graded $*$-homomorphism (not necessarily unital) for any $n.$  
\end{definition}

The $\gamma$-component of a graded ultramatricial $*$-algebra $R=\bigcup_{n=1}^\infty R_n$ as in the definition above is $R_\gamma= \bigcup_{n=1}^{\infty} (R_n)_{\gamma}.$ In particular, $R_0$ is an ultramatricial $*$-algebra over $A_0.$ Since the inclusion $R_n\subseteq R_{n+1}$ is a graded $*$-homomorphism, the corresponding map on the $K_0^{\gr}$-level is contractive. Moreover, if the algebras $ R_n$ are unital and $1_{R_n}=1_{R_{n+1}}$, then $R$ is a unital ring.  

If $A$ is a graded $*$-field, the group $\Z_2$ acts trivially on the $K_0^{\gr}$-group of any matricial algebra over $A$ by 
Proposition \ref{graded_division_ring}. Consequently, $\Z_2$ acts trivially on the $K_0^{\gr}$-group of any graded ultramatricial $*$-algebra over $A$ as well.   
In Examples \ref{infinite_graphs} and \ref{line_and_clock}, we present $K_0^{\gr}$-groups of some specific graded ultramatricial $*$-algebras over a graded $*$-field.  

\subsection{The main result}
We prove the main theorem of this section using Lemmas \ref{dir_lim_group_lemma} and \ref{unit_preservation} and Proposition \ref{matricial_to_ultramatricial} below. 
Lemma \ref{dir_lim_group_lemma} and  Proposition \ref{matricial_to_ultramatricial} and their proofs follow \cite[Lemma IV.3.2]{Davidson} and \cite[Lemma IV.4.2]{Davidson} which are the analogous statements for abelian groups and AF $C^*$-algebras. 

\begin{lemma} \label{dir_lim_group_lemma}
Let $\Gamma$ be a group and let $G$ be a direct limit of $\Z[\Gamma]$-modules $G_n$ with connecting maps $\psi_{nm}: G_n\to G_m$ for $m\geq n$ and translational maps $\psi_n: G_n\to G.$ If $H$ is a finitely generated $\Z[\Gamma]$-module and $f: H\to G_n$ is a $\Z[\Gamma]$-module homomorphism such that $\psi_nf=0,$ for some $n,$ then there is $m> n$ such that $\psi_{nm}f=0.$  
\end{lemma}
\begin{proof}
If $h_1, \ldots, h_k$ are generators of $H,$ then $\psi_n(f(h_i))=0$ for all $i=1,\ldots, k.$ Thus,
for every $i=1,\ldots, k$ there is $m_i\geq n$ such that $\psi_{nm_i}(f(h_i))=0$ by the universal property of the direct limit (more details can be found in the proof of \cite[Lemma IV.3.2]{Davidson}). We can take $m$ to be any integer strictly larger than the maximum of $m_1, \ldots, m_k$ since then we have that 
$\psi_{nm}f(h_i)=\psi_{m_im}\psi_{nm_i}f(h_i)=\psi_{m_im}(0)=0$ for any $i=1,\ldots,k.$ 
\end{proof}

If $A$ is a field and $K_0$ is considered as a functor mapping the category of $A$-algebras into the category of pre-ordered abelian groups with generating intervals and contractive morphisms, then $K_0$ preserves the direct limits by \cite[Proposition 12.2]{Goodearl_Handelman}. The graded and involutive version of this result can be proven analogously. The lemma below follows from this fact. 

\begin{lemma}\label{unit_preservation}
Let $A$ be a $\Gamma$-graded $*$-field, let $R_n,$ $n=1,2,\ldots,$ be unital graded matricial $*$-algebras, and let $R=\varinjlim R_n$ be a graded ultramatricial $*$-algebra over $A$ with graded translational $*$-maps $\phi_n: R_n\to R,$ $n=1,2,\ldots$ and graded connecting $*$-maps $\phi_{mn}$ (not necessarily unital). If $x_n\in K_0^{\gr}(R_n)$ is such that $0\leq K_0^{\gr}(\phi_n)(x_n)\leq [1_{R^u}],$ then there is $m> n$ such that $0\leq K_0^{\gr}(\phi_{nm})(x_n)\leq [1_{R_m}].$
\end{lemma}
\begin{proof}
Throughout the proof, we use $\overline{h}$ to shorten the notation $K_0^{\gr}(h)$ for any graded $*$-homomorphism $h.$ 

Let  $x_n\in K_0^{\gr}(R_n)$ be such that $0\leq \ol\phi_n(x_n)\leq [1_{R^u}].$ Since the functor $K_0^{\gr}$ preserves the direct limits, there is $k\geq n$ and $x_k$ in the generating interval of $R_k$ such that $\ol\phi_k(x_k)=\ol\phi_n(x_n)=\ol\phi_k\ol\phi_{nk}(x_n).$ Thus, the map $\ol\phi_k$ maps the submodule of $K_0^{\gr}(R_k)$ generated by $\ol\phi_{nk}(x_n)-x_k$ to zero. By Lemma \ref{dir_lim_group_lemma}, there is $m>k,$ such that $\ol\phi_{km}(\ol\phi_{nk}(x_n)-x_k)=0$ which implies that $\ol\phi_{nm}(x_n)=\ol\phi_{km}(x_k).$ Since the connecting maps preserve the generating intervals and $x_k$ is in the generating interval of $R_k,$ we have that $ \ol\phi_{nm}(x_n)$ is in the generating interval of $R_m.$ Thus, $0\leq \ol\phi_{nm}(x_n)\leq [1_{R_m}].$
\end{proof}

\begin{proposition}\label{matricial_to_ultramatricial}
Let $A$ be a $\Gamma$-graded $*$-field with enough unitaries, let $R$ be a graded matricial $*$-algebra over $A,$ and let $S=\varinjlim S_m$ be a graded ultramatricial $*$-algebra over $A$ with graded translational $*$-maps $\psi_m: S_m\to S,$ $m=1,2,\ldots$ If $f:K^{\gr}_0(R)\to K^{\gr}_0(S)$ is a contractive $\Z[\Gamma]$-module homomorphism, then there is a positive integer $m$ and a graded $*$-homomorphism $\phi: R\to S_m$ such that $K^{\gr}_0(\psi_m\phi)=f$.     
\end{proposition}
\begin{proof} 
Throughout the proof, we use $\overline{h}$ for $K_0^{\gr}(h)$ for any graded $A$-algebra $*$-homomorphism $h$ just as in the previous proof.  

Let $R=\bigoplus_{i=1}^n\M_{p(i)}(A)(\gamma^i_1,\dots,\gamma^i_{p(i)})$ and let $e_{kl}^i$, $i=1, \ldots, n$,  $k,l=1,\ldots, p(i),$ be the elements of $R$ such that $\pi_i(e^i_{kl})$ are the standard matrix units of $\M_{p(i)}(A)(\gamma^i_1,\dots,\gamma^i_{p(i)})$ for every $i=1,\ldots, n,$ where $\pi_i$ is a projection onto the $i$-th component of $R.$ Let $\psi_{lm}: S_l\to S_m$ be the connecting graded $*$-homomorphisms (not necessarily unital) for every $l\leq m.$
By the definition of the direct limit, there is a positive integer $l$ such that $f([1_R])$ and $f([e^i_{11}]), i=1,\ldots, n$ are all in $\ol\psi_l(K^{\gr}_0(S_l)).$ Thus, there are idempotent matrices $g_l$ and $g_{li}$ over $S_l$ such that 
\[\ol\psi_l([g_l])=f([1_R]) \;\; \text{ and }  \;\; \ol\psi_l([g_{li}])=f([e^i_{11}])\;\mbox{ for }\;i=1, \ldots, n.\] 
Since $[1_R]=\sum_{i=1}^n \sum_{k=1}^{p(i)}(\gamma^i_1-\gamma^i_k) [e^i_{11}]$ by Lemma \ref{matrix_units_lemma}, we have that 
\[\ol\psi_l\Big([g_l]-\sum_{i=1}^n \sum_{k=1}^{p(i)}(\gamma^i_1-\gamma^i_k)[g_{li}]\Big)=0.\]
Let $H$ be the finitely generated $\Z[\Gamma]$-submodule of $K^{\gr}_0(S_l)$ generated by the element $[g_l]-\sum_{i=1}^n \sum_{k=1}^{p(i)}(\gamma^i_1-\gamma^i_k)[g_{li}].$
By Lemma \ref{dir_lim_group_lemma} applied to $H$ and the inclusion of $H$ into $K^{\gr}_0(S_l),$ 
there is $m'> l$ such that $\ol\psi_{lm'}$ is zero on $H.$ Hence $\ol\psi_{lm'}\Big([g_l]-\sum_{i=1}^n \sum_{k=1}^{p(i)}(\gamma^i_1-\gamma^i_k)[g_{li}]\Big)=0.$

For $[g_{m'}]=\ol\psi_{lm'}([g_l]),$ we have that
\[0\leq \ol\psi_{m'}([g_{m'}])=\ol\psi_{m'}\ol\psi_{lm'}([g_l])=\ol\psi_{l}([g_l])=f([1_R])\leq [1_{S^u}].\]
Using Lemma \ref{unit_preservation}, we can find $m> m'$ such that 
\[0\leq \ol\psi_{m'm}([g_{m'}])\leq [1_{S_m}].\]

Define a map $g: K^{\gr}_0(R)\to K^{\gr}_0(S_m)$ by letting 
\[g([e^i_{11}])=\ol\psi_{lm}([g_{li}])\]
and extending this map to a $\Z[\Gamma]$-homomorphism. This construction ensures that $g$ is order-preserving.
It is also contractive since 
\begin{align*}
g([1_R]) &=g\Big(\sum_{i=1}^n \sum_{k=1}^{p(i)}(\gamma^i_1-\gamma^i_k) [e^i_{11}]\Big) =\sum_{i=1}^n \sum_{k=1}^{p(i)}(\gamma^i_1-\gamma^i_k)g([e^i_{11}])
=\sum_{i=1}^n \sum_{k=1}^{p(i)}(\gamma^i_1-\gamma^i_k)\ol\psi_{lm}([g_{li}])\\
& =\ol\psi_{m'm}\left(\sum_{i=1}^n \sum_{k=1}^{p(i)}(\gamma^i_1-\gamma^i_k)\ol\psi_{lm'}([g_{li}])\right)
=\ol\psi_{m'm}\ol\psi_{lm'}([g_l]) =\ol\psi_{m'm}[g_{m'}]\leq [1_{S_m}].
\end{align*}

The relation $\ol\psi_m g=f$ holds since $\ol\psi_mg([e_{11}^i])=\ol\psi_m\ol\psi_{lm}([g_{li}])= \ol \psi_l([g_{li}])=f([e^i_{11}]).$ 
We can apply Theorem~\ref{fullness} to matricial algebras $R$ and $S_m$ and the homomorphism $g$ to obtain a graded $*$-homomorphism $\phi: R\to S_m$ such that $\ol \phi=g.$ Thus, 
$\ol\psi_m\ol\phi=\ol\psi_m g=f.$
\end{proof}

We prove the main result of this section now. 

\begin{theorem}\label{ultramatricial}
Let $A$ be a $\Gamma$-graded $*$-field which has enough unitaries, and such that 
$A_0$ is 2-proper and $*$-pythagorean. 
Let $R$ and $S$ be (possibly non-unital) graded ultramatricial $*$-algebras over $A.$ If the map $f: K^{\gr}_0(R)\to K^{\gr}_0(S)$ is a contractive $\Z[\Gamma]$-module isomorphism, then there is a graded $A$-algebra $*$-isomorphism $\phi: R\to S$ such that $K^{\gr}_0(\phi)=f.$
\end{theorem}
\begin{proof}  
Throughout the proof, we use $\overline{h}$ to denote $K_0^{\gr}(h)$ for any graded $A$-algebra $*$-homomorphism $h.$ 

Let $\phi_{nm}: R_n\to R_m$ and $\psi_{nm}: S_n\to S_m$ for $n\leq m$ denote connecting graded $*$-homomorphisms and $\phi_n: R_n\to R$ and $\psi_n: S_n\to S$ denote  translational graded $*$-homomorphisms such that $R=\bigcup_n \phi_n(R_n)$ and $S=\bigcup_n \psi_n(R_n)$. We construct two increasing sequences of integers $n(1)< n(2)<\dots$ and $m(1)< m(2)<\dots$
and graded $*$-homomorphisms $\rho_i: R_{n(i)}\to S_{m(i)}$ and $\sigma_i: S_{m(i)}\to R_{n(i+1)}$ for $i=1, 2,\ldots$ such that the following relations hold, two on the algebra level and two on the $K^{\gr}_0$-group level.

\[\mbox{(1)}_i\;\;\;\sigma_i\rho_i=\phi_{n(i) n(i+1)}, \hskip3.2cm\mbox{(2)}_i\;\;\;\rho_{i+1}\sigma_i = \psi_{m(i) m(i+1)},\] 
\begin{center}
$\xymatrix{
{R_{n(i)}}\ar[r]^{\rho_i}\ar[dr]_{\phi_{n(i) n(i+1)}} & {S_{m(i)}}\ar[d]^{\sigma_i} \\ 
 & {R_{n(i+1)}}
}$\hskip3.4cm
$\xymatrix{
{S_{m(i)}}\ar[r]^{\sigma_i}\ar[dr]_{\psi_{m(i) m(i+1)}} & {R_{n(i+1)}}\ar[d]^{\rho_{i+1}} \\ 
 & {S_{m(i+1)}}
}$
\end{center}
\[\mbox{(3)}_i\;\;\;\overline \psi_{m(i)}\rho_i=f \overline \phi_{n(i)},\hskip3.4cm \mbox{(4)}_i\;\;\;\overline \phi_{n(i+1)}\sigma_i=f^{-1} \overline \psi_{m(i)}.\]
\begin{center}
$\xymatrix{
{K^{\gr}_0(R_{n(i)})}\ar[r]^{\overline\phi_{n(i)}}\ar[d]^{\overline \rho_i} & {K^{\gr}_0(R)}\ar[d]^{f} \\ 
{K^{\gr}_0(S_{m(i)})} \ar[r]^{\overline\psi_{m(i)}} & {K^{\gr}_0(S)}
}$\hskip3.4cm
$\xymatrix{
{K^{\gr}_0(S_{m(i)})}\ar[r]^{\overline \psi_{m(i)}}\ar[d]^{\overline\sigma_i} & {K^{\gr}_0(S)}\ar[d]^{f^{-1}} \\ 
{K^{\gr}_0(R_{n(i+1)})} \ar[r]^{\overline \phi_{n(i+1)}} & {K^{\gr}_0(R)}
}$
\end{center}
 
By (1)$_i$ and (2)$_i,$ we have $\psi_{m(i) m(i+1)}\rho_i=\rho_{i+1}\sigma_i\rho_i=\rho_{i+1} \phi_{n(i) n(i+1)},$ and  by (2)$_i$ and (1)$_{i+1},$ we have
$\phi_{n(i+1) n(i+2)}\sigma_i=\sigma_{i+1}\rho_{i+1}\sigma_i=\sigma_{i+1} \psi_{m(i) m(i+1)}.$ Thus we have the following commutative diagrams.  
\begin{center}
$\xymatrix{
{R_{n(i)}}\ar[r]^{\phi_{n(i) n(i+1)}}\ar[d]^{\rho_i} & {\;\;\;R_{n(i+1)}}\ar[d]^{\rho_{i+1}} \\ 
{S_{m(i)}} \ar[r]^{\psi_{m(i) m(i+1)}} & {\;\;\;S_{m(i+1)}}
}$\hskip3.4cm
$\xymatrix{
{S_{m(i)}}\ar[r]^{\psi_{m(i) m(i+1)}}\ar[d]^{\sigma_i} & {\;\;\;S_{m(i+1)}}\ar[d]^{\sigma_{i+1}} \\ 
{R_{n(i+1)}} \ar[r]^{\phi_{n(i+1) n(i+2)}} & {\;\;\;R_{n(i+2)}}
}$
\end{center}
 
First, we define $n(1), m(1), n(2)$ and maps $\rho_1, \sigma_1, \rho_2$ and show that all four relations hold for $i=1.$

{\bf Defining $\mathbf{\rho_1}$}. Let $n(1)=1$ and consider the composition of maps $f \overline \phi_1: K^{\gr}_0(R_1)\to K^{\gr}_0(S).$ By Proposition \ref{matricial_to_ultramatricial}, there exists a positive integer $m(1)$ and a graded $*$-homomorphism $\rho_1: R_1\to S_{m(1)}$ such that
$\overline \psi_{m(1)}\ol\rho_1=f \overline \phi_1,$ i.e. such that the relation (3)$_1$ holds. Note that Proposition \ref{matricial_to_ultramatricial} uses the assumption that $A$ has enough unitaries. 

{\bf Defining $\mathbf{\sigma_1}.$} 
Applying the same argument and Proposition \ref{matricial_to_ultramatricial} again to $f^{-1} \overline \psi_{m(1)}: K^{\gr}_0(S_{m(1)})\to K^{\gr}_0(R),$ we obtain a positive integer $n'(2)$ and a graded $*$-homomorphism $\sigma'_1: S_{m(1)}\to R_{n'(2)}$ such that $f^{-1} \overline \psi_{m(1)}=\ol\phi_{n'(2)}\ol{\sigma'_1}$.
In addition, 
\[\ol\phi_{n'(2)}\overline{\sigma'_1}\ol\rho_1=f^{-1}\overline\psi_{m(1)}\overline\rho_1=f^{-1}f \overline\phi_1=\overline\phi_1=\overline\phi_{n'(2)}\ol\phi_{1 n'(2)}
\]
so $\overline\phi_{n'(2)}\left(\overline{\sigma'_1}\ol\rho_1-\overline\phi_{1 n'(2)}\right)=0.$ By Lemma \ref{dir_lim_group_lemma}, there is integer $n(2)> n'(2)$ such that 
\[\overline\phi_{n'(2)n(2)}\left(\overline{\sigma'_1}\ol\rho_1-\overline\phi_{1 n'(2)}\right)=0.\]
Thus we have that $\ol\phi_{n'(2)n(2)}\ol{\sigma'_1}\ol\rho_1=\ol\phi_{n'(2)n(2)}\ol\phi_{1 n'(2)}= \ol\phi_{1 n(2)}.$ By Theorem \ref{faithfulness}, there is a unitary $x\in R_{n(2)}$ of degree zero such that the inner $*$-automorphism $\theta: a\mapsto xax^{-1}$ of $R_{n(2)}$ 
satisfies $\phi_{1 n(2)}=\theta\phi_{n'(2)n(2)}\sigma'_1\rho_1.$ The assumptions that $A_0$ is 2-proper and $*$-pythagorean guarantee that $\theta$ is a $*$-map on entire $R_{n(2)}$ not just on the image of $\phi_{1 n(2)}.$  

Define a graded $*$-homomorphism $\sigma_1: S_{m(1)}\to R_{n(2)}$ by $\sigma_1= \theta\phi_{n'(2)n(2)}\sigma'_1.$ The relation (1)$_1$ follows since $\sigma_1\rho_1=\theta\phi_{n'(2)n(2)}\sigma'_1\rho_1= \phi_{1 n(2)}.$ The relation (4)$_1$ also holds since 
\[\ol\phi_{n(2)}\ol\sigma_1=\ol\phi_{n(2)}\ol\theta \, \ol\phi_{n'(2)n(2)}\ol{\sigma'_1}=
\ol\phi_{n(2)}\ol\phi_{n'(2)n(2)}\ol{\sigma'_1}=\ol\phi_{n'(2)}\ol{\sigma'_1}=f^{-1}\overline\psi_{m(1)}.\]

{\bf Defining $\mathbf{\rho_2}$}. By Proposition \ref{matricial_to_ultramatricial}, for $f \overline\phi_{n(2)}: K^{\gr}_0(R_{n(2)})\to K^{\gr}_0(S),$  there is a positive integer $m'(2)>m(1)$ and a graded $*$-homomorphism $\rho'_2: R_{n(2)}\to S_{m'(2)}$ such that 
$\ol\psi_{m'(2)}\overline{\rho'_2}=f\ol\phi_{n(2)}.$ Since  
$\ol\psi_{m'(2)}\overline{\rho'_2}\ol\sigma_1=f\ol\phi_{n(2)}\ol\sigma_1=ff^{-1}\ol\psi_{m(1)}=\ol\psi_{m(1)}=\ol\psi_{m'(2)}\ol\psi_{m(1) m'(2)},$ we have that 
$\overline\psi_{m'(2)}\left(\overline{\rho'_2}\ol\sigma_1-\ol\psi_{m(1) m'(2)}\right)=0.$ By Lemma \ref{dir_lim_group_lemma},  there is an integer $m(2)> m'(2)$ such that $$\ol\psi_{m'(2)m(2)}\left(\overline{\rho'_2}\ol\sigma_1-\ol\psi_{m(1) m'(2)}\right)=0.$$ Thus, 
$\ol\psi_{m'(2)m(2)}\overline{\rho'_2}\ol\sigma_1=\ol\psi_{m'(2)m(2)}\ol\psi_{m(1) m'(2)} = \ol\psi_{m(1) m(2)}.$ By Theorem~\ref{faithfulness},  
there is a unitary $y\in S_{m(2)}$ of degree zero producing an inner $*$-automorphism $\theta'$ of $S_{m(2)}$ such that $\psi_{m(1) m(2)}=\theta'\psi_{m'(2)m(2)}\rho'_2\sigma_1.$ Define a graded $*$-homomorphism $\rho_2: R_{n(2)}\to S_{m(2)}$ by $\rho_2= \theta'\psi_{m'(2)m(2)}\rho'_2.$ The relation (2)$_1$ follows since $\rho_2\sigma_1=\theta'\psi_{m'(2)m(2)}\rho'_2\sigma_1= \psi_{m(1) m(2)}$ showing that all four relations hold for $i=1.$ 
 
{\bf The Inductive Step.} Continuing by induction, we obtain two increasing sequences of integers $n(1)< n(2)< \ldots$ and $m(1)<m(2)<\ldots$ and maps $\rho_i$ and $\sigma_i$ for which the four diagrams (1)$_i$ to (4)$_i$ commute for all $i=1,2,\ldots$.

{\bf The Final Step.}  We define the maps $\rho: R\to S$ and $\sigma: S\to R$ now. Since $R=\bigcup_{i} \phi_{n(i)}(R_{n(i)})$, every element $r\in R$ is in the image of some map $\phi_{n(i)}.$ If $r=\phi_{n(i)}(r_i),$ define $\rho(r)=\psi_{m(i)}\rho_i(r_i).$ 
This automatically produces the relation $\rho\phi_{n(i)}=\psi_{m(i)}\rho_i.$ We define $\sigma$ analogously so that  $\sigma\psi_{m(i)}=\phi_{n(i+1)}\sigma_i$. 
The maps $\rho$ and $\sigma$ are graded $A$-algebra $*$-homomorphisms by their definitions and we have the following diagram.    
\[
\xymatrix{
{R_{n(1)}\;}\ar[r]^{\phi_{n(1) n(2)}}\ar[d]^{\rho_1} & {\;\;R_{n(2)}\;\;}\ar[d]^{\rho_2}\ar[r]^{\phi_{n(2) n(3)}} &  {\;R_{n(3)}\;\;}\ar[d]^{\rho_3}\ar[r]^{\phi_{n(3) n(4)}} & {\ldots}\ar[r]\ar[d] & {\;R\;} \ar[d]^{\sigma} \\ 
{S_{m(1)}\;} \ar[r]_{\psi_{m(1) m(2)}}\ar[ur]^{\sigma_1} & {\;\;S_{m(2)}\;\;}\ar[r]_{\psi_{m(2) m(3)}}\ar[ur]^{\sigma_2} &  {\;S_{m(3)}\;\;\;}\ar[r]_{\psi_{m(3) m(4)}}\ar[ur]^{\sigma_3} & {\ldots}\ar[r]\ar[ur] & {\;S\;} \ar@<-.5ex>[u]^{\rho}
}
\]

Lastly, we check that $\rho$ and $\sigma$ are mutually inverse. If $r\in R$ is such that $r=\phi_{n(i)}(r_i),$ then, using  (1)$_i$ we check that
\[\sigma\rho(r)=\sigma\rho\phi_{n(i)}(r_i)=\sigma\psi_{m(i)}\rho_i(r_i) =\phi_{n(i+1)}\sigma_i\rho_i(r_i)=\phi_{n(i+1)}\phi_{n(i) n(i+1)}(r_i)=\phi_{n(i)}(r_i)=r.\] Similarly, if $s\in S$ is such that 
$s=\psi_{m(i)}(s_i),$ then,
using (2)$_{i}$ we check that
\[\rho\sigma(s)=\rho\sigma\psi_{m(i)}(s_i)=\rho\phi_{n(i+1)}\sigma_i(s_i) = \psi_{m(i+1)}\rho_{i+1}\sigma_i(s_i)=\psi_{m(i+1)}\psi_{m(i)m(i+1)}(s_i)=\psi_{m(i)}(s_i)=s.\] 

The commutativity of the diagrams (3)$_i$ and (4)$_i$ for every $i$ implies that $\overline\rho=f$ and $\overline \sigma=f^{-1}.$ 
\end{proof}

We wonder whether the assumptions of Theorem \ref{ultramatricial} on the graded $*$-field can be weakened. 
In other words, we wonder whether Theorem \ref{fullness} holds without assuming that the graded $*$-field has enough unitaries, whether Theorem \ref{faithfulness} holds without the assumptions that the 0-component of the graded $*$-field is 2-proper and $*$-pythagorean, and, consequently, whether these assumptions can be removed from Theorem \ref{ultramatricial}. 

Theorem \ref{ultramatricial} has the following direct corollary. 

\begin{corollary}\label{isomorphism_corollary}
Let $A$ be a $\Gamma$-graded $*$-field which has enough unitaries and such that $A_0$ is 2-proper and $*$-pythagorean.    
If $R$ and $S$ are graded ultramatricial $*$-algebras over $A,$ then $R$ and $S$ are isomorphic as graded rings if and only if $R$ and $S$ are isomorphic as graded $*$-algebras. 

If the involutive structure is not considered and $A$ is any graded field, two graded ultramatricial algebras over $A$ are isomorphic as graded rings if and only if they are isomorphic as graded algebras.
\end{corollary}
\begin{proof}
One direction of the first claim follows directly from Theorem \ref{ultramatricial} since every ring isomorphism $R\to S$ induces a contractive $\Z[\Gamma]$-module isomorphism of the $K_0^{\gr}$-groups. The other direction is trivial.  

The last sentence follows from Corollary \ref{noninvolutive_fullness}, the proof of Theorem \ref{ultramatricial}, and the fact that the conditions (1) and (4) of Theorem \ref{faithfulness} are equivalent for any graded field $A.$ 
\end{proof}

Disregarding the involutive structure of the field $A$ and imposing no assumption on the graded field $A$, we obtain \cite[Theorem 5.2.4]{Roozbeh_graded_ring_notes} as a direct corollary of Theorem \ref{ultramatricial}. 

\begin{corollary}\cite[Theorem 5.2.4]{Roozbeh_graded_ring_notes}\label{noninvolutive_ultramatricial}
Let $A$ be a $\Gamma$-graded field and let $R$ and $S$ be graded ultramatricial algebras over $A.$ If $f: K^{\gr}_0(R)\to K^{\gr}_0(S)$ is a contractive $\Z[\Gamma]$-module isomorphism, then there is a graded $A$-algebra isomorphism $\phi: R\to S$ such that $K^{\gr}_0(\phi)=f.$ 
\end{corollary}

Considering grading by the trivial group, we obtain the following corollary of Theorem \ref{ultramatricial}.

\begin{corollary}\label{ultramatricial_trivial_grading}
Let $A$ be a 2-proper and $*$-pythagorean $*$-field and let $R$ and $S$ be ultramatricial $*$-algebras over $A.$ If $f: K_0(R)\to K_0(S)$ is a contractive $\Z$-module isomorphism, then there is an $A$-algebra $*$-isomorphism $\phi: R\to S$ such that $K_0(\phi)=f.$
\end{corollary}

This corollary implies \cite[Proposition 3.3]{Ara_matrix_rings}. Note that Corollary \ref{isomorphism_corollary} also implies \cite[Proposition 3.3]{Ara_matrix_rings}.

\begin{corollary}\cite[Proposition 3.3]{Ara_matrix_rings}
Let $A$ be a 2-proper and $*$-pythagorean $*$-field. If $R$ and $S$ are ultramatricial $*$-algebras over $A,$ then $R$ and $S$ are isomorphic as rings if and only if $R$ and $S$ are isomorphic as $*$-algebras.   
\end{corollary}

Considering the field of complex numbers and the complex conjugate involution with the trivial grade group, the classification theorem for AF $C^*$-algebras \cite[Theorem 4.3]{Elliott} follows from Corollary \ref{ultramatricial_trivial_grading}. 

\begin{corollary}\cite[Theorem 4.3]{Elliott}
Let $R$ and $S$ be AF $C^*$-algebras. If $f: K_0(R)\to K_0(S)$ is a contractive $\Z$-module isomorphism, then there is a $*$-isomorphism $\phi: R\to S$ such that $K_0(\phi)=f.$
\end{corollary}

Considering the grade group to be trivial and disregarding the involutive structure, we obtain \cite[Theorem 15.26]{Goodearl_book} as a direct corollary.
\begin{corollary}\cite[Theorem 15.26]{Goodearl_book}
Let $A$ be a field and let $R$ and $S$ be ultramatricial algebras over $A.$ If $f: K_0(R)\to K_0(S)$ is a contractive $\Z$-module isomorphism, then there is an $A$-algebra isomorphism $\phi: R\to S$ such that $K_0(\phi)=f.$
\end{corollary}

\subsection{Graded rings of matrices of infinite size}
We conclude this section with a short discussion on the grading of the matrix rings of infinite size which we shall use in section \ref{subsection_no-exit}. Let $A$ be a $\Gamma$-graded ring and let $\kappa$ be an infinite cardinal. We let $\M_\kappa (A)$ denote the ring of infinite matrices over $A$, having rows and columns indexed by $\kappa$, with only finitely many nonzero entries. If $\ol\gamma$ is any function $\kappa\to \Gamma,$ we can consider it as an element of $\Gamma^\kappa.$ For any such $\ol\gamma,$ we let $\M_\kappa(A)(\ol\gamma)$ denote the $\Gamma$-graded ring $\M_\kappa(A)$ with the $\delta$-component consisting of the matrices $(a_{\alpha\beta}),$ $\alpha, \beta\in \kappa,$ such that $a_{\alpha\beta}\in A_{\delta+\ol\gamma(\beta)-\ol\gamma(\alpha)}.$ If $A$ is a graded $*$-ring,  the graded ring $\M_\kappa(A)(\ol\gamma)$ is a graded $*$-ring with the $*$-transpose involution. 

With these definitions, Proposition \ref{permutation_of_components} generalizes as follows. 

\begin{proposition}\label{permutation_general}
Let $A$ be a $\Gamma$-graded $*$-ring, $\kappa$ a cardinal and $\ol\gamma\in \Gamma^\kappa$. 
\begin{enumerate}[\upshape(1)]
\item If $\delta$ is in $\Gamma$, $\pi$ is a bijection $\kappa\to \kappa,$ and $\ol\gamma\pi+\delta$ denotes the map $\kappa\to \Gamma$ given by $\alpha\mapsto \ol\gamma(\pi(\alpha))+\delta,$ then the matrix rings
\begin{center}
$\M_\kappa (A)(\ol \gamma)\;\;$ and $\;\;\M_\kappa (A)(\ol\gamma\pi+\delta)$
\end{center}
are graded $*$-isomorphic.

\item If $\ol\delta \in \Gamma^\kappa$ is such that there is a unitary element  $a_\alpha\in A_{\ol\delta(\alpha)}$ for every $\alpha\in \kappa,$ then the matrix rings
\begin{center}
$\M_\kappa (A)(\ol \gamma)\;\;$ and $\;\;\M_\kappa (A)(\ol \gamma+\ol \delta)$
\end{center}
are graded $*$-isomorphic.
\end{enumerate}
\end{proposition}
\begin{proof}
The proof parallels the proof of Proposition \ref{permutation_of_components}. To prove the first part, note that adding $\delta$ to each of the shifts does not change the matrix ring. For $\alpha,\beta\in \kappa,$ let $e_{\alpha \beta}$ denote the matrix with the field identity element in the $(\alpha, \beta)$ spot and zeros elsewhere. Define a map $\phi$ by mapping $e_{\alpha \beta}$ to $e_{\pi^{-1}(\alpha)\pi^{-1}(\beta)}$ and extending it to an $A$-linear map. Define also a map $\psi$ by mapping $e_{\alpha \beta}$ to $e_{\pi(\alpha)\pi(\beta)}$ and extending it to an $A$-linear map. The maps $\phi$ and $\psi$ are mutually inverse $*$-homomorphisms. Moreover,  $\phi$ and $\psi$ are graded maps by the same argument as the one used in Proposition \ref{permutation_of_components}.   

To prove the second part, let $a_\alpha \in A_{\ol\delta(\alpha)}$ be a unitary element for every $\alpha\in\kappa$. Consider a map $\phi$ given by $e_{\alpha \beta}\mapsto a^*_\alpha a_\beta e_{\alpha\beta}$ for every $\alpha,\beta\in \kappa$ and extend it to an $A$-linear map. This extension is clearly a graded $*$-homomorphism with the inverse given by the $A$-linear extension of the map $e_{\alpha \beta}\mapsto a_\alpha a^*_\beta e_{\alpha\beta}.$ 
\end{proof}

If $A$ is a graded field and $\omega$ the first infinite ordinal, then the algebra $\M_\omega(A)(\ol\gamma)$ is a graded ultramatricial algebra since $\M_\omega(A)(\ol\gamma)$ is the directed union, taken over the finite subsets $I$ of $\omega,$ of the subalgebras which have nonzero entries in the rows and columns which correspond to the elements of $I$.

\section{Isomorphism Conjecture for a class of Leavitt path algebras}\label{section_LPAs}

In this section, we use Theorem \ref{ultramatricial} to show that the graded version of the Generalized Strong Isomorphism Conjecture holds for a class of Leavitt path algebras. We start with a quick review of the main concepts. Just as in the previous section, rings are not assumed to be unital and homomorphisms are not assumed to be unit-preserving.

\subsection{Leavitt path algebras.}\label{subsection_LPAs}
Let $E=(E^0, E^1, \so, \ra)$  be a directed graph where $E^0$ is
the set of vertices, $E^1$ the set of edges, and $\so,\ra: E^1
\to E^0$ are the source and the range maps, respectively. For brevity, we refer to directed
graphs simply as graphs.

A vertex $v$ of a graph $E$ is said to be {\em regular} if the set $\so^{-1}(v)$ is nonempty
and finite. A vertex $v$ is called a {\em sink} if $\so^{-1}(v)$ is empty.  
A graph $E$ is \emph{row-finite} if sinks are the only vertices which are not regular, it is
\emph{finite} if $E$ is row-finite and $E^0$
is finite (in which case $E^1$ is necessarily finite as well), and it is {\em
countable} if both $E^0$ and $E^1$ are countable. 

A {\em path} $p$ in a graph $E$ is a finite sequence of edges
$p=e_1\ldots e_n$ such that $\ra(e_i)=\so(e_{i+1})$ for $i=1,\dots,n-1$. Such
path $p$ has length $|p|=n.$  The maps $\so$ and $\ra$ extend
to paths by $\so(p)=\so(e_1)$ and $\ra(p)=\ra(e_n)$ respectively and $\so(p)$ and $\ra(p)$ are the \emph{source} and the \emph{range} of $p$ respectively. We consider a vertex $v$ to be a \emph{trivial}
path of length zero with $\so(v)=\ra(v)=v$. A path $p = e_1\ldots e_n$ is said to be \emph{closed} if
$\so(p)=\ra(p)$. 
If $p$ is a nontrivial path, and if $v=\so(p)=\ra(p)$, then $p$ is called a \emph{closed path based at} $v$. If  $p = e_1\ldots e_n$  is a closed path and $\so(e_i) \neq \so(e_j)$ for every $i \neq j$, then $p$ is called a \emph{cycle}.  A graph $E$ is said to be {\em no-exit} if
$\so^{-1}(v)$ has just one element for every vertex $v$ of every cycle. 

An infinite path of a graph $E$ is a sequence of edges $e_1e_2\ldots$ such
that $\ra(e_i)=\so(e_{i+1})$ for all $i=1,2,\ldots$. An 
infinite path $p$ is an \emph{infinite sink} if it has
no cycles or exits. An infinite path $p$ \emph{ends in a sink} if there is $n\geq 1$
such that the subpath
$e_ne_{n+1}\hdots$ is an infinite sink, and $p$
\emph{ends in a cycle} if there is $n\geq 1$ and a
cycle $c$ of positive length such that the subpath
$e_ne_{n+1}\hdots$ is equal to the path $cc\hdots$.  

For a graph $E,$ consider the extended graph of $E$ to be the graph with the
same vertices and with edges $E^1\cup \{e^*\ |\ e\in E^1\}$ where
the range and source relations are the same as in $E$ for $e\in
E^1$ and $\so(e^*)=\ra(e)$  and $\ra(e^*)=\so(e)$ for the added edges.
Extend the map $^*$ to all the paths by defining $v^*=v$ for
all vertices $v$ and $(e_1\ldots e_n)^*=e_n^*\ldots e_1^*$ for all paths
$p=e_1\ldots e_n.$ If $p$ is a path, we refer to elements of the form
$p^*$ as ghost paths. Extend also the maps $\so$ and $\ra$ to ghost paths by
$\so(p^*)=\ra(p)$ and $\ra(p^*)=\so(p)$.  

For a graph $E$ and a field $K$, the \emph{Leavitt path algebra of $E$ over $K$}, denoted by $L_K(E)$, is the free $K$-algebra generated by the set  $E^0\cup E^1\cup\{e^*\ |\ e\in E^1\}$ satisfying the following relations for all vertices $v,w$ and edges $e,f$.
\begin{itemize}
\item[(V)]  $vw = \delta_{v,w}v$,

\item[(E1)]  $\so(e)e=e\ra(e)=e$,

\item[(E2)] $\ra(e)e^*=e^*\so(e)=e^*$,

\item[(CK1)] $e^*f=\delta _{e,f}\ra(e)$,

\item[(CK2)] $v=\sum_{e\in \so^{-1}(v)} ee^*$ for every regular vertex $v$.
\end{itemize}

The first four axioms imply that every
element of $L_K(E)$ can be represented as a sum of the form $\sum_{i=1}^n
a_ip_iq_i^*$ for some $n$, paths $p_i$ and $q_i$, and elements $a_i\in K,$ for 
$i=1,\ldots,n.$ Also, these axioms imply that $L_K(E)$ is a ring with identity if and only if $E^0$ is finite (in this case the identity is the sum of all the elements of $E^0$). If $E^0$ is not finite, the finite sums of vertices are the local units for the algebra $L_K(E).$   

If $K$ is a field with involution $*$ (and there is always at
least one such involution, the identity), the Leavitt path algebra $L_K(E)$ becomes an involutive algebra by 
$(\sum_{i=1}^n a_ip_iq_i^*)^* =\sum_{i=1}^n a_i^*q_ip_i^*$ for $a_i\in K$ and paths $p_i, q_i, i=1,\ldots, n.$

For an arbitrary abelian group $\Gamma$, one can equip $L_K(E)$ with a $\Gamma$-graded structure as follows. For any map $w: E^1\rightarrow \Gamma$ called the \emph{weight} map, define $w(e^*)=-w(e)$ for $e \in E^1,$ $w(v)=0$ for $v\in E^0,$ and extend $w$ to paths by letting $w(p)=\sum_{i=1}^n w(e_i)$ if $p$ is a path consisting of edges $e_1\ldots e_n.$ Since all of the Leavitt path algebra axioms agree with the weight map, if we let $L_K(E)_\gamma$ be the $K$-linear span of the elements $p q^*$ where $p, q$ are paths with $w(p)-w(q)=\gamma,$ then $L_K(E)$ is a $\Gamma$-graded $K$-algebra with the $\gamma$-component equal to $L_K(E)_\gamma.$ It can be shown that this grading is well-defined (i.e. that it  is independent of the representation of an element as $\sum_i a_i p_i q_i^*$ in the above definition of $L_K(E)_\gamma$). This grading is such that $L_K(E)_\gamma^*=L_K(E)_{-\gamma}$ for every $\gamma\in \Gamma$ and so $L_K(E)$ is a graded $*$-algebra. 

The \emph{natural} grading of a Leavitt path algebra is obtained by taking $\Gamma$ to be $\Z$ and the weight of any edge to be 1. Thus, the $n$-component $L_K(E)_n$ is  the $K$-linear span of the elements $pq^*$ where $p, q$ are paths with $|p|-|q|=n.$
Unless specified otherwise, we always assume that $L_K(E)$ is graded naturally by $\Z.$

\subsection{A class of no-exit row-finite graphs}\label{subsection_no-exit}

For any countable, row-finite graph $E$ and a field $K,$ the following conditions are equivalent by \cite[Theorem 3.7]{AAPM}. 
\begin{enumerate} 
\item $E$ is a no-exit graph such that every infinite path ends either in a cycle or a sink.
\item $L_K(E) \cong \bigoplus_{i\in I} \M_{\kappa_i} (K) \oplus \bigoplus_{i \in J} \M_{\kappa_i} (K[x,x^{-1}]),$ where $I$ and $J$ are countable sets, and each $\kappa_i$ is a countable cardinal.
\end{enumerate}
In \cite[Corollary 32]{Zak_Lia}, it is shown that the assumption on countability of the graph $E$ can be dropped (in which case $I,$ $J,$ and $\kappa_i$ may be of arbitrary cardinalities) and that the isomorphism in condition (2) can be taken to be a $*$-isomorphism. However, since the ultramatricial algebras are defined as countable direct limits of matricial algebras, we prove the main result of this section, Theorem \ref{classification}, by considering just countable graphs.  

Our first goal is to show the following.  
\begin{proposition} \label{no-exit_graphs}
If $K$ is a field and $E$ a row-finite, no-exit graph such that every infinite path ends either in a cycle or a sink, then 
$L_K(E)$ is graded $*$-isomorphic to the algebra 
$$\bigoplus_{i\in I} \M_{\kappa_i} (K)(\ol\alpha^i) \oplus \bigoplus_{i \in J} \M_{\mu_i} (K[x^{n_i},x^{-n_i}])(\ol\gamma^i)$$
where $I$ and $J$ are some sets, $\kappa_i$ and $\mu_i$ cardinals, $n_i$ positive integers, $\ol\alpha^i\in \Z^{\kappa_i}$ for $i\in I,$ and $\ol\gamma^i\in \Z^{\mu_i}$ for $i\in J.$
\end{proposition}
\begin{proof}
The proof adapts the idea of the proofs of \cite[Theorem 3.7]{AAPM} and \cite[Corollary 32]{Zak_Lia} to the grading setting. 

Let $E$ be a row-finite, no-exit graph such that every infinite path ends either in a cycle or a sink and consider $K$ to be trivially graded by $\Z.$ Let $\{s_i\}_{i \in I_1}$ be all the finite sinks in $E^0$, and let $\{u_i\}_{i \in I_2}$ be all the infinite sinks in $E$ (where distinct $u_i$ have no edges in common, and where the starting vertex of each $u_i$ is fixed, though it can be chosen arbitrarily). Let $\{c_i\}_{i \in J}$ be all the cycles which contain at least one edge and where the starting vertex of each $c_i$ is a fixed, though arbitrary, vertex of $c_i.$ 

For each $i \in I_1,$ let $\{p_{ij}\}_{j \in L_i}$ be all the paths which end in $s_i.$ For each $i \in I_2,$ let $\{q_{ij}\}_{j \in M_i}$ be all the paths which end in $u_i$ such that $\ra(q_{ij})$ is the only vertex of both $q_{ij}$ and $u_i.$ For each $i \in J,$ let $\{r_{ij}\}_{j \in N_i}$ be all the paths which end in $\so(c_i)$ but do not contain $c_i.$ It can be shown (see \cite[Theorem 3.7]{AAPM} and \cite[Corollary 32]{Zak_Lia}) that $$\bigcup_{i \in I_1} \{p_{ij}p_{il}^* \mid j,l \in L_i\} \cup \bigcup_{i \in I_2} \{q_{ij}x_{ijl}q_{il}^* \mid j,l \in M_i\} \cup \bigcup_{i \in J} \{r_{ij}c_i^kr_{il}^* \mid j,l \in N_i, k\in \Z\}$$ is a basis for $L_K(E)$, where $c_i^k$ is understood to be $\so(c_i)$ when $k=0$ and $(c_i^*)^{-k}$ when $k < 0$, and where, writing each $u_i$ as $u_i = e_{i1}e_{i2}\hdots$ ($e_{ij} \in E^1$), we define
$$x_{ijl} = \left\{ \begin{array}{cl}
e_{ik}e_{i(k+1)}\hdots e_{i(m-1)} & \text{if } \ra(q_{ij}) = \so(e_{ik}), \ra(q_{il}) = \so(e_{im}), \text{ and } 1\leq k<m\\
e_{i(k-1)}^*e_{i(k-2)}^*\hdots e_{im}^* & \text{if } \ra(q_{ij}) = \so(e_{ik}), \ra(q_{il}) = \so(e_{im}), \text{ and } k>m\geq 1\\
\ra(q_{ij}) & \text{if } \ra(q_{ij})=\ra(q_{il})
\end{array}\right.$$
for $i\in I_2, j,l\in M_i.$ Note that $\deg(x_{ijl})=m-k$ in any of the three cases of the above formula.

Then one can define a map $$f: L_K(E)\to \bigoplus_{i\in I_1} \M_{|L_i|} (K) \oplus \bigoplus_{i \in I_2} \M_{|M_i|} (K)  \oplus \bigoplus_{i \in J} \M_{|N_i|} (K[x^{|c_i|},x^{-|c_i|}])\mbox{ by letting}$$  
\begin{center}
$f(p_{ij}p_{il}^*) = e^i_{jl} \in \M_{|L_i|} (K)$, $f(q_{ij}x_{ijl}q_{il}^*) = e^i_{jl} \in \M_{|M_i|} (K)$, and $f(r_{ij}c_i^kr_{il}^*) = x^ke^i_{jl} \in \M_{|N_i|} (K[x^{|c_i|},x^{-|c_i|}]),$
\end{center}
where $e^i_{jl}$ is the matrix unit in the corresponding matrix algebra, 
and then extending this map $K$-linearly to all of $L_K(E)$. Note that for all the basis elements we have
\begin{center}
\begin{tabular}{l}
$f((p_{ij}p_{il}^*)^*) = f(p_{il}p_{ij}^*) = e^i_{lj} = (e^i_{jl})^* = f(p_{ij}p_{il}^*)^*,$\\
$f((q_{ij}x_{ijl}q_{il}^*)^*) = f(q_{il}x_{ijl}^*q_{ij}^*) = e^i_{lj} = (e^i_{jl})^* = f(q_{ij}x_{ijl}q_{il}^*)^*,$ and\\
$f((r_{ij}c_i^kr_{il}^*)^*) = f(r_{il}c_i^{-k}r_{ij}^*) = x^{-k}e^i_{lj} = (x^ke^i_{jl})^* = f(r_{ij}c_i^kr_{il}^*)^*,$
\end{tabular}
\end{center}
so that the map $f$ is a $*$-homomorphism. 

We turn to the graded structure now. Let us define the elements
\begin{center}
\begin{tabular}{lll}
$\ol\alpha^i \in \Gamma^{|L_i|}$ by &$\ol\alpha^i(j)=|p_{ij}|,$&\\
$\ol\beta^i \in \Gamma^{|M_i|}$ by  &$\ol\beta^i(j) =|q_{ij}|-k $ &if $\ra(q_{ij})=\so(e_{i(k+1)})$ and \\
$\ol\gamma^i \in \Gamma^{|N_i|}$ by &$\ol\gamma^i(j)=|r_{ij}|.$&
\end{tabular}
\end{center}
and let us consider the graded algebra 
$$R:=\bigoplus_{i\in I_1} \M_{|L_i|} (K)(\ol\alpha^i) \oplus \bigoplus_{i \in I_2} \M_{|M_i|}(K)(\ol\beta^i)  \oplus \bigoplus_{i \in J} \M_{|N_i|} (K[x^{|c_i|},x^{-|c_i|}])(\ol\gamma^i).$$ 
If $e^i_{jl}$ is a matrix unit in  $\M_{|L_i|} (K)(\ol\alpha^i),$ then  $\deg(e^i_{jl})=\ol\alpha^i(j)-\ol\alpha^i(l)=|p_{ij}|-|p_{il}|=\deg(p_{ij}p_{il}^*)$
for all $i\in I_1,$ $j,l\in L_i.$ Similarly, if $e^i_{jl}$ is a matrix unit in $\M_{|N_i|}(K[x^{|c_i|},x^{-|c_i|}]),$ then  $ \deg(x^{k|c_i|}e^i_{jl})=k|c_i|+\ol\gamma^i(j)-\ol\gamma^i(l)=k|c_i|+|r_{ij}|-|r_{il}|=\deg(r_{ij}c_i^kr_{il}^*)$ for all  $i\in J,$ $j,l\in N_i,$ and $k\in \Z.$ If $e^i_{jl}$ is a matrix unit in $\M_{|M_i|}(K)(\ol\beta^i),$ then  $\deg(e^i_{jl})=\ol\beta^i(j)-\ol\beta^i(l)=|q_{ij}|-k-|q_{il}|+m=\deg(q_{ij}x_{ijl}q_{il}^*)$ since $\deg(x_{ijl})=m-k.$

This shows that the map $f$ is also a graded map. Hence, $f$ is a graded $*$-isomorphism of $L_K(E)$ and $R.$ 
\end{proof}

Before proving the main theorem of this section, let us look at some examples. Since \cite{Roozbeh_Israeli} and \cite{Roozbeh_Annalen} contain examples of Leavitt path algebras of {\em finite} graphs from the class which we consider (and a related class called polycephaly graphs), we concentrate on graphs with infinitely many vertices.  

\begin{example} \label{infinite_graphs}
Let us consider a graph with vertices $v_0, v_1\ldots$ and edges $e_0, e_1,\ldots$ such that $\so(e_n)=v_n$ and $\ra(e_n)=v_{n+1}$ for all $n=0,1\ldots$ (represented below).
\[\xymatrix{{\bullet^{v_0}} \ar [r] & {\bullet^{v_1}} \ar [r] &  {\bullet^{v_2}} \ar [r] & {\bullet^{v_3}} \ar@{.}[r] & }\] 
By taking $v_0$ as the source of the infinite sink, the Leavitt path algebra of this graph is graded $*$-isomorphic to  
$\M_\omega(K)(0, -1, -2, -3, \ldots)$ where $\omega$ is the first infinite ordinal. 

By taking $v_2,$ for example, as the source of the infinite sink, the Leavitt path algebra of this graph can also be represented by   
$\M_\omega(K)(2, 1, 0, -1, -2,\ldots).$ By Proposition \ref{permutation_general}, these two representations are graded $*$-isomorphic. 

Consider the graph below now. 
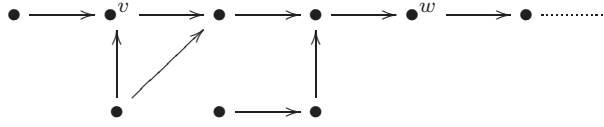
\[\xymatrix{{\bullet} \ar [r] & {\bullet^{v}} \ar [r]  & {\bullet} \ar [r] &  {\bullet} \ar [r] &  {\bullet^w} \ar [r] &
{\bullet} \ar@{.}[r]  & \\ & {\bullet}\ar[u] \ar[ur] & {\bullet}\ar[r] & {\bullet}\ar[u]
}\]
By taking $v$ as the source of the infinite sink, the Leavitt path algebra of this graph can be represented as $\M_\omega(K)(1, 1, 0, 0, 0, -1, -1, -2, -3, -4, \ldots).$ Similarly, considering $w,$ for example, as the source of the infinite sink, we obtain another graded $*$-isomorphic representation $\M_\omega(K)(4, 4, 3, 3, 3, 2, 2, 1, 0, -1, \ldots).$
\end{example}

The graphs from the previous example have just one (infinite) sink and so the graded Grothendieck groups of the corresponding Leavitt path algebras are isomorphic to $\Z[x,x^{-1}]$ as $\Z[x,x^{-1}]$-modules. From this example, it may appear that the graded Grothendieck group does not distinguish between Leavitt path algebras of these graphs. However, the identity map on $\Z[x,x^{-1}]\cong K^0_{\gr}(L_K(E))\to \Z[x,x^{-1}]\cong K^0_{\gr}(L_K(F)),$ where $E$ and $F$ are two graphs from Example \ref{infinite_graphs}, is not contractive. So, it turns out that $K_0^{\gr}$ is sensitive enough to distinguish the algebras which are not graded $*$-isomorphic. The next example illustrates this point in more details.  

\begin{example}\label{line_and_clock}
Let $E$ and $F$ be the graphs below and let $K$ be any field trivially graded by $\Z$. 
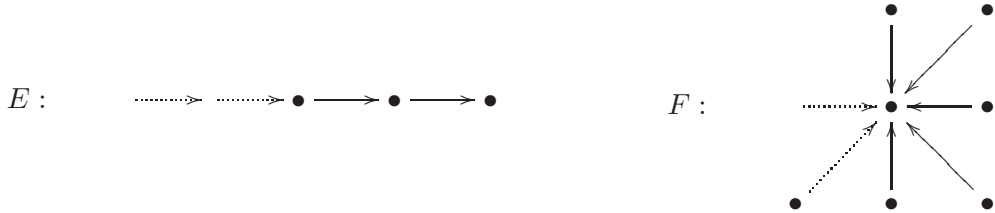
\[
\xymatrix{\\
E: &  \ar@{.>}[r] & \ar@{.>}[r] & \bullet \ar[r] & \bullet \ar[r] &  \bullet  
\\}\hskip2cm
\xymatrix{
&& \bullet \ar[d] & \bullet \ar[dl]\\
F: & \ar@{.>}[r] & \bullet  &  \bullet \ar[l]   \\
&\bullet \ar@{.>}[ru] & \bullet \ar[u] &  \bullet \ar[ul]   
}\]
Using arguments similar to those in the previous example, we can deduce that $L_K(E)$ is graded $*$-isomorphic to $\M_\omega(K)(0,1,2,\ldots)$ and $L_K(F)$ to $\M_\omega(K)(0, 1, 1, 1,\ldots).$ The Grothendieck groups of both of these algebras are isomorphic to $\Z$ and the two algebras are $*$-isomorphic as non-graded algebras. The {\em graded} Grothendieck groups of both algebras are isomorphic to $\Z[x,x^{-1}]$ and we claim that there is no contractive $\Z[x,x^{-1}]$-module isomorphism of the two graded Grothendieck groups. So, these algebras are not graded $*$-isomorphic. 

The order-unit $[1_{\M_n(K)(0,1,2,\ldots, n-1)}]$ of $K_0^{\gr}(\M_n(K)(0,1,2,\ldots, n-1))$ for $n\geq 1,$ corresponds to the element $1+x^{-1}+x^{-2}+\ldots+x^{-n+1}\in \Z[x,x^{-1}].$ By using Lemma \ref{unit_preservation}, one can see that the generating interval of the algebra $\M_\omega(K)(0,1,2,\ldots)$ consists of the elements of the form $\sum_{i\in I} x^{-i}$ where $I$ is a finite set of nonnegative integers. Similarly, the order-unit of $K_0^{\gr}(\M_n(K)(0,1,1,\ldots, 1))$ for $n\geq 1,$ corresponds to the element $1+(n-1)x^{-1}\in \Z[x,x^{-1}]$ and 
the generating interval of the algebra  $\M_\omega(K)(0,1,1,\ldots)$ consists of the elements of the form $\delta+kx^{-1}$ where $k$ is a nonnegative integer and $\delta=0$ or $\delta=1.$

Assume that there is a contractive isomorphism $f$ between the copy of  $\Z[x,x^{-1}]$ corresponding to $K^{\gr}_0(L_K(E))$ and the copy of  $\Z[x,x^{-1}]$ corresponding to $K^{\gr}_0(L_K(F))$. Then, for every $n>1,$ there is a nonnegative integer $m$ and $\delta\in\{0, 1\}$ such that $f(1+x^{-1}+\ldots+x^{-n+1})=\delta+mx^{-1}.$ If $k$ is a nonnegative integer and $\varepsilon\in\{0, 1\}$ such that $f(1)=\varepsilon+kx^{-1},$ then $(1+x^{-1}+\ldots+x^{-n+1})(\varepsilon+kx^{-1})=\delta+mx^{-1}.$ This leads to a contradiction by the assumption that $n>1.$ Thus, $f$ cannot be contractive. 
\end{example}

\subsection{Generalized Strong Isomorphism Conjecture for a class of Leavitt path algebras}

The finishing touch for the main result of this section, Theorem \ref{classification}, is the proposition below. 
\begin{proposition}
Consider the $\Z[x,x^{-1}]$-module structure on $\Z^n$ for any positive integer $n,$ to be given by  $x(a_1,\ldots, a_n)=(a_n, a_1, \ldots, a_{n-1}).$ 
If $I, J$ are countable sets, $n_i, m_j$ are positive integers for $i\in I$ and $j\in J,$ and $f$ is a contractive $\Z[x,x^{-1}]$-module isomorphism 
\[f:\bigoplus_{i\in I} \Z^{n_i} \to \bigoplus_{j\in J} \Z^{m_j},\]
then there is a bijection $\pi:I\to J$ such that $n_i=m_{\pi(i)}$ and such that $f$ maps $\Z^{n_i}$ onto $\Z^{m_{\pi(i)}}.$
\label{missing_part} 
\end{proposition}

\begin{proof}
Let $u^i_1,\ldots, u^i_{n_i}$ denote the standard $\Z$-basis of $\Z^{n_i},$ $i\in I,$ and let $v_1^j,\ldots, v_{m_j}^j$ denote the standard $\Z$-basis of $\Z^{m_j},$  $j\in J.$ Note that $xu^i_k=u^i_{k+1}$ for $k=1,\ldots, n_i-1$ and  $xu^i_{n_i}=u^i_1.$ 
To shorten the notation, for $k> n_i$ we define $u^i_k=u^i_{k'}$ if $k\equiv k'$ modulo $n_i$ for some $k'=1,\ldots, n_i.$ Thus $xu^i_k=u^i_{k+1}$ for every $k=1,\ldots, n_i.$ Similarly, $xv^j_l=v^j_{l+1}$ for $l=1,\ldots, m_j-1,$ $xv^j_{m_j}=v^j_1,$ and we use the analogous convention for $v_l^j$ if $l>m_j.$

For every $i\in I,$ there is a finite subset $J_i$ of $J$ and integers $a_{ji}^l,$ $j\in J_i,$ $l=1,
\ldots, m_j,$ such that $f(u^i_1)=\sum_{j\in J_i}\sum_{l=1}^{m_j} a_{ji}^l v_{l}^j.$
The condition that $f$ is order preserving implies that the integers $a_{ji}^l$ are nonnegative. 
The condition that $f$ is a $\Z[x,x^{-1}]$-homomorphism implies that $$f(u^i_k)=\sum_{j\in J_i}\sum_{l=1}^{m_j} a_{ji}^l v_{l+k-1}^j$$ for every $k=1,\ldots, n_i.$ 

Analogously, for every $j\in J,$ there is a finite subset $I_j$ of $I$ and nonnegative integers $b_{ij}^k,$ $i\in I_j,$ $k=1,\ldots, n_i,$ such that 
$$f^{-1}(v^j_l)=\sum_{i\in I_j}\sum_{k=1}^{n_i} b_{ij}^k u_{k+l-1}^i$$ for every $l=1,\ldots, m_j.$

Fix $i_0\in I.$ We intend to find a unique $j_0\in J$ with $m_{j_0}=n_{i_0}$ and $f(u_1^{i_0})=v_{l_0}^{j_0}$ for some $l_0=1, \ldots, m_{j_0}$ so that we can define $\pi(i_0)=j_0.$

The condition that $u^{i_0}_1=f^{-1}(f(u^{i_0}_1))$ implies that 
\[\sum_{j\in J_{i_0}}\sum_{l=1}^{m_j}  \sum_{i\in I_j}\sum_{k=1}^{n_{i}} a_{ji_0}^l b_{ij}^{k} u_{k+l-1}^{i}=u_1^{i_0}.\]
The condition that the coefficients are nonnegative implies that all the products  $a_{ji_0}^l b_{ij}^{k}$ are zero except for exactly one product $a_{j_0i_0}^{l_0} b_{i_0j_0}^{k_0}$ for some $j_0\in J_{i_0},$ some $l_0=1, \ldots, m_{j_0}$ and some $k_0=1, \ldots, n_i$ and, in this case, $a_{j_0i_0}^{l_0} b_{i_0j_0}^{k_0}=1$ and $k_0+l_0-1\equiv 1$ modulo $n_{i_0}.$ Hence $a_{j_0i_0}^{l_0}=1$ and $b_{i_0j_0}^{k_0}=1.$ 
The relation $b_{i_0j_0}^{k_0}=1$ implies that $i_0\in I_{j_0}$ and the relation $a_{j_0i_0}^{l_0}=1$ implies that $b_{ij_0}^{k}=0$ for all $i\in I_{j_0}$ and $k=1,\ldots, n_{i}$ unless $i=i_0$ and $k=k_0.$ Thus, $f^{-1}(v^{j_0}_1)=u^{i_0}_{k_0}$ and $f^{-1}(v^{j_0}_{l_0})=u^{i_0}_{k_0+l_0-1}=u^{i_0}_1$ because $k_0+l_0-1\equiv 1$ modulo $n_{i_0}.$ Hence $f(u^{i_0}_{k_0})=v^{j_0}_1$ and $f(u^{i_0}_1)=v^{j_0}_{l_0}.$ 

Since $v^{j_0}_{l_0}=x^{m_{j_0}}v^{j_0}_{l_0},$ we have that $f(u^{i_0}_1)=x^{m_{j_0}}f(u^{i_0}_1)=f(x^{m_{j_0}}u^{i_0}_1).$ This implies that $u^{i_0}_1=x^{m_{j_0}}u^{i_0}_1$ and so $n_{i_0}\leq m_{j_0}.$
Also, since $u^{i_0}_{k_0}=x^{n_{i_0}}u^{i_0}_{k_0},$ we have that $f^{-1}(v^{j_0}_1)=x^{n_{i_0}}f^{-1}(v^{j_0}_1)=f^{-1}(x^{n_{i_0}}v^{j_0}_1).$ This implies that $v^{j_0}_1=x^{n_{i_0}}v^{j_0}_1$ and so $m_{j_0}\leq n_{i_0}.$
Thus we have that $n_{i_0}=m_{j_0}.$ 

The element $j_0$ is unique for $i_0,$ so the correspondence $i_0\mapsto j_0$ defines an injective map $\pi: I\to J$ with the required properties. Starting with any $j\in J$ and repeating the same procedure produces a unique $i\in I_j$ such that $n_i=m_j$ and $f(u_1^i)=v^j_l$ for some $l=1,\ldots, m_j.$ This shows that $\pi$ is onto.    
\end{proof}

We prove the main theorem of this section now. 

\begin{theorem}\label{classification} 
If $E$ and $F$ are countable, row-finite, no-exit graphs in which each infinite path ends in a sink or a cycle and $K$ is a 2-proper, $*$-pythagorean field, then $L_K(E)$ and $L_K(F)$ are graded $*$-isomorphic if and only if there is a contractive $\Z[x,x^{-1}]$-module isomorphism $K^{\gr}_0(L_K(E))\cong K^{\gr}_0(L_K(F)).$
\end{theorem}
\begin{proof}
Let $E$ and $F$ be countable, row-finite, no-exit graphs in which each infinite path ends in a sink or a cycle and consider $K$ to be trivially graded by $\Z.$ By Proposition \ref{no-exit_graphs}, there is a graded ultramatricial $*$-algebra
\[R=\bigoplus_{i\in I} \M_{\kappa_i} (K)(\ol\alpha^i) \oplus \bigoplus_{i \in J} \M_{\mu_i} (K[x^{n_i},x^{-n_i}])(\ol\gamma^i)\] which is graded $*$-isomorphic to $L_K(E)$ and a graded ultramatricial $*$-algebra \[S=\bigoplus_{i\in I'} \M_{\kappa'_i} (K)((\ol\alpha')^i) \oplus \bigoplus_{i \in J'} \M_{\mu'_i} (K[x^{n'_i},x^{-n'_i}])((\ol\gamma')^i)\] which is graded $*$-isomorphic to $L_K(F).$ To prove the theorem, it is sufficient to show that if there is a 
contractive $\Z[x,x^{-1}]$-module isomorphism $K_0^{\gr}(R) \cong K_0^{\gr}(S),$ then there is a graded algebra $*$-isomorphism of rings $R$ and $S$. 

Since $K_0^{\gr}(R)\cong \Z[x, x^{-1}]^I\oplus \bigoplus_{i\in J}\Z^{n_i}$ and $K_0^{\gr}(S)\cong \Z[x, x^{-1}]^{I'}\oplus  \bigoplus_{i\in J'}\Z^{n'_i},$ let us assume that $f$ is a contractive   $\Z[x,x^{-1}]$-module isomorphism $\Z[x, x^{-1}]^I\oplus \bigoplus_{i\in J}\Z^{n_i} \to \Z[x, x^{-1}]^{I'}\oplus  \bigoplus_{i\in J'}\Z^{n'_i}.$ We claim that $f$ maps $ \bigoplus_{i\in J}\Z^{n_i}$ onto $ \bigoplus_{i\in J'}\Z^{n'_i}$ and $\Z[x, x^{-1}]^I$ onto $\Z[x, x^{-1}]^{I'}.$ The key point  of the argument is that $x$ acts differently on an element of $\Z[x, x^{-1}]^I$ than on an element of $ \bigoplus_{i\in J}\Z^{n_i}.$ In particular, if $a \in \bigoplus_{i\in J}\Z^{n_i},$ and $a(i)=(a^i_1, \ldots, a^i_{n_i})$ for some $i\in J,$ then $(xa)(i)=(a^i_{n_i}, a^i_1, \ldots,a^i_{n_i-1}).$ If $a\in \Z[x,x^{-1}]^I,$ then $(xa)(i)=xa(i)$ for $i\in I.$ 
  
To prove this claim, let  $a \in \bigoplus_{i\in J}\Z^{n_i}.$ If $f(a)=(b,c)$ where $b\in \Z[x, x^{-1}]^{I'}$ and $c\in  \bigoplus_{i\in J'}\Z^{n'_i},$ we show that $b=0.$ Let $J_0$ be a finite subset  $J$ which is the support of $a,$ i.e. such that $a(i)\neq 0$ if and only if $i\in J_0$ and let $m$ be the least common multiple of $n_i,$ $i\in J_0$ so that $x^ma=a.$ Thus, we have that 
\[(b,c)=f(a)=f(x^ma)=x^mf(a)=(x^mb, x^mc).\]
This implies that $x^mb=b.$ Since the action of $x^m$ on any element $\Z[x, x^{-1}]^{I'}$ is trivial just if that element is trivial, we obtain that $b=0.$ Thus, the image of $\bigoplus_{i\in J}\Z^{n_i}$ under $f$ is contained in $\bigoplus_{i\in J'}\Z^{n'_i}.$ Using the analogous argument for $f^{-1},$ we obtain that the image of $\bigoplus_{i\in J'}\Z^{n'_i}$ under $f^{-1}$ is contained in $\bigoplus_{i\in J}\Z^{n_i}.$

Thus, $f$ restricts to the contractive isomorphisms 
\[f_a:  \Z[x, x^{-1}]^{I}\to \Z[x, x^{-1}]^{I'}\;\;\mbox{ and }\;\;f_c:\bigoplus_{i\in J}\Z^{n_i}\to \bigoplus_{i\in J'}\Z^{n'_i}.\]
We can use Theorem \ref{ultramatricial} 
directly for the map $f_a$. Indeed, $K$ is 2-proper and $*$-pythagorean by the assumptions. Since $K$ is trivially graded by $\Z,$ it has enough unitaries. Theorem \ref{ultramatricial} 
implies that there is a graded $K$-algebra $*$-isomorphism  
\[
\phi_a: \bigoplus_{i\in I} \M_{\kappa_i} (K)(\ol\alpha^i)\to  \bigoplus_{i\in I'} \M_{\kappa'_i} (K)((\ol\alpha')^i).
\]

Let us consider $f_c:\bigoplus_{i\in J}\Z^{n_i}\to \bigoplus_{i\in J'}\Z^{n'_i}$ now. Proposition \ref{missing_part} implies that there is a bijection $\pi: J\to J'$ such that $n_i=n'_{\pi(i)}$ and that $f_c$ maps $\Z^{n_i}$ onto $\Z^{n'_{\pi(i)}}$ for any $i\in J.$ 
Let $\sim$ be the equivalence on $J$ given by $i\sim j$ if and only if $n_i=n_j$ and let $\ol J$ denote the set of cosets.  
The relation $\sim$ induces an equivalence on $J',$ which we also denote by $\sim,$ given by $i'=\pi(i)\sim j'=\pi(j)$ if and only if $i\sim j.$   
Thus, $f_c$ restricts to the map $f_c^{\ol j}:\bigoplus_{k\in\ol j}\Z^{n_{k}}\to \bigoplus_{k\in\ol{\pi(j)}}\Z^{n'_{\pi(k)}}$ for any $\ol j\in \ol J.$ 
The graded field $K[x^{n_j},x^{-n_j}]$ is such that $K[x^{n_j},x^{-n_j}]_0\cong K$ is 2-proper and $*$-pythagorean. A nonzero component of $K[x^{n_j},x^{-n_j}]$  is of the form $\{kx^{in_j}| k\in K\}$ for some $i\in \Z$ and it contains $x^{in_j}$ which is a unitary element. 
Hence the assumptions of Theorem \ref{ultramatricial}
are satisfied so there is a graded $K[x^{n_j},x^{-n_j}]$-algebra (hence also $K$-algebra) $*$-isomorphism 
\[\phi_c^{\ol j}: \bigoplus_{k \in \ol j} \M_{\mu_k} (K[x^{n_j},x^{-n_j}])(\ol\gamma^k)\to \bigoplus_{k \in \ol{\pi(j)}} \M_{\mu'_k} (K[x^{n'_{\pi(j)}},x^{-n'_{\pi(j)}}])((\ol\gamma')^k)
\] for any $\ol j\in \ol J.$  

Letting $\phi_c =\bigoplus_{\ol j\in \ol J}\phi_c^{\ol j}$ and $\phi=\phi_a\oplus \phi_c,$ we obtain a graded $K$-algebra $*$-isomorphism $\phi: R\to S.$ 
\end{proof}

\begin{corollary}\label{noninvolutive_classification}
If $E$ and $F$ are countable, row-finite, no-exit graphs in which each infinite path ends in a sink or a cycle and $K$ is any field, then $L_K(E)$ and $L_K(F)$ are graded isomorphic if and only if there is a contractive $\Z[x,x^{-1}]$-module isomorphism $K^{\gr}_0(L_K(E))\cong K^{\gr}_0(L_K(F)).$
\end{corollary}
\begin{proof}
The proof is analogous to the proof of Theorem \ref{classification} with Corollary \ref{noninvolutive_ultramatricial} used instead of Theorem \ref{ultramatricial}.  
\end{proof}

Theorem \ref{classification} implies the graded version of the Generalized Strong Isomorphism Conjecture for Leavitt path algebras of graphs and over fields which are as in the corollary below. 

\begin{corollary}\label{graded_GSIC} {\bf (Graded GSIC)}
Let $E$ and $F$ be countable, row-finite, no-exit graphs in which each infinite path ends in a sink or a cycle. Consider the following conditions. 
\begin{enumerate}[\upshape(1)]
\item $L_K(E) \cong_{\gr} L_K(F)$ as graded rings.
 
\item $L_K(E) \cong_{\gr} L_K(F)$ as graded algebras.

\item $L_K(E) \cong_{\gr} L_K(F)$ as graded $*$-algebras.
\end{enumerate}
The conditions (1) and (2) are equivalent for any field $K$. If $K$ is a 2-proper, $*$-pythagorean field, then all three conditions are equivalent.  

\end{corollary}
\begin{proof}
If $K$ is any field, the conditions (1) and (2) are equivalent by Corollary \ref{noninvolutive_classification}.

If $K$ is a 2-proper, $*$-pythagorean field, the condition (1) implies (3) since every graded ring homomorphism induces a contractive isomorphism of the $K^{\gr}_0$-groups so that Theorem \ref{classification} can be used. The implications (3) $\Rightarrow$ (2) $\Rightarrow$ (1) hold trivially. 
\end{proof}


\begin{thebibliography}{10}

\bibitem{Abrams_Aranda_Pino_1} G. Abrams, G. Aranda Pino, \emph{The Leavitt path algebra of a graph,} J. Algebra {\bf 293 (2)} (2005), 319--334.

\bibitem{AAPM} G. Abrams, G. Aranda Pino, F. Perera, and M. Siles Molina, \textit{Chain conditions for Leavitt path algebras,} Forum Math. \textbf{22} (2010) 95--114.

\bibitem{Abrams_Tomforde} G. Abrams, M. Tomforde, \emph{Isomorphism and Morita equivalence of graph algebras}, Trans. Amer. Math. Soc. {\bf 363 (7)} (2011),  3733--3767.

\bibitem{Ara_matrix_rings} P. Ara, \emph{Matrix rings over *-regular rings and pseudo rank functions}, Pacific J. Math., {\bf 129 (2)} (1987), 209--241. 

\bibitem{Ara_Moreno_Pardo} P. Ara, M.A. Moreno, E. Pardo, \emph{Nonstable $K$-theory for graph algebras,} Algebr. Represent. Theory {\bf 10 (2)} (2007), 157--178.

\bibitem{Gonzalo_Ranga_Lia} G. Aranda Pino, K.L. Rangaswamy, L. Va\v s, \emph{$^\ast$-regular Leavitt path algebra of arbitrary graphs}, Acta Math. Sci. Ser. B Engl. Ed. \textbf{28} (2012), 957--968.

\bibitem{Gonzalo_paper} G. Aranda Pino, L. Va\v s, \emph{Noetherian Leavitt path algebras and their regular algebras}, Mediterr. J. Math., {\bf 10 (4)}
(2013), 1633 -- 1656.

\bibitem{Baranov} A. A. Baranov, \emph{Classification of the direct limits of involution simple associative algebras and the corresponding dimension groups}, J. Algebra {\bf 381} (2013), 73--95.

\bibitem{Berberian} S.K. Berberian, \emph{Baer $*$-rings}, Die Grundlehren der mathematischen Wissenschaften 195, Springer-Verlag, Berlin-Heidelberg-New York, 1972.

\bibitem{Berberian_web} S.K. Berberian, \emph{Baer rings and Baer $*$-rings}, 1988, preprint at \url{https://www.ma.utexas.edu/mp_arc/c/03/03-181.pdf}.

\bibitem{Davidson} K.R. Davidson, \emph{$C^*$-algebras by Example}, Field Institute Monographs {\bf 6}, American Mathematical Society, 1996.

\bibitem{Elliott} G.A. Elliott, \emph{On the classification of inductive limits of sequences of semisimple finite-dimensional algebras}, J. Algebra {\bf 38 (1)} (1976), 29--44. 

\bibitem{Gabe_et_al} J. Gabe, E. Ruiz, M. Tomforde, T. Whalen, {\em $K$-theory for Leavitt path algebras: Computation and classification}, J. Algebra {\bf 433 (1)} (2015), 35--72.

\bibitem{Goodearl_book} K.R. Goodearl, \emph{von Neumann regular rings}, 2nd ed., Krieger Publishing Co., Malabar, FL, 1991.

\bibitem{Goodearl_Handelman}  K. R. Goodearl, D. E. Handelman,  \emph{Classification of ring and $C^*$-algebra direct limits of finite-dimensional semisimple real algebras}, Memoirs of the American Mathematical Society {\bf 372}, 1987.

\bibitem{Handelman}  D. E. Handelman, \emph{Coordinatization applied to finite Baer $*$-rings}, Trans. Amer. Math. Soc, {\bf 235} (1978), 1--34.

\bibitem{Roozbeh_Annalen} R. Hazrat, \emph{The graded Grothendieck group and classification of Leavitt path algebras,} Math. Annalen {\bf 355 (1)} (2013), 273--325.

\bibitem{Roozbeh_Israeli} R. Hazrat, \emph{The graded structure of Leavitt path algebras}, Israel J. Math, {\bf 195} (2013), 833--895.

\bibitem{Roozbeh_graded_ring_notes} R. Hazrat, \emph{Graded rings and graded Grothendieck groups}, London Math. Soc. Lecture Note Ser. 435, Cambridge Univ. Press, 2016.

\bibitem{Zak_Lia} Z. Mesyan, L. Va\v s, \emph{Traces on semigroup rings and Leavitt path algebras}, Glasg. Math. J., {\bf 58 (1)} (2016), 97--118. 

\bibitem{Prijatelj_Vidav} N. Prijatelj, I. Vidav, \emph{On special $*$-regular rings}, Michigan Math. J., {\bf 18} (1971), 213--221.

\bibitem{Quillen} D. Quillen, \emph{Higher algebraic $K$-theory. I}. Algebraic $K$-theory, I: Higher $K$-theories (Proc. Conf., Battelle Memorial Inst., Seattle, Wash., 1972), pp. 85--147. Lecture Notes in Math., Vol. 341, Springer, Berlin 1973.

\bibitem{Ruiz_Tomforde} E. Ruiz, M. Tomforde, {\em Classification of unital simple Leavitt path algebras of infinite graphs}, J. Algebra. {\bf 384} (2013), 45--83.

\end{thebibliography}
\end{document}